\documentclass[10 pt]{amsart}

\usepackage{setspace}
\usepackage{graphicx}
\usepackage{amsthm}

\usepackage{a4wide}
\usepackage{array, xcolor}
\usepackage{amsfonts}
\usepackage{amsmath}
\usepackage{bm}

\usepackage{enumitem}
\usepackage{bbm}
\usepackage{mathtools}
\usepackage{amssymb}
\usepackage{graphicx}
\usepackage{float}
\usepackage{accents}
\usepackage{verbatim}
\newtheorem{proposition}{Proposition}
\newtheorem{lemma}[proposition]{Lemma}
\newtheorem{corollary}[proposition]{Corollary}
\newtheorem{theorem}[proposition]{Theorem}

\usepackage{tikz}
\usepackage{subfig}
\usepackage{xcolor}
\usetikzlibrary{matrix}
\usetikzlibrary{arrows,calc}
\usepackage{fancyhdr}
\usepackage[hidelinks]{hyperref}
\usepackage{slashed}

\definecolor{bluep}{rgb}{0.2,0.2,0.6}
\definecolor{cadmium}{rgb}{0.0,0.42,0.24}
\hypersetup{%
		pdfborder = {0 0 0},
		colorlinks,
		citecolor=bluep,
		filecolor=blue,
		linkcolor=black,
		urlcolor=black
		}

\theoremstyle{definition}
\newtheorem{definition}[proposition]{Definition}
\newtheorem{remark}[proposition]{Remark}
\newtheorem{hypothesis}[proposition]{Hypothesis}
\newtheorem*{remark*}{Remark}
\newtheorem*{example*}{Example}
\newtheorem{example}[proposition]{Example}

\newcommand{\Int}{\operatorname{Int}}
\newcommand{\fib}{\operatorname{fib}}

\newcommand{\codim}{\operatorname{codim}}
\newcommand{\Aa}{\mathcal{A}}

\newcommand{\Bu}{\operatorname{B_1}}
\newcommand{\Bd}{\operatorname{B_2}}
\newcommand{\Bt}{\operatorname{B_3}}

\newcommand{\Pu}{\operatorname{P_1}}
\newcommand{\Pd}{\operatorname{P_2}}
\newcommand{\Pt}{\operatorname{P_3}}

\newcommand{\Cu}{\operatorname{C_1}}
\newcommand{\Cd}{\operatorname{C_2}}
\newcommand{\Ct}{\operatorname{C_3}}

\newcommand{\Tu}{\operatorname{T_1}}
\newcommand{\Td}{\operatorname{T_2}}
\newcommand{\Tt}{\operatorname{T_3}}

\newcommand{\rank}{\operatorname{rank}}
\newcommand{\modd}{\operatorname{mod}}
\newcommand{\Rr}{\operatorname{R}}
\newcommand{\LC}{\operatorname{LC}}
\newcommand{\Diff}{\operatorname{Diff}}
\newcommand{\cl}{\operatorname{cl}}
\newcommand{\Cl}{\operatorname{Cl}}
\newcommand{\PSOn}{\operatorname{P_{\SO(n)}}}
\newcommand{\PSO}{\operatorname{P_{\SO(2)}}}
\newcommand{\scal}{\operatorname{scal}}
\newcommand{\SO}{\operatorname{SO}}
\newcommand{\Imag}{\operatorname{Im}}
\newcommand{\Spin}{\operatorname{Spin}}
\newcommand{\Psp}{\operatorname{P_{Spin(n)}}}
\newcommand{\Pspd}{\operatorname{P_{Spin(2)}}}
\newcommand{\Te}{\mathcal{T}}
\newcommand{\NN}{\mathcal{N}}
\newcommand{\Ss}{\mathcal{S}}
\newcommand{\II}{\mathcal{I}}
\newcommand{\sus}{\operatorname{sus}}
\newcommand{\sbm}{\operatorname{sb}}
\newcommand{\cc}{\operatorname{c}}
\newcommand{\cp}{\operatorname{cp}}
\newcommand{\fffb}{\operatorname{fff_b}}
\newcommand{\fffc}{\operatorname{fff_c}}
\newcommand{\Xs}{\operatorname{X_{cp}^1}}
\newcommand{\Xd}{\operatorname{X_{cp}^2}}
\newcommand{\Xt}{\operatorname{X_{cp}^3}}
\newcommand{\Bm}{\operatorname{Br_2}}
\newcommand{\Ba}{\operatorname{Br_1}}

\newcommand{\Ddd}{\operatorname{D}}
\newcommand{\Dd}{\slashed{\Ddd}}

\newcommand{\WF}{\operatorname{WF}}
\newcommand{\ff}{\operatorname{ff}}
\newcommand{\ffb}{\operatorname{ff_b}}
\newcommand{\tb}{\operatorname{tb}}
\newcommand{\ffc}{\operatorname{ff_c}}

\newcommand{\tf}{\operatorname{tf}}

\newcommand{\lp}{\left(}
\newcommand{\rp}{\right)}
\newcommand{\EE}{\mathcal{E}}
\newcommand{\PP}{\operatorname{P}}
\newcommand{\Diag}{\operatorname{Diag}}

\newcommand{\Ker}{\operatorname{Ker}}

\newcommand{\dvol}{\operatorname{dvol}}
\newcommand{\Spec}{\operatorname{Spec}}
\newcommand{\End}{\operatorname{End}}
\newcommand{\supp}{\operatorname{supp}}

\newcommand{\id}{\operatorname{id}}

\newcommand{\Tr}{\operatorname{Tr}}

\newcommand{\dd}{\operatorname{d}}
\newcommand{\RR}{\mathbb{R}}
\newcommand{\CC}{\mathbb{C}}

\title[Degenerating Dirac spectrum]{On the Dirac spectrum on degenerating Riemannian surfaces}
\author{Cipriana Anghel}
\email{cipriana.anghel-stan@mathematik.uni-goettingen.de; cianghel@imar.ro}

\begin{document}

\address{
\begin{flushleft}
Mathematisches Institut, Universität Göttingen, Bunsenstr. 3-5, 37073, Göttingen, Germany
\end{flushleft}  
\begin{flushleft}
Institutul de Matematică al Academiei Române, Calea Griviței 21, 010702, Bucharest, Romania
\end{flushleft}
}

\date{\today}

\begin{abstract}
We study the behavior of the spectrum of the Dirac operator on degenerating families of compact Riemannian surfaces, when the length $t$ of a simple closed geodesic shrinks to zero, under the hypothesis that the spin structure along the pinched geodesic is non-trivial. The difficulty of the problem stems from the non-compactness of the limit surface, which has finite area and two cusps. The main idea in this investigation is to construct an adapted pseudodifferential calculus, in the spirit of the celebrated \emph{b}-algebra of Melrose, which includes both the family of Dirac operators on the family of compact surfaces and the Dirac operator on the limit non-compact surface, together with their resolvents. We obtain smoothness of the spectral projectors, and $t^2 \log t$ regularity for the cusp-surgery trace of the relative resolvent in the degeneracy process as $t \searrow 0$.
\end{abstract}

\maketitle

\section{Introduction} \label{introducere}
\subsection*{Context and Motivation}\label{degen}
Let $X$ be a smooth compact oriented surface of genus $g \geq 2$ and denote by $ \mathcal M_{-1}(X)$ the set of hyperbolic metrics on $X$. By the classical Poincaré-Koebe uniformization theorem of Riemann surfaces, $ \mathcal M_{-1}(X)$ is in one-to-one correspondence with the set of complex structures on $X$. The Teichmüller space $\mathcal T_g$ is defined by factoring the set $ \mathcal M_{-1}(X)$ by the connected component of $\id_M$ in the group of diffeomorphisms of $X$. It is a non-complete Kähler manifold of dimension $3g-3$. 

Let $\gamma \subset X$ be a simple closed curve. Consider a smooth family of metrics $(g_t)_{1 \geq t \geq 0}$ on $X \setminus \gamma$ which near $\gamma$ look like hyperbolic cylinders:
\begin{equation}\label{metr}
\begin{aligned}
g_t = \frac{dx^2}{x^2 +t^2} + \lp x^2 + t^2 \rp d y^2,  && && && && && (x,y) \in  \lp  - \frac{t}{\sinh \tfrac{t}{2}}, \frac{t}{\sinh \tfrac{t}{2}} \rp \times S^1.
\end{aligned}
\end{equation}

For $t>0$, $g_t$ extends to a smooth metric on $X$, hyperbolic on a uniform neighborhood of the pinched curve $\gamma$ which becomes a simple closed geodesic, while $(X \setminus \gamma,g_0)$ is a complete Riemannian non-compact surface with two hyperbolic cusps. We call such a family of metrics a \emph{pinching process} along the geodesic $\gamma$. We stress that the metric $g_t$ does not need to be hyperbolic everywhere, but only locally, near $\gamma$. Our results remain valid when $\gamma$ is disconnected, i.e., when we pinch simultaneously several disjoint simple curves, but for simplicity we will describe below the case where we only pinch one geodesic.

If we choose the family of metrics $(g_t)_{t \geq 0}$ in the set $\mathcal M_{-1}(X)$, pinching up to $3g-3$ disjoint simple closed geodesics gives a path in the Teichmüller space escaping towards infinity. Such degeneration phenomena were studied by Ji \cite{Ji}, Bär \cite{Bar}, Schulze \cite{Schul}, and very recently by Stan \cite{rares}. Furthermore, using an adapted Selberg trace formula, Monk and Stan \cite{rarlaura} studied the distribution of the eigenvalues of the Dirac operator on typical (in Weil–Petersson sense) surfaces of large genus $g$ with $o(\sqrt{g})$ cusps. We could also mention an ongoing study of Benabida \cite{benabida} on the spectrum of the Laplacian undergoing a certain type of degeneracy in higher dimensions.

It is well-known that the spectrum of a geometric elliptic differential operator like the Laplacian or the Dirac operator varies continuously under smooth perturbations of the metric. Our aim is:
\begin{quote}
Study the continuity of the spectrum of the spin Dirac operator during a pinching process. 
\end{quote} 
The difficulty of the problem arises from the non-compactness of the limit surface. 

\subsection*{Key analytic point} 
There exist examples of geometric Laplacians with purely discrete spectrum on certain non-compact manifolds, like the Laplacian acting on forms \cite{laplforme}, and the magnetic Laplacian \cite{laplmagnetic}. In general, the scalar Laplacian may even have purely continuous spectrum on non-compact manifolds, e.g. on euclidean spaces. Bär \cite{Bar} proved that under some \emph{invertibility condition} on each cusp, the spectrum of the Dirac operator on complete hyperbolic surfaces of finite area is purely discrete. More precisely, if we glue a circle at the ``end" of each cusp, we require the spin structure to be non-trivial along this circle, which is equivalent to the invertibility of the Dirac operator on the circle defined with respect to the induced spin structure. Moroianu \cite{moroweyl} generalised Bär's result on manifolds of arbitrary dimension proving a Weyl's law for a wide class of open manifolds with cusp-like ends.

\subsection*{Idea of the solution} The main technique for investigating the continuity of the spectrum of the Dirac operator during a pinching process is to construct a tailored pseudodifferential calculus (similar to Melrose's \emph{b}-calculus \cite{melrose}) which includes both the family of Dirac operators $\Dd_t$ on the family of compact surfaces $(X,g_t)_{t > 0}$, and the Dirac operator $\Dd_0$ on the limit non-compact surface $(X \setminus \gamma,g_0)$. This adapted pseudodifferential calculus is closely related to the cusp calculus \cite{melnistor} (a particular case of the $\varphi$-calculus \cite{melmaz98}). More precisely, it is the cusp-calculus with a time parameter, we denote it by $\Psi^{*,*,*}_{\cp}(X)$, and we construct it for an $n$-dimension manifold, even though we will apply it for our particular case of a surface. An alternate strategy that we do not pursue here is the groupoid approach to constructing pseudodifferential algebras as in \cite{nistorammannlauter}.

McDonald \cite{mcdonald} developed the surgery \emph{b}-calculus to study families of Laplace-Beltrami operators on a compact manifold with a family of metrics which degenerate to a conic one. Building upon his work, a slightly modified version of the calculus was also utilized in \cite{melmaz}. Furthermore, Mazzeo-Melrose \cite{melmaz98} introduced the fibered cusp calculus, which provided a pseudodifferential generalization of the Atiyah-Patodi-Singer index theorem. 

Recently, Albin, Rochon, Sher \cite{ars} introduced the fibered cusp calculus with a parameter in order to study the spectrum of the Hodge Laplacian having coefficients in a flat bundle on a compact manifold which degenerates to a manifold with fibered cusps. Our cusp-surgery calculus is a particular case of their $\varphi$-surgery calculus, since they treat the more complicated case when the boundary fibrates over another closed manifold. We believe that it is worth including in this work all the details of the construction, since the focus of our investigation regards fully-elliptic differential operators, leading to different and more straightforward proofs. Very recently, Melrose \cite{generalisedproducts} announced a general method for constructing calculi with parameters using generalised products. 

We will describe the cusp-surgery pseudodifferential operators in $\Psi^{*,*,*}_{\cp}(X)$ as distributions on a certain blown-up space $\Xd$, conormal to the closure of the $(0, \infty) \times \Diag$, where $\Diag$ is the diagonal inside $X \times X$ (see Fig. \ref{defop}). Furthermore, we will impose certain polyhomogeneous behavior of the distributions towards the boundary faces of $\Xd$. Since we are interested in studying the limit case $\{ t=0 \}$ in the pinching process, we need our cusp-surgery pseudodifferential operators to be well-defined only for a small interval time $[0, t_0]$. 

A challenging and crucial result in constructing this calculus is the so-called Composition Theorem, which establishes that the calculus is closed under composition. As customary, this theorem is proved through the use of a triple space, along with multiplication of conormal distributions, and the Pull-back and Push-forward Theorems for conormal distributions (see for instance \cite{mel92}, \cite{grieser}), however we stress that the geometric structure of the triple space and of the companion \emph{b}-fibrations are by no means trivial.

To each cusp-surgery pseudodifferential operator $A \in \Psi^{*,*,*}_{\cp}(X)$, we will associate three leading symbols: $\sigma_{\cp}(A)$, $\NN(A)$, and $\Te(A)$. More precisely, the cusp-parameter symbol $\sigma_{\cp}(A)$ is the leading term in the principal symbol of the conormal distribution $k_A$. The normal operator $\NN(A)$ is a normalization of the restriction of $A$ to the cusp front face $\ffc$. Finally, the temporal operator $\Te (A)$ is the normalized restriction of $A$ to $\tb$, the lift of the temporal boundary $\{ t=0 \} \times X \times X \subset [0, \infty) \times X \times X$ to the double space $\Xd$ (see Fig. \ref{doublesp}).

If an operator $A \in  \Psi^{*,*,*}_{\cp}(X)$ has all the three symbols invertible, one can construct a parametrix modulo residual operators, i.e., operators belonging to $\Psi^{-\infty,-\infty,-\infty}_{\cp}(X)$.

\subsection*{Main results} 
Let $\Pspd X \longrightarrow X$ be a spin structure non-trivial along $\gamma$ (see Hypothesis \ref{hypot}). We manage to merge the correspondent family of Dirac operators $\lp \Dd_t \rp_{t>0}$ on the Riemannian surfaces $(X,g_t)$, together with the Dirac operator $\Dd_0$ on the limit non-compact surface $(X \setminus \gamma ,g_0)$, to obtain a cusp-surgery differential operator 
\[ \Dd \in \Psi^{1,1,0}_{\cp}(X), \]
in the sense that the restriction of $\Dd$ at time $t$ is $\Dd_t$, for every $t \geq 0$.

A crucial result is to prove that the normal operator $\NN(\Dd)$ of $\Dd$ is invertible. Here we rely on the hypothesis of non-triviality of the spin structure along the geodesic $\gamma$. We first prove that if $\lambda$ is not an eigenvalue of $\Dd_0$, then the resolvent family $(\Dd_t-\lambda)^{-1}$ for small time $t$ belongs to the calculus. 

\begin{theorem}\label{resolvintro}
Let $X$ be a compact oriented surface and let $\gamma \subset X$ be a simple closed curve, possibly with several connected components. Consider a smooth family of metrics $(g_t)_{t \geq 0}$ on $X \setminus \gamma$ defining a pinching process along $\gamma$ as in \eqref{metr}. Fix a spin structure $\Pspd X \longrightarrow X$ non-trivial along $\gamma$ in the sense of Hypothesis \ref{hypot}. Let $(\Dd_t)_{t > 0}$ be the family of Dirac operators on $X$ corresponding to $(g_t)_{t > 0}$, and let $\Dd_0$ be the Dirac operator on $X \setminus \gamma$ corresponding to the metric $g_0$. 

If $\lambda \in \RR \setminus \Spec \Dd_0$, then there exists $t_0 (\lambda) >0$ such that the operator $(\Dd_t - \lambda)$ is invertible for every $ 0 \leq t \leq t_0 (\lambda)$, and the resolvents $(\Dd_t-\lambda)^{-1}_{0 \leq t \leq t_0 (\lambda)}$ are the restriction at time $t$ of a cusp-surgery operator
\[(\Dd-\lambda)^{-1} \in \Psi^{-1,-1,0}_{\cp}(X). \]
\end{theorem}

Our second main result is the smoothness of the spectral projectors in a pinching process.
\begin{theorem}\label{convprspectrintro1}
Under the hypotheses of Theorem \ref{resolvintro}, let $\lambda_0$ be an eigenvalue for the limit operator $\Dd_0$. Consider $\epsilon>0$ such that 
\[ [\lambda_0-\epsilon, \lambda_0 +\epsilon] \cap \Spec \Dd_0 = \{ \lambda_0 \}. \]
Then there exists $t_1(\lambda_0)>0$ such that the family of spectral projectors $\PP_{ [\lambda_0-\epsilon, \lambda_0 +\epsilon]}$ for $t \in [0, t_1(\lambda)]$  belongs to $\Psi^{-\infty,-\infty,0}_{\cp}(X)$. More precisely, its Schwartz kernel is smooth on $[0, t_1(\lambda_0)) \times X \times X$ and it vanishes rapidly at $\{ t=0 \} \times \gamma \times \gamma$.
\end{theorem}

When the metrics $(g_t)_{t \geq 0}$ are hyperbolic (on the entire surface $X$), Pfäffle \cite[Theorem 1.2]{pfaffle} proved that the Dirac spectrum of the limit surface is approximated by the spectra of the compact surfaces $(X,g_t)$. As a corollary to Theorem \ref{convprspectrintro1}, we obtain the convergence of the spectrum of the Dirac operator in the pinching process along the geodesic $\gamma$, generalizing Pfäffle's result for metrics that need to be hyperbolic \emph{only} near the pinched geodesic $\gamma$.

\begin{corollary}\label{convspectruluii}
Let $\epsilon>0$, $a,b \in \mathbb R \setminus \Spec \Dd_0$. Then there exists $t_2 (\epsilon,a,b)>0$ for which the Dirac operators $(\Dd_t)_{0 \leq t \leq t_2}$ have the same number $k$ of eigenvalues with multiplicities in the interval $(a,b)$. We denote by $\lp \lambda_j(t)\rp_{1 \leq j \leq k}$ the eigenvalues of $\Dd_t$ in the interval $(a,b)$ arranged in increasing order. Then for any $\lambda_j(0) \in \Spec \Dd_0 \cap (a,b)$ and for any $t \leq t_2$ we have
\[  |\lambda_j(t) - \lambda_j(0)| < \epsilon.  \]
\end{corollary}

If $A \in \Psi^{m, \alpha, \beta}_{\cp}(X)$ is a trace-class cusp-surgery operator, meaning that the orders $m, \alpha, \beta$ satisfy the following inequalities:
\begin{align*}
 m<  -2, && \alpha < -1, && \beta \leq 0,
 \end{align*}
its cusp-trace is a function ${}^{\cp} \! \Tr (A) : [0, \infty) \longrightarrow \mathbb C$ which associates to each time $t$ the integral over the $t$-time slice in the diagonal ``plane" $\Delta \subset \Xd$ (see Fig. \ref{defop}). In fact, for $t>0$, the cusp trace associates to $A$ exactly the $L^2$-trace of the operator $A$ at time $t$ acting on $(X,g_t)$.

\begin{proposition}\label{traceclassintro}
Let $A \in \Psi^{m, \alpha, \beta}_{\cp}(X)$ be a trace-class cusp-surgery pseudodifferential operator.
\begin{itemize}
\item[$i)$] If $\alpha - \beta \notin \mathbb Z$, then 
\[ {}^{\cp}{\!} \Tr A \in t^{- \alpha} \mathcal C^{\infty}  [0, \infty) + t^{- \beta} \mathcal C^{\infty}  [0, \infty) . \]
\item[$ii)$] If $\alpha - \beta \in \mathbb Z$, then 
\[  {}^{\cp}{\!} \Tr A \in t^{ \min (-\alpha, - \beta)} \mathcal C^{\infty}  [0, \infty) + t^{ \max(-\alpha,-\beta) } \log t \cdot \mathcal C^{\infty}  [0, \infty) . \]
\end{itemize}
\end{proposition}

 Notice that for $t>0$, the cusp-surgery trace exists whenever $\alpha>-1$, and is clearly a $\mathcal C^{\infty}$ function. The relevance of the result above is that it describes the behavior of the cusp surgery trace towards $\{ t=0 \}$. In particular, it is of class $\mathcal C^{1}$. As a corollary, we study the cusp-surgery trace of the resolvent of the Dirac operator.

\begin{corollary}\label{puterirez}
Consider an integer $k \geq 3$, and let $\lambda \in \mathbb R \setminus \Spec \Dd_0$. Denote by  
\[ \Rr(\lambda)=(\Dd-\lambda)^{-1} \in \Psi_{\cp}^{-1,-1,0} \lp X \rp\]
the family of resolvents of the Dirac operator as in Theorem \ref{resolvintro}. Then the $k^{\text{th}}$ power of the resolvent is trace-class as in Proposition \ref{traceclassintro}, and its cusp-surgery trace function ${}^{\cp}{\!}\Tr \lp \Rr (\lambda)^k \rp$ is of Hölder class $\mathcal C^{k-1, \alpha}$ in $t \geq 0$, for any $\alpha \in (0,1)$.
\end{corollary}

Remark that ${}^{\cp}{\!} \Tr \lp \Rr(\lambda)^k \rp$ is smooth for $t>0$ so, as above, the content of this corollary lies in the behavior of the cusp-surgery trace as a function of $t$ towards $t \to 0$.

\begin{corollary}\label{improverares}
Let $\lambda, \lambda_0 \in \mathbb R$ such that the cusp differential operators $\Dd_0^2-\lambda$ and $\Dd_0^2-\lambda_0$ are invertible. Denote the resolvents of the squared Dirac operator by 
\begin{align*}
\widetilde{\Rr}(\lambda):= \lp \Dd^2-\lambda\rp^{-1} \in \Psi_{\cp}^{-2,-2,0}(X), && \widetilde{\Rr}(\lambda_0):=  \lp \Dd^2-\lambda_0 \rp^{-1} \in \Psi_{\cp}^{-2,-2,0}(X).
\end{align*}
Then the relative resolvent $\widetilde{\Rr} (\lambda) - \widetilde{\Rr} (\lambda_0) \in \Psi_{\cp}^{-3,-2,0}(X)$ is trace-class and its cusp-surgery trace behaves, as a function of $t$ near $t=0$, as follows:
\[ {}^{\cp}{\!} \Tr \lp \widetilde{\Rr} (\lambda) - \widetilde{\Rr} (\lambda_0) \rp \in \mathcal C^{\infty}  [0, \infty) + t^{2 } \log t \ \mathcal C^{\infty}  [0, \infty) . \]
\end{corollary}

In a future work, we intend to apply these corollaries to improve the result of Stan \cite{rares} for the asymptotic behavior of the Dirac Selberg zeta function on degenerating hyperbolic surfaces.

\subsection*{Acknowledgements} 
I am indebted to my former PhD advisor, Sergiu Moroianu, for suggesting to me this project as part of my PhD thesis, for countless enlightening discussions which led to the development of this work, and for introducing me to the world of pseudodifferential calculi. I would like to thank Daniel Grieser for giving me access to his excellent online course on Singular Analysis and for several discussions on the subject. I am grateful to Frédéric Rochon for sharing insights on his paper \cite{ars} and also on the paper of McDonald \cite{mcdonald} during a conference at CIRM. I am thankful to Rareș Stan, Léo Bénard, Thomas Schick, Victor Nistor, Oussama Benabida and Álvaro Sánchez Hernández for interesting questions and discussions regarding this project.

\subsection*{Funding}
This work was partly supported by the PNRR-III-C9-2023-I8 grant CF 149/31.07.2023 {\em Conformal Aspects of Geometry and Dynamics}.

\section{Manifolds with corners}\label{sem1}
We review below the elements from Melrose's theory of manifolds with corners \cite{melrose}, \cite{melmaz} needed for the construction of the cusp-surgery calculus. In the presentation, we will mainly follow \cite{cursgrieser}, \cite{grieser}, and \cite{mel92}. The local model for such a manifold is given by $\RR^n_k:=[0,\infty)^k \times \mathbb R^{n-k}$. We say that the non-negative variables near the corner $\{ 0 \} \times \mathbb R^{n-k}$ are of \emph{$x$-type} and the real variables are of \emph{$y$-type}.

Using the topology induced by the metric from $\RR^n$, we say that $U \subset \mathbb R^n_k$  is an \emph{open set} if for any $p \in U$ there exists an $\epsilon>0$ such that $\{ q \in \mathbb R^n_k: \ \vert p-q \vert < \epsilon  \} \subset U$. This definition is equivalent to the existence of an open set $\tilde{U} \subset \mathbb R^n$ such that $U=\tilde{U} \cap \mathbb R^n_k$, where in this description we used the induced topology from $\RR^n$.

Let $U \subset \mathbb R^n_k$ be an open set. A function $f: U \longrightarrow \mathbb C$ is \emph{smooth} if $f_{\vert_{\accentset{\circ}{U}}}$ is smooth in the usual sense and all the derivatives of $f$ extend continuously to $U$. Using Borel's Lemma, it is easy to prove that $f \in \mathcal C^{\infty}(U)$ if and only if it is the restriction to $U$ of a smooth function defined on an open set in $\mathbb R^n$. Furthermore, a function $f=(f_1,...,f_m): U \longrightarrow \mathbb R^m$ is called \emph{smooth} if each component $f_j$ is smooth, for all $j=\overline{1,m}$. Finally, a function $f: U_1 \longrightarrow U_2$ between two open sets $U_1 \subset \mathbb R^n_{k_1}$, $U_2 \subset \mathbb R^n_{k_2}$ is said to be a \emph{diffeomorphism} if $f$ is bijective, smooth and the inverse $f^{-1}$ is also smooth. 

Consider a point $p$ in an open set $U \subset \mathbb R^n_k$ and let $T_p^+U$  be the inward tangent cone, i.e. the set of all vectors $v \in \mathbb R^n$ such that there exist $\epsilon>0$ and a smooth integral curve $\gamma: [0,\epsilon) \longrightarrow U$ such that $\dot{\gamma}(0)=v$. Then we define the \emph{codimension of the point $p$} denoted by $\codim p$ to be the codimension  in $\RR^n_k$ of the biggest vector space inside $T_pU^+$. Remark that the codimension of a point is invariant to diffeomorphisms. 

\begin{definition}
A Hausdorff, second countable, topological space $X$ is called a \emph{weak manifold with corners} of dimension $n$ if there exist an open covering $(U_j)_{j \in J}$ and a family of homeomorphisms $\varphi_j: U_j \longrightarrow \varphi_j(U_j) \subset \mathbb R^n_n$ such that $\varphi_i \circ \varphi_j^{-1}$ is a diffeomorphism for all $i,j \in J$.
\end{definition}

Notice that any point in $\RR^n_k$, $0 \leq k \leq n$, has a neighborhood which is diffeomorphic to a neighbourhood of a point in $\RR^n_n$. In the above definition, a chart centered at a point $p \in X$ actually takes values into $\RR^n_{k(p)}$, where $k(p)$ is a function on $X$ with values in $\{ 0,1,...,n \}$. If $X$ is a weak manifold with corners, the codimension of every point $p \in X$ is well-defined as 
$ \codim p := \codim \varphi_j(p), $
for any $j \in J $ such that $p \in U_j$. 

A \emph{face of codimension k} of $X$ is the closure of a connected component of the set of points $p \in X$ of codimension $k$. Furthermore, a \emph{boundary hypersurface} is a face of codimension $1$ and we will denote the set of all boundary hypersurfaces by $\mathcal F_1(X)$.

\begin{definition}\label{submfd}
Let $X$ be a weak manifold with corners. We say that a connected set $S \subset X$ is a \emph{submanifold of dimension $n'$} if for any point $p \in S$ there exist a chart $(\varphi,U)$ centered at $p$ and an open set $U'$ in the neighborhood of $0$ in $\RR^n$ such that
\[ \varphi_{\vert_S} : S \cap U \longrightarrow \lp \RR^{n'}_{k'} \times \{ 0 \}^{n-n'} \rp \cap U' \subset \mathbb R^n_k.  \] 
\end{definition}

\begin{definition}
A space $X$ is called a \emph{manifold with corners} if it is a weak manifold with corners and all its boundary hypersurfaces are submanifolds.
\end{definition}

From now on, let $X$ be a manifold with corners. A function $f: X \longrightarrow \mathbb C$ is \emph{smooth} if $f \circ \varphi_j^{-1}$ is smooth for any $j \in J$, and more general, a function $f: X \longrightarrow Y$ between manifolds with corners is called \emph{smooth} if $\varphi_i \circ f \circ \psi_j^{-1}$ is smooth for any $i,j \in J$. A central notion in Melrose's theory of blow-ups of manifolds with corners is the notion of product-submanifold.

\begin{definition}
We say that $Y \subset X$ is a \emph{$p$-submanifold} if for any point $p \in Y$ there exist local coordinates in a neighborhood $U$ centered in $p$ 
\[ \varphi= \lp x_1,...,x_k,y_1,...,y_{n-k} \rp : U \longrightarrow U' \subset \mathbb R^n_k,  \] 
in which $\varphi (Y \cap U)$ is given by the vanishing of some $x_i$'s and/or some $y_j$'s.
\end{definition}
Remark that any $p$-submanifold is a submanifold in the sense of Definition \ref{submfd}, while the converse is not always true. For example, $\RR^2_2=[0,\infty) \times [0, \infty) \times \{ 0 \}$ is a submanifold of $\RR^3_1= \RR \times \RR \times [0,\infty)$, but it is \emph{not} a $p$-submanifold.
Furthermore, $\mathbb R^2_1 = [0, \infty) \times \mathbb R \times \{ 0\}$ is a p-submanifold of $ \mathbb R^3_2=[0,\infty) \times \RR \times [0,\infty)$.

\begin{definition}\label{defbdf}
Let $H$ be a boundary hypersurface in $X$. We say that a smooth function $\rho: X \longrightarrow [0,\infty)$ is a \emph{boundary defining function} for $H$ if $\rho^{-1}(0)=H$ and $d_x \rho \neq 0$, for any $x \in H$.
\end{definition} 

If $\rho$ and $\rho'$ are boundary defining functions for $H$, then there exists a unique smooth function $u:X \longrightarrow (0, \infty)$ such that $\rho'=u \rho$. Furthermore, if $H$ is a boundary hypersurface of $X$, then $H$ is a $p$-submanifold if and only if $H$ has a boundary defining function (see for instance \cite{cursgrieser}).

A function $f=(f_1,...,f_{n'}): U_1 \longrightarrow U_2$ between open sets $U_1 \subset \mathbb R^n_k$, $U_2 \subset \mathbb R^{n'}_{k'}$ is a \emph{b-map} if for any $j=\overline{1,k'}$, we either have $f_j \equiv 0$ or $f_j \lp x_1,...,x_k,y \rp = a_j(x,y) \prod_{i=1}^k x_i^{e_{ij}}$, where $a_j$ is a strictly positive smooth function and the $ e_{ij} $'s are natural numbers.

\begin{definition}\label{bmap}
Let $f: X \longrightarrow Y$ be a smooth function between manifolds with corners. We say that $f$ is a \emph{b-map} if for any boundary hypersurface $H$ of $Y$, either $f^* \rho_H \equiv 0$, or 
\[f^* \rho_H = a_H \prod_{G \in \mathcal F_1 (X)} \rho_G^{e(G,H)}, \]
where $a_H$ is a strictly positive smooth function, all the $e(G,H)$'s are natural numbers, and we denoted by $\rho_G$ a boundary defining function for the boundary hypersurface $G$.
\end{definition}

First, the condition $f^* \rho_{H} \equiv 0$ is equivalent to saying that $f(x) \in H$ for all $x \in H$, which means that the image of $f$ is a subset of $H$. If this case never appears, we call $f$ an \emph{interior b-map}. Second, let $G$ be a fixed boundary hypersurface of $X$ and $H$ a boundary hypersurface of $Y$. If the coefficient $e(G,H)=0$, then $f^* \rho_H >0$ on the interior of $G$, so $\rho_H \lp f(x) \rp >0 $ for any interior point $x$ of $G$, which implies that $f(\accentset{\circ}{G}) \cap H = \emptyset$, meaning that the points in the image of $\accentset{\circ}{G}$ are far away from $H$. Finally, if $e:=e(G,H)>0$, then $f^* \rho_H = a \rho_G^e$, where the function $a$ is strictly positive on a neighborhood of the interior of $G$. It follows that $\rho_H (f(p))=a(p) \rho_G^e(p)$, thus if $p$ is at a certain distance $\rho$ of $\accentset{\circ}{G}$, then $f(p)$ is approximately at distance $\rho^e$ from $H$.

\section{Blowing-up \texorpdfstring{$p$-submanifolds}{p submanifolds}}
In order to blow-up a point $p$ in $X$, we first remove it, and then we add a point for each direction (half-line) which emanates from $p$ inside $X$ (see e.g. \cite{grieserquasi}). The simplest example is introducing polar coordinates on $\RR^n \setminus \{ 0 \}$, i.e., the blow-up of the origin in $\RR^n$ which is given by
\[  \left[ \RR^n; \{ 0 \} \right]:= [0, \infty) \times S^{n-1}, \]
together with the blow-down map
\begin{align*}
 \beta: \left[ \RR^n; \{ 0 \} \right] \longrightarrow \RR^n, && \beta(r, \omega)= r \omega.
\end{align*} 
The new added face $\ff:= \beta^{-1} \lp \{ 0 \} \rp$ is called the \emph{front face} of the blow-up. Following  \cite[Lecture 5]{cursgrieser}, we obtain the following result. 

\begin{proposition}\label{bwmwc}
Let $U_i^{\pm}= \left\{ (r, \omega) \in [0, \infty) \times S^{n-1}: \  \pm \omega_i >0  \right\}$, for $i=\overline{1,n}$, where we denoted the components of $\omega$ by $( \omega_1,...,\omega_n)$. The space $[\RR^n, \{ 0 \}]$ is a manifold with the following $2n$ charts:
\begin{align*}
{}&\varphi_i^{+} : U_i^{+} \longrightarrow \{ z_i >0   \} \times \RR^{n-1} \subset \RR^n, &&\varphi_i^{+} (r, \omega)= \lp r \omega_i, \frac{\omega_i'}{\omega_i}  \rp,
\end{align*}
where $\omega_i'=(\omega_1,...,\hat{\omega_i},...,\omega_n)$ is obtained from $\omega$ by erasing the $i^{th}$ component. 
\end{proposition}

Remark that $\beta :   \left[ \RR^n, \{ 0 \} \right]  \setminus \ff \longrightarrow \RR^n \setminus \{ 0 \} $ is the identity diffeomorphism. Furthermore, we define
\[  \left[ \RR^n_k ; \{ 0 \} \right]:= [0, \infty) \times \lp S^{n-1} \cap \RR^n_k \rp ,  \]
which is again a manifold with corners using similar charts as in Proposition \ref{bwmwc}. Remark that blow-ups typically can produce corners of higher codimension than the initial manifold. 

In general, the blow-up of a point $p$ in a manifold with corners $X$ is given by 
\[    \left[ X; \{ p \} \right]:= X \setminus \{ p \}  \sqcup S_p^+ X, \]
where $S_p^+ X:= \lp  T_p^+ X \setminus \{ 0 \} \rp / (0, \infty) $ is the set of inward  directions starting at $p$. 

Let us fix $Y$ a \emph{p}-submanifold in $X$. Consider the \emph{normal bundle}
$  NY:=(TX)_{\vert_{Y}} / TY  $
and denote by $N^+Y$ the set of vectors in $NY$ which are represented by an inward pointing vector. Moreover, let $S^+Y$ be the unit vectors in $N^+Y$. Then as a set,
\[ [X; Y]:= \lp X \setminus Y \rp \sqcup S^+Y, \]
and if we denote by $\pi: S^+ Y \longrightarrow Y$ the projection, then the blow-down map is given by
\[
\beta (p)= \left\{
    \begin{array}{ll}
        p & \mbox{ if } p \in X \setminus Y, \\
        \pi(p) & \mbox{ if } p \in Y.
    \end{array}
\right.
\]
One can check that $[X;Y]$ has a structure of manifold with corners using projective coordinates as in Proposition \ref{bwmwc} for the local model:
\[ \left[ \RR^n_{k} \times \RR^{n'}_{k'}; \RR^n_{k} \times \{ 0 \} \right]:= \RR^n_k  \times  \left[ \RR^{n'}_{k'};  \{ 0 \} \right]. \]
The new added face $S^+Y$ is called the \emph{front face} and it is a fibration of rank $n-m-1$ over the blown-up $p$-submanifold $Y$, where $n$ is the dimension of the ambient $X$, and $m$ is the dimension of $Y$. Moreover, $\beta$ is a \emph{b}-map. 

Furthermore, if $T$ is a connected subspace of $X$, we define the \emph{lift} of $T$ to $[X; Y]$ as
\[
\beta^* (T)= \left\{
    \begin{array}{ll}
        \beta^{-1} (T) & \mbox{ if } T \subset Y, \\
        \overline{\beta^{-1} (T \setminus Y)} & \mbox{ if  $T$ is not a subset of $Y$}.
    \end{array}
\right.
\]

If we have a sequence of p-submanifolds $Y_1 \subset X$,...,$Y_k \subset \left[ \left[ \left[ X;Y_1 \right]; Y_2 \right]...; Y_{k-1} \right]$, then we define the \emph{iterated blow-up}
$   \left[ X; Y_1;... ;Y_k \right] :=  \left[ \left[ \left[ X ;Y_1 \right] ; Y_2 \right] ...; Y_{k} \right].      $

\begin{definition}
Let $X$ be a manifold with corners, and let $Y,Z \subset X$ be p-submanifolds.
\begin{itemize} 
\item[$i)$] We say that $Y$ and $Z$ intersect \emph{cleanly} if  for any point $p \in Y \cap Z$, there exists local coordinates centered at $p$ in which both $Y$ and $Z$ are subspaces of coordinates. 
\item[$ii)$] We say that $Y$ and $Z$ intersect \emph{transversally} if they intersect cleanly and for any point $p \in Y \cap Z$ we have $ T_pY + T_p Z = T_p X.  $
\end{itemize}
\end{definition}
One can prove that in the context of the previous definition, if $Y$ and $Z$ intersect cleanly and $\beta: [X; Y] \longrightarrow X$ is the blow-down map, then $\beta^* Z$ is a p-submanifold in $[X; Y]$. Now we can state the result regarding the commutation of blow-ups (see e.g. \cite[Section~2.2]{melmaz})).

\begin{theorem}\label{comuteblowups}
Let $X$ be a manifold with corners, and let $Y,Z \subset X$ be p-submanifolds which intersects cleanly. Then $\left[ \left[ X;Y \right] ; Z \right] \simeq \left[ \left[ X;Z \right] ; Y \right]$ (in the sense that the lift of $\id : X \setminus \lp Y \cup Z \rp \longrightarrow  X \setminus \lp Y \cup Z \rp$ to the blow-ups extends smoothly up to the boundaries) if and only if $Y$ and $Z$ either intersect transversally, or $Y \subset Z$, or $Z \subset Y$.
\end{theorem}

\section{The {\it b}-tangent space}
We say that a vector field $V \in \mathcal V(X)$ is a \emph{b-vector field} if $V$ is tangent to all the boundary hypersurfaces of $X$. We denote the set of \emph{b}-vector fields by $\mathcal V_b(X)$ and remark that it forms a Lie algebra. One can check that in local coordinates $\lp x_1,...,x_k,y_1,...y_{n-k} \rp$, such a vector field can be written as
\[ V=\sum_{i=1}^n a_i x_i \partial_{x_i} + \sum_{j=1}^{n-k} b_j \partial_{y_j},  \]
where the $a_i$'s and the $b_j$'s are smooth functions on $X$.

Using the Serre-Swan Theorem, there exists a bundle ${}^b TX$ over $X$ such that its sections are exactly the \emph{b}-vector fields. It is  the celebrated \emph{b-tangent bundle} (see e.g. \cite[Section 2]{mel92}) and it is the proper replacement for the tangent bundle $TX$ in the context of \emph{b}-geometry. The fiber of ${}^b TX$ to a point $p \in X$ can be described as
\[ {}^bT_pX = \mathcal V_b(X) / I_p  \mathcal V_b(X),   \]
where $I_p=\{  f \in \mathcal C^{\infty} (X) : f(p)=0 \}$ is the ideal inside $\mathcal C^{\infty}(X)$ of smooth functions vanishing at $p$. 
There exists a natural inclusion map 
$ \iota: {}^b TX \longrightarrow TX, $
which is a diffeomorphism over the interior of $X$. The \emph{b-normal space} at $p \in X$ is defined as 
$ {}^b N_pX:=\Ker \iota_p: {}^b T_pX \longrightarrow T_pX, $
and remark that its dimension is equal to $\codim p $.

If $f :X \longrightarrow Y$ is an interior \emph{b}-map between manifolds with corners, one can check that its differential $ d_p f : T_pX \longrightarrow T_{f(p)} Y  $ extends by continuity and defines the \emph{b-differential} ${}^b d_p f$ which makes the following diagram commutative:
\begin{equation}
\centering
\begin{split}
\begin{tikzpicture}
[x=1mm,y=1mm]
\node (tl) at (0,15) {$T_p X$};
\node (tr) at (30,15) {$T_{f(p)} Y$};
\node (bl) at (0,0) {${}^b T_p X$};
\node (br) at (30,0) {${}^b T_{f(p)} Y.$};
\draw[->] (bl) to (tl);
\draw[->] (br) to (tr);
\draw[->,densely dashed] (bl) to node[below,font=\small]{${}^b d_pf$} (br);
\draw[->] (tl) to node[above,font=\small]{$d_pf$} (tr);
\end{tikzpicture}
\end{split}
\end{equation}

The set of \emph{b-differential operators} is the universal envelopping algebra of the Lie algebra $\mathcal V_b(X)$:
\[   \Diff_b^m (X):=\left\{ a + \sum_{k=1}^m V_{i_1}...V_{i_k}: \ a \in \mathcal C^{\infty}(X), \ V_{i_l} \in \mathcal V_b(X) \right\}. \]

\section{{ \it b}-fibrations}
Following \cite[Section 2]{mel92}) , an interior \emph{b}-map $f: X \longrightarrow Y$ between manifolds with corners is called a \emph{b-fibration} if 
\begin{itemize}
\item[$i)$] $f$ is a \emph{b-submersion} i.e., $d_p f: {}^b T_p X \longrightarrow {}^{b}T_{f(p)} Y$ is surjective for any $p \in X$.
\item[$ii)$] $f$ is \emph{b-normal} i.e., $d_p f: {}^b N_p X \longrightarrow {}^{b}N_{f(p)} Y$ is surjective for any $p \in X$.
\end{itemize}

Notice  that $f: \mathbb R^2_2 \longrightarrow \mathbb R^1_1$, $f(x_1,x_2)=x_1x_2$ is a \emph{b}-fibration, but the blow-down map $\beta: [\mathbb R^2_2 ; 0] \longrightarrow \mathbb R^2_2$ is \emph{not} a \emph{b}-fibration. Recall the following result from \cite[Proposition 2.4.2]{cartemelrose}.
\begin{proposition}\label{echivbnormal}
A \emph{b}-map $f: X \longrightarrow Y$ is \emph{b}-normal if and only if for any boundary hypersurface of $X$, $e(H,H') \neq 0$ (see Definition \ref{bmap}) for at most one boundary hypersurface $H'$ of $Y$.
\end{proposition}

Let $f: \mathbb R^3_3 \longrightarrow \mathbb R^2_2$ given by  $f(x_1,x_2,x_3)=(x_1':=x_1x_2^3,x_2':=x_2 x_3^3)$ and consider a point $p$ in the plane $ \{ x_2=0 \}$. Then $\codim p=1$, while $\codim f(p) = 2$. It follows that $f_* :{}^b N_p \RR_3^3 \longrightarrow {}^{b}N_{f(p)} \RR^2_2$ is not surjective, thus $f$ is not \emph{b}-normal.
The direct implication of Proposition \ref{echivbnormal} is an easy generalisation of this example. Before proving the other implication, we need the following remark (see e.g. \cite[Section 2]{mel92}).

\begin{remark}\label{rmkbnorm}
If a \emph{b}-map $f:X \longrightarrow Y$ is \emph{b}-normal, then we can choose coordinates such that 
\[ f^*x_j'=\prod_{r \in I_j} x_r^{e(r_j,j)},  \]
where the sets $I_j \subset \{1,2,...,k \}$ are disjoint. Indeed, since $f$ is a \emph{b}-map, we have that
\[ f^*x_1'=a_1 \prod x_{i_1}^{e(i_1,1)}...x_{i_l}^{e(i_l,1)}.  \]
Using the direct implication of Proposition \ref{echivbnormal} and the fact that $f$ is \emph{b}-normal, it follows that $x_{i_1}$ does not appear in any other $f^* x_j'$, $j \neq 1$. Adjust the boundary defining function $x_{i_1}$ such that the term $x_{i_1}^{e(i_1,1)}$ absorbs the positive function $a_1$. We proceed in a similar way for the rest of $x_j'$'s.
\end{remark}

\begin{proof}[Proof of Proposition \ref{echivbnormal}]
"$\Leftarrow$" By Remark \ref{rmkbnorm}, we can suppose that $f^* x_1'=x_{i_1}^{\alpha_1}...x_{i_k}^{\alpha_{i_k}}$. Then one can easily check that $f_*(x_{i_1} \partial_{x_{i_1}})=\alpha_1 x_1' \partial_{x_1'}$. Since ${}^b N_{f(p)}Y$ is generated by $x_i' \partial_{x_i'}$ and these are attained through $f_*$, it follows that $f_*: {}^b N_p X \longrightarrow {}^b N_{f(p)}Y$ is surjective.
\end{proof}

If $f$ is a \emph{b}-submersion, we can choose coordinates such that $f^* y_j'=y_j$ for all $j \in \{1,2,...,n'-k' \}$.  Furthermore, one can check that a composition of two \emph{b}-fibrations remains a \emph{b}-fibration. We conclude this section with the following statement.
\begin{proposition}\label{bfib}
An interior \emph{b}-map $f: X \longrightarrow Y$ between manifolds with corners is a \emph{b}-fibration if and only if we can find local coordinates on $X$ and $Y$
such that f is a projection between the $y$-type variables, and any $x_i$ appears in the \emph{b}-map formula of the pull-back $f^* x_j$ for at most one of the $x_j'$'s:
\[ (x_1,...,x_k,y_1,...,y_{n-k}) \longmapsto (x_1',...,x_l',y_1',...,y_{m-l}') . \]
\end{proposition}

\section{Polyhomogeneous conormal functions}
A set $E \subset \CC \times \mathbb N$ is called an \emph{index set} if the following conditions are satisfied:
\begin{itemize}
\item[1)] For any $s \in \mathbb R$, the set $E_{\leq s}:= \{ (z,k) \in E : \ \Re z \leq s  \}$ is finite.
\item[2)] If $(z,k) \in E$ and $l \geq k$, then also $(z,l) \in E$.
\item[3)] If $(z,k) \in E$, then $(z+1,k) \in E$.
\end{itemize}

An index set which satisfies the third condition is actually called a \emph{smooth} index set; we will only work with this type of index sets. An \emph{index family} for a manifold with corners $X$ is a family of index sets for all the boundary hypersurfaces $H$ of $X$. From now on, we will denote by $0$ or $\mathbb N$ the index set $\{ (n,0) : \ n \in \mathbb N \}$, and we denote by $\infty$ the empty index set. 

The \emph{extended union} of two index sets $E,F$ is defined as
\[  E \overline{\cup} F = E \cup F \cup \{ (z,k+l+1): (z,k)\in E, (z,l) \in F \}. \]

Let us introduce polyhomogeneous conormal functions, i.e., smooth functions on the interior of a manifold with corners, but with a certain behavior towards the boundary. Following \cite[Section 1.4]{gripolyhom}, \cite[Section 2.2]{grieserquasi}, we start with the simpler case of $\RR^2_2$. 

\begin{definition}\label{defconormal}
For any $s,t \in \RR$, the \emph{conormal functions} on $\RR_+^2$ are 
\[ \Aa^{(s,t)} \lp \RR^2_2 \rp:= \left\{ u \in \mathcal C^{\infty} \lp (0, \infty)^2 \rp : \ (x \partial_x)^j (y \partial_y)^k u(x,y) = \mathcal O (x^sy^t) \ \forall j,l \in \mathbb N  \right\} . \]
\end{definition}

\begin{definition}\label{defpoly}
We say that $u \in \Aa^{E,F} \lp \RR^2_2 \rp$ is a \emph{polyhomogeneous conormal function} on $\RR^2_2$ with index sets $E,F$ if $u$ is smooth on the interior $(0, \infty)^2$ and there exist $\lp a_{z,k} \rp_{(z,k) \in E} \in \Aa^{F} \lp {(0, \infty)}_{y} \rp$, $\lp b_{w,l} \rp_{(w,l) \in F} \in \Aa^{E} \lp {(0, \infty)}_{x} \rp$, and a natural number $N$ such that 
\begin{align*}
{}&u(x,y)= \sum_{(z,k) \in E_{\leq s}} a_{z,k}(y) x^z \log^k x + R_s(x,y) \text{ and} \\
{}&u(x,y)= \sum_{(w,l) \in F_{\leq t}} b_{w,l}(x) y^w \log^l y + R'_t(x,y),
\end{align*} 
where the remainders are conormal functions $R_s \in \Aa^{s,-N}(\RR^2_2)$, $R'_t \in \Aa^{-N,t}(\RR^2_2)$ 
\end{definition}

The key in Definition \ref{defpoly} is the existence of the natural number $N$ which provides the compatibility near the corner. In the same manner as in Definitions \ref{defconormal} and \ref{defpoly}, we define the \emph{conormal} functions and \emph{polyhomogeneous conormal} functions on the local model $\RR^n_k$. One can prove the following result (see e.g. \cite[Proposition 1.24]{gripolyhom}).

\begin{proposition}
If $\EE$ be an index family for $\RR^n_k$, then $\Aa^{\EE} \lp \RR^n_k \rp$ is invariant to coordinate changes.
\end{proposition}
Remark that a function $u \in \Aa^{\EE} \lp \RR^n_k \rp$ if and only if $\rho u \in \Aa^{\EE} \lp \RR^n_k \rp$, for any $\rho \in \mathcal C^{\infty} \lp \RR^n_k \rp$. Thus, following \cite[Section 1.6]{gripolyhom}, for an open set $U \subset \RR^n_k$, we can define the \emph{polyhomogeneous conormal} functions on the open set $U$ as smooth functions $u$ on $U$ intersected with the interior of $\RR^n_k$ such that $\rho u \in \Aa^{\EE} \lp \RR^n_k \rp $ for all $ \rho \in \mathcal C^{\infty}_{c} (U)$. 

Let $X$ be a manifold with corners, let $\EE=({\EE}_H)_{H \in \mathcal F_1 (X)}$ be an index family for $X$, i.e., a family of index sets for all the boundary hypersurfaces $H \in \mathcal F_1(X)$. Furthermore, let $s=(s_H)_{H \in \mathcal F_1(X)}$ be a family of real numbers. Denote by $\Aa^s(X)$ the set of \emph{conormal} functions on $X$, i.e., smooth functions $f$ on the interior of $X$ which satisfy the condition $Pf = \mathcal O (\rho^s),$ for any $P \in \Diff_b^*(X)$, where 
\[  \rho^s:= \prod_{H \in \mathcal F_1(X)}  \rho_H^{s_H}.  \]

\begin{definition}
Let $X$ be a manifold with corners and $\EE$ an index family for $X$. We denote by $\Aa^{\EE}(X)$ the set of \emph{polyhomogeneous conormal functions} $f$ on $X$ i.e., the smooth functions on the interior of $X$ which are polyhomogeneous conormal in  each chart with the corresponding index sets. More precisely, for any chart $\varphi: U \longrightarrow \tilde{U} \subset \RR^n_k$, 
\[u \circ \varphi^{-1} \in \Aa^{\tilde{\EE}} \lp \tilde{U} \rp, \]
where $\tilde{\EE} \lp \varphi (H \cap U) \rp:=\EE(H)$, for any boundary hypersurface $H$ which intersects $U$.
\end{definition}

\section{Densities and {\it b}-densities}
We say that $\mu : TX^{\wedge} {}^{n} \longrightarrow [0, \infty)$ is an \emph{$\alpha$-density} on the manifold with corners $X$ of dimension $n$ if for any real number $\lambda \neq 0$ and any $u \in TX^{\wedge} {}^n$
$ \mu (\lambda u) = \vert \lambda \vert^{\alpha} \mu(u).  $
We denote the space of $\alpha$-densities by $\Omega^{\alpha}(X)$ and notice that it is a $\RR$-vector bundle of rank $1$. We can integrate $1$-densities with compact support using partitions of unity, and furthermore, we regard a density $\mu$ as a distribution in the following way
\begin{align*}
\mu : \mathcal C^{\infty}_{c} (X) \longrightarrow \RR, &&
\varphi \longmapsto \int_X \varphi \mu.
\end{align*}
If $f : X \longrightarrow Y$ is a map between manifolds with corners and $\mu$ is a density on $X$, one can check that \emph{the push-forward} $f_* \mu$ is well-defined as a distribution
\[ \int_Y \varphi f_* \mu : = \int_X \lp f^* \varphi \rp \mu . \]
If $X$ is a product space, then the push-forward is integration in the fibers, and the result is given by a density. In the study of \emph{b}-geometry, it is suitable to work with \emph{b}-densities.
\begin{definition}
Let $\EE$ be a family of index sets for the manifold with corners $X$. We say that $\mu \in \Aa^{\EE} \lp X , \Omega(X) \rp$ if $\mu$ is a smooth density on the interior of $X$ and in local coordinates
\[ \mu=v(x,y) dx_1...d x_k dy_1...dy_{n-k}, \]
where $v$ is  polyhomogeneous conormal function on $X$ with index set $\EE$.

Furthermore, we say that $\mu$ is a \emph{b-density} $\mu \in \Aa^{\EE} \lp X , \Omega^b(X) \rp$ if locally
\[ \mu=v(x,y) \frac{dx_1}{x_1}...\frac{d x_k}{x_k} dy_1...dy_{n-k}, \]
where $v \in \Aa^{\EE}(X)$. 
\end{definition}

\begin{proposition}\label{ridbdens}
Let $X$ be a manifold with corners, let $Y \subset X$ be a p-submanifold, and consider $\mu \in \Omega^b(X)$ a \emph{b}-density. If the blow-up $[X;Y]$ is locally given by $[\RR^n_k; 0 ]$ and $\beta: [X; Y] \longrightarrow X$ is the blow-down map, then 
\[ \beta^* \mu \in \rho_{\ff}^{n-k} \Omega^b \lp [X; Y] \rp, \]
where $\rho_{\ff}$ is a boundary defining function for the front face of the blow-up.
\end{proposition}
\begin{proof}
It suffices to prove the result locally. Consider $[ \RR^n_k; \{ 0 \}  ]$ and denote the coordinates on $\RR^n_k$ by $(x_1,...,x_k,y_1,...,y_{n-k})$, where the $x_j$'s are positive variables of $x$-type and $y_j$'s are real variables. Consider the standard \emph{b}-density
\[ \mu = \frac{dx_1}{x_1}\wedge...\wedge \frac{dx_k}{x_k} \wedge dy_1 \wedge...\wedge dy_k \in \Omega^b \lp \RR^n_k \rp. \]
A first chart on the blow-up space $[\RR^n_k; \{ 0 \}]$ is given by the projective coordinates
\[ \lp x_1, \xi_2:= \frac{x_2}{x_1},..., \xi_k:= \frac{x_k}{x_1}, \eta_1:= \frac{y_1}{x_1},...,\eta_k:= \frac{y_k}{x_1}  \rp. \]
Then the lift of $\mu$ through the blow-down map $\beta$ is given by
\begin{align*}
{}&  \frac{dx_1}{x_1}\wedge \frac{x_1 d\xi_2 + \xi_2 dx_1}{x_1 \xi_2}...\wedge \frac{x_1 d \xi_k + \xi_k dx_1}{x_1 \xi_k} \wedge \lp x_1 d\eta_1 + \eta_1 dx_1 \rp \wedge...\wedge \lp x_1 d\eta_{n-k} + \eta_{n-k} dx_1  \rp \\
={}& x_1^{n-k} \frac{d x_1}{x_1} \wedge \frac{d \xi_2}{\xi_2}\wedge...\wedge \frac{d \xi_{k}}{\xi_{k} } d\eta_1 \wedge ... \wedge d\eta_{n-k} \in \rho_{\ff}^{n-k} \Omega^b \lp [\RR^n_k; \{ 0 \}] \rp,
\end{align*}
where we denoted by $\rho_{\ff}$ the boundary defining function for the front face of the blow-up (which in these coordinates is $x_1$). In all the other projective coordinates in which we divide by $x_j$, $j = \overline{2, n-k}$, the computations work in a similar manner. Let us now consider the projective coordinates
\[ \lp X_1:=\frac{x_1}{y_1},..., X_k:=\frac{x_k}{y_1}, y_1, Y_2:= \frac{y_2}{y_1},..., Y_{n-k}:= \frac{y_{n-k}}{y_1} \rp \]
and notice that $y_1$ is a boundary defining function for the front face. 
It follows that
\begin{align*}
\beta^* \mu= {}& \frac{y_1 dX_1}{X_1 y_1} \wedge ...\wedge  \frac{y_1 dX_k}{y_1 X_k} \wedge dy_1 \wedge \lp y_1 dY_2 \rp \wedge ... \wedge \lp y_1 dY_{n-k} \rp \\
={}& y_1^{n-k} \frac{dX_1}{X_1} \wedge...\wedge \frac{d X_k}{X_k} \wedge \frac{dy_1}{y_1} \wedge dY_2 \wedge...\wedge Y_{n-k} \in \rho_{\ff}^{n-k} \Omega^b \lp [\RR^n_k; \{ 0\} ] \rp,
\end{align*}
We proceed in a similar way for all the remaining projective charts where we divide by $y_j$'s.
\end{proof}

\section{The Pull-Back and the Push-Forward Theorems for polyhomogeneous conormal functions}
Now we are able to state the Push-Forward and Pull-Back Theorems (the proofs can be found in \cite[Theorem 3, Theorem 5]{mel92}).

\begin{theorem}[The Pull-back Theorem for polyhomogeneous conormal functions]\label{pullbackphg}
Let $f: X \longrightarrow Y$ be an interior \emph{b}-map between manifolds with corners and let $u \in \mathcal A^{\EE}(Y)$ be a polyhomogeneous conormal function on $Y$ with index set $\EE$. Then the pull-back $f^*u $ belongs to $ \mathcal A^{f^* \EE}(X)$, where for every $G \in \mathcal F_1(X)$,
\[  \lp f^* \EE \rp (G):=\sum_{H \in \mathcal F_1(Y)} e(G,H) \EE (H). \]
\end{theorem}

\begin{theorem}[The Push-forward theorem for polyhomogeneous conormal functions]\label{pft}
Let $f:X \longrightarrow Y$ be a \emph{b}-fibration between manifolds with corners and let $\EE$ be an index family for $X$. Suppose that $0< \inf \EE(G) \lp := \inf \{ \Re z : (z,k) \in \EE(G) \} \rp$ for each boundary hypersurface $G$ which is sent through $f$ in the interior of $Y$ (\emph{the integrability condition}). If $\mu \in \mathcal A \lp X, \Omega_b(X)  \rp$ and $f$ is proper on the support of $\mu$, then the push-forward $f_* \mu$ belongs to $\mathcal A^{f_* \EE} \lp Y,  \Omega_b(Y)  \rp$, where
\[  f_*(\EE)(H)=\overline{\bigcup}_{e(G,H)>0}  \frac{1}{e(G,H)} \EE(G).  \]
\end{theorem}

\begin{proposition}\label{desumflare}
Let $X$ be a manifold with corners, $Y \subset X$ a $p$-submanifold. Then the blow-down map 
$  \beta: [X;Y] \longrightarrow X $
provides a bijection between functions on $[X;Y]$ having empty index set towards the front face, and functions on $X$ which vanish rapidly towards the submanifold $Y$.
\end{proposition}

\section{Conormal distributions}
The central idea in Melrose's program of studying singular spaces is to describe the classical pseudodifferential operators (and later the \emph{b}-pseudodifferential operators, the fibered cusp operators, etc) as \emph{conormal distributions} on certain (maybe blown-up) geometric spaces. 

Following \cite[Chapter 2]{melroseconormal}, let $\overline{\RR^n}$ be the radial compactification of $\RR^n$, and consider the following boundary defining function for $\partial \overline{\RR^n}$
\[ \rho:= \frac{1}{(1+ |\xi|^2)^{1/2}}. \] 

The space of classical symbols of order $m$ on $\RR^n$ is exactly the space $\rho^{-m} \mathcal C^{\infty}(\overline{\RR^n})$. The Fourier transform and the inverse Fourier transform of a Schwartz function $\varphi \in \Ss(\RR^n)$ are defined as
\begin{align*}
\mathcal{F}(\varphi)(\xi) := \int_{\RR^n} e^{-ix\cdot \xi} \phi (x) dx, &&  \mathcal{F}^{-1}(\varphi)(x) = \frac{1}{2\pi} \int_{\RR^n} e^{ix \cdot \xi} \varphi (\xi) d\xi.
\end{align*}
and we extend it to tempered distributions by duality. 
\begin{definition}\label{con0}
We consider the space of distributions on $\RR^n$ which are \emph{conormal at the origin}
\[ \II^m \lp \RR^n, \{ 0 \} \rp := \mathcal{F}^{-1}\left( \rho^{-m} C^{\infty}(\overline{\RR^n}) \right).
\] 
\end{definition}
Thus the Fourier transform of a conormal distribution of order $m$ is a classical symbol of order $m$ on $\RR^n$. Let us remark that $\delta_0 \in \II^0 \lp \RR^n, \{ 0 \} \rp$ has a singularity at $\{ 0 \}$, but it is smooth everywhere else. Actually this statement holds true for any conormal distribution in that space (see e.g. \cite[Lemma 5]{noterares}).

\begin{proposition}
If $\psi \in \mathcal C^{\infty}_{c}(\RR^n)$ is a compactly supported function which is equal to $1$ in a neighborhood of the origin, then for any conormal distribution $u \in \II^{m}(\RR^n, \{ 0\}) $,
\[(1-\psi)u \in \mathcal{S}(\RR^n). \] 
\end{proposition}

Furthermore, one can prove that for any $m \in \mathbb C$,
$ \Ss (\RR^n) \cdot \II^m \lp \RR^n, \{ 0 \} \rp \subset \II^m \lp \RR^n, \{ 0 \} \rp.   $

In order to extend the definition of a conormal distribution at the origin of an arbitrary vector space $W$, an useful trick is to slighty modify the definition of the Fourier transform to take into account densities: 
\begin{align*}
\mathcal F : \Ss(W) \longrightarrow \Ss \lp W',\Omega W' \rp, && \mathcal F (u) (\xi)=  \lp \int_W e^{-i \xi(x)} u(x) \mu_W(x) \rp \mu_{W'}(\xi),
\end{align*}
where $\mu_W$ and $\mu_{W'}$ are corresponding densities in the bundle $W$ and its dual $W'$. Then 
\[ G^*u(x) = \mathcal F^{-1} \lp \lp {}^t G^{-1} \rp^*  a  \rp (x). \]

\begin{definition}\label{distrconvect}
The space of conormal distribution at the origin of a vector space $W$ is given by
\[  \II^m \lp W, \{ 0 \} \rp : = \mathcal F^{-1} \lp \rho^{-m} \mathcal C^{\infty} \lp \overline{W} , \Omega W' \rp   \rp .  \]
\end{definition}

Furthermore, Definition \ref{distrconvect} also works for defining conormal distributions on a trivial bundle $W$ over a manifold $X$ (the points on $X$ are regarded as parameters, and we perform the Fourier transform in the fibers of the trivial bundle). In general, to give the definition of a distribution on a (non-trivial) vector bundle conormal to its zero section, one proceeds locally (see for instance \cite[(L3.50)]{melroseconormal} or \cite[Section 1.3]{noterares}).
  
\begin{definition}
The space of conormal distributions of order $m$ on a vector bundle $W$ conormal to its zero section is given by
\[ \II^m \lp W, 0_W \rp:=\mathcal F_{\fib} \lp    \rho^{-m} \mathcal C^{\infty} \lp \overline{W^*}, \Omega_{\fib} \rp \rp.  \]
\end{definition}

In order to introduce  distributions of order $m$ on a manifold which are conormal to a submanifold $Y \subset X$, it would be useful to regard $X$ as a bundle over $Y$ at least locally. We are able to do so by using the Collar neighborhood Theorem (see for instance \cite[Theorem 1]{melroseconormal}).

\begin{theorem}[Collar Neighborhood Theorem]\label{collar}
Let $X$ be a compact manifold and let $Y \subset X$ be a closed submanifold. Then there exist an open set $Y \subset D \subset X$, an open set $D'$ contained in the normal bundle $NY$ and a diffeomorphism $f:D \longrightarrow D'$ such that $f_{\vert_Y}: Y \longrightarrow NY$  identifies Y with the zero section $0_{NY}$ in the normal bundle and
$ d_y f_{\vert_{N_yY}}:N_yY \longrightarrow N_yY  $
is the identity.
\end{theorem}

\begin{definition}
Let $X$ be a compact manifold with corners, and let $Y \subset X$ a closed submanifold. A distribution $u \in \mathcal C^{-\infty}(X)$ is said to be \emph{conormal to $Y$} if it can be written as $u=u_1+u_2$, where $u_2 \in \mathcal C^{\infty}(X)$, and $u_1=f^* v$, for $v \in I^m(NY,0_{NY})$ with compact support and $f$ is a diffeomorphism as in Theorem \ref{collar}.
\end{definition}

Now we can state the Push-forward and Pull-back theorems for conormal distributions (see for instance \cite[Theorem 3.17, Theorem 3.18]{grieser}).
\begin{theorem}[Push-forward of conormal distributions]\label{pftc}
Let $f : X \longrightarrow Y$ be a fibration between manifolds with corners, and let $Z \subset X$ be a submanifold such that for any point $z \in Z$
\[ T_z Z \cap T_z \lp f^{-1} (f(z)) \rp = \{0 \}.  \]
Let $\mu \in \II^m(X,Z)$ be a distributional density on $X$, conormal to $Z$, such that $f$ is proper on the support of $\mu$. Then
\begin{itemize}
\item[$i)$] If $f_{\vert_Z} : Z \longrightarrow f(Z)$ is a diffeomorphism, then the push-forward $f_* \mu$ is a smooth density.
\item[$i)$] Otherwise, $f(Z)$ is a proper submanifold of $Y$ and the push-forward is a conormal distributional density to $f(Z)$:
$f_* \mu \in \II^m \lp Y, f(Z) \rp. $
\end{itemize}
\end{theorem}

\begin{theorem}[Pull-back of conormal distributions]\label{pbtc}
Let $f : X \longrightarrow Y$ be a \emph{b}-map between manifolds with corners, and consider $Z \subset Y$ a submanifold such that for any point $y:=f(x) \in Y$, the tangent space $T_yY$ is spanned by $T_y Z$ and $df(T_xX)$.
Then $f^{-1}(Z) \subset X$ is a submanifold, and if $u \in \II^m(Y,Z)$ is a conormal distribution on $Y$ with respect to $Z$, then the pull-back $f^* u \in \II^m \lp X, f^{-1}(Z) \rp$ is also a conormal distribution. 
\end{theorem}

If $u$ is a tempered distribution on $\RR^n$ with compact support, one can define its \emph{wavefront set} $\WF (u)$, which is a subset in $\RR^n \times \lp \RR^n \setminus \{ 0 \} \rp$ (see for instance \cite[Section 4.2]{wfmelrose}). If $u_1,u_2 \in \mathcal C^{-\infty}_c(\mathbb R^n)$ are two compactly supported distributions satisfying
\[  (x, \xi) \in \WF(u_1) \Longrightarrow (x, -\xi) \notin \WF(u_2),  \]
then one can define the \emph{product of distributions} as
$ u_1 \cdot u_2: = \mathcal F^{-1} \lp \mathcal F(u_1) * \mathcal F(u_2) \rp, $
which is s a well-defined tempered distribution with
\begin{align*}
\WF(u_1 \cdot u_2) \subset{}& \left\{ (x, \xi) : \ x \in \supp u_1, \ (x,\xi) \in \WF(u_2)  \right\} \cup  \left\{ (x, \xi) : \ x \in \supp u_2, \ (x,\xi) \in \WF(u_1)   \right\}  \\
{}&  \left\{ (x, \xi) : \ \xi=\eta_1+\eta_2, \ (x,\eta_1) \in \WF(u_1), \ (x,\eta_2) \in \WF(u_2)  \right\},
\end{align*}
see for instance \cite[Proposition 4.13]{wfmelrose}.
Furthermore, these notion make sense for conormal distributions on manifolds with corners, see e.g. \cite[Theorem 3.19]{grieser} and  \cite[Appendix B]{epstein91}. 

\section{The cusp-surgery simple space}\label{calculndim}
We now proceed towards the construction of our adapted pseudodifferential calculus which will be used as a tool in the context of pinching a simple closed curve in a smooth compact oriented surface. Actually we are able to construct the \emph{cusp-surgery calculus} in a more general setup.  

Albin-Rochon-Sher \cite{ars} introduced the surgery fibered calculus in order to study the spectrum of the Hodge Laplacian on a closed manifold which degenerates to a manifold with fibered cusps. The calculus that we need here is a particular case of their surgery fibered calculus. First, our geodesic fibrates over a single point. Secondly, our condition of the invertibility of the spin structure in Hypothesis \ref{hypot} will imply that the resolvent family belongs to the small calculus, in the sense that it will have non-trivial polyhomogeneous behavior just towards the cusp diagonal face, and to the temporal boundary face at $\{ t=0 \}$ (see Fig. \ref{doublesp}). Furthermore, we rely on a self-contained analysis of blow-down maps in projective coordinates, defining and proving at the same time the existence of a \emph{b}-fibration from the double space to the simple space, and from the triple space to the double space. We hope that our construction will be helpful for those readers in search of a complete proof, but who may not be inclined to invest the effort required to fully familiarize themselves with all the subtleties of the theory of \emph{b}-fibrations between manifolds with corners.

As in \cite{ars}, let $X$ be a smooth $n$-dimensional compact manifold, and let $\gamma \subset X$ be a hypersurface. Let $x$ be a smooth function defined on a tubular neighborhood of $\gamma$ whose zero locus is exactly $\gamma$ and such that $d_p x \neq 0$ for any $p \in \gamma$. Consider $(h_t)_{t \geq 0}$ a family of metrics on $X \setminus \gamma$ which locally near $\gamma$ look like
\begin{align*}
h_t = (x^2+t^2)^{\frac{2-n}{n}}  \lp \frac{dx^2}{x^2 + t^2} + (x^2+t^2)g_{\gamma} \rp,
\end{align*}
where $g_{\gamma}$ is a Riemannian metric on $\gamma$. In particular, if $\gamma$ is a simple closed curve in a compact surface $X$, the family of metrics $(h_t)_{t \geq 0}$ describe a pinching process along $\gamma$ as in Section \ref{introducere}.

For $t>0$, $h_t$ extends to a smooth metric on $X$, and notice that the volume form of $h_t$ does not depend on $t$. The family $(h_t)_{t \geq 0}$ is singular as both $x,t \to 0$, thus following \cite{mcdonald}, we define \emph{the simple space} by the blow-up of $\gamma$ at time $t=0$ inside $[0,\infty) \times X$
\begin{equation}\label{spsimplu}
\Xs: = \left[ [0, \infty) \times X; \{ t=0 \} \times \gamma \right].
\end{equation}

Let $\ff$ be the front face of $\Xs$, and let $\tf$ be the temporal face, i.e., the lift of the old temporal boundary $\{ t=0 \}$ to $\Xs$. Denote by 
\begin{equation}\label{beta}
\beta : \Xs \longrightarrow [0,\infty) \times X
\end{equation}
the blow-down map and let us describe the projective coordinates on the simple space. We will denote by $y$ a $(n-1)$-tuple of coordinates on $\gamma$ and since it is a hypersurface without corners in $X$, all the components of $y$ are of $y$-type. 

First, the local coordinates $\lp t, \frac{x}{t},y  \rp$ are valid everywhere except at ${\beta}^{-1} \lp \{ t=0 \} \rp$, and $t$ is a boundary defining function for the front face. Secondly, the coordinates $\lp \frac{t}{x}, x, y \rp$ are suitable for $\Xs \setminus  {\beta} ^{-1} \lp \{x=0\} \rp$ in the region where $x > 0$, and the front face in this chart is given by the boundary defining function $x$. Finally,  the coordinates $\lp \frac{t}{-x},- x, y \rp$ work on the part where $x<0$, and in this chart, the boundary defining function for the front face is $-x$ (see Fig. \ref{simplesp}, where the $y$ variables are not represented since they play the role of parameters in this blow-up).

\begin{figure}[H]
\begin{center}
\includegraphics[width=12.5cm, height=6cm]{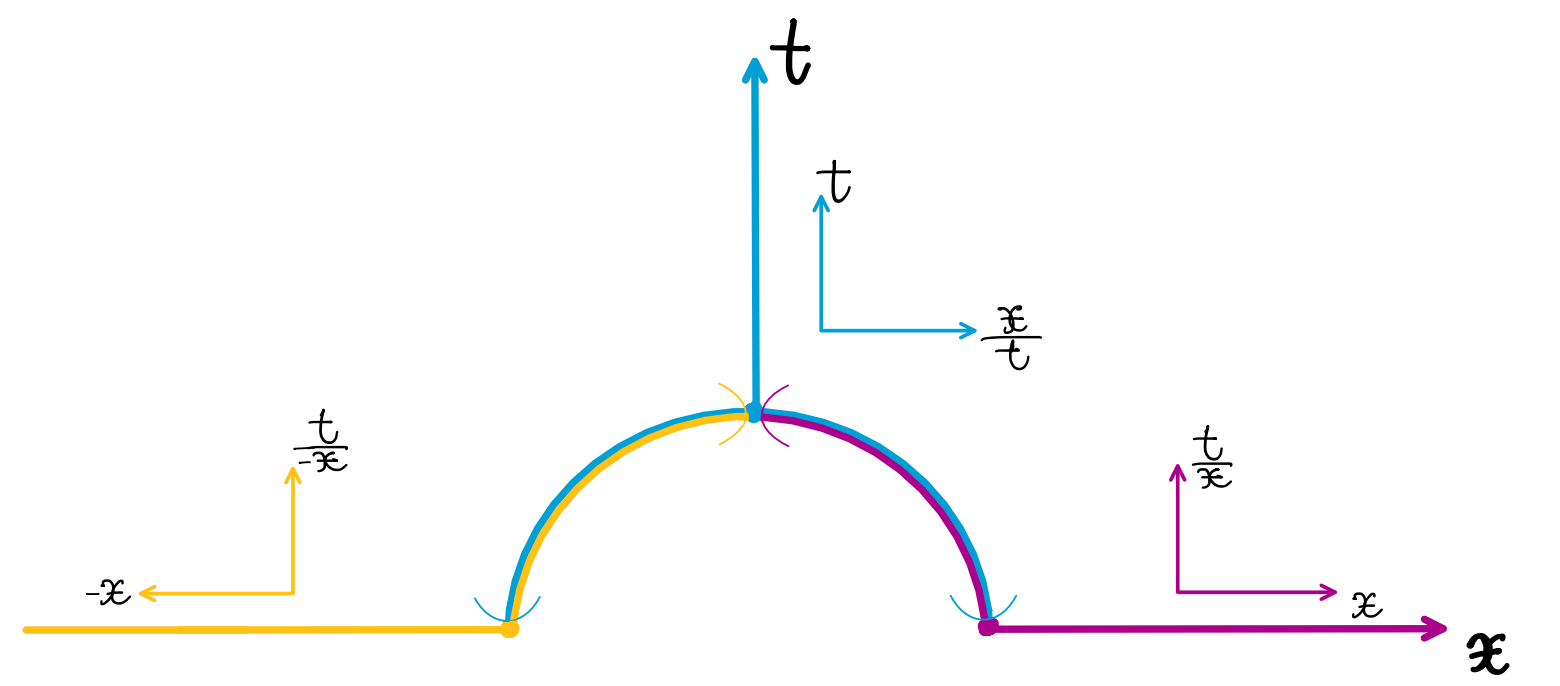}
\caption{The cups-surgery simple space with $y \in \gamma$ omitted}\label{simplesp}
\end{center}
\end{figure}

If $\Ss$ is a smooth vector bundle over the simple space $\Xs$, we denote by $\mathcal A^{0, 0} \lp \Xs, \Ss \rp$ the polyhomogeneous conormal sections on the simple space with values in $\Ss$ which have the index set $\mathbb N$ towards the front face $\ff$ and towards the temporal face $\tf$. For later use, we introduce 
\[ \mathcal A^{\alpha,\beta} \lp \Xs , \Ss \rp:= \rho_{\ff}^{\alpha} \rho_{\tf}^{\beta} \mathcal A^{0,0}  \lp  \Xs, \Ss \rp.  \]

\begin{remark}\label{notatie}
We denote by $X_{(i)}$, $i=\overline{1,2}$, the surface $X$ in the $i^{th}$ set of variables. For example, if we discuss about the product space $X \times X \times [0, \infty)$ and we need to be specific about which factor of $X$ we use, we will denote by $X_{(1)}$ the first factor, and by $X_{(2)}$ the second factor. We will also use these notation related to $\gamma$ or the densities coming from certain factors in a product space.
\end{remark}

\section{The cusp-surgery double space}
Denote by $\ffb$ the face obtained through the blow-up of the submanifold $\{t=0 \} \times \gamma_{(1)} \times \gamma_{(2)}$  inside $ [0, \infty) \times X^2$ and let $\tb$ be the lift of the temporal boundary $\{ t=0 \}$ through this blow-up. Remark that $\lp t, \xi_1:=\frac{x_1}{t},\xi_2:= \frac{x_2}{t} ,y_1, y_2 \rp$ are projective local coordinates near $\ffb$, suitable in the region where $x_1, x_2 \geq 0$, and away from $\tb$ (see Fig. \ref{bw1}, where $y_1$ and $y_2$ are parameters). 
\begin{figure}[H]
\begin{center}
\includegraphics[width=11.5cm, height=6.5cm]{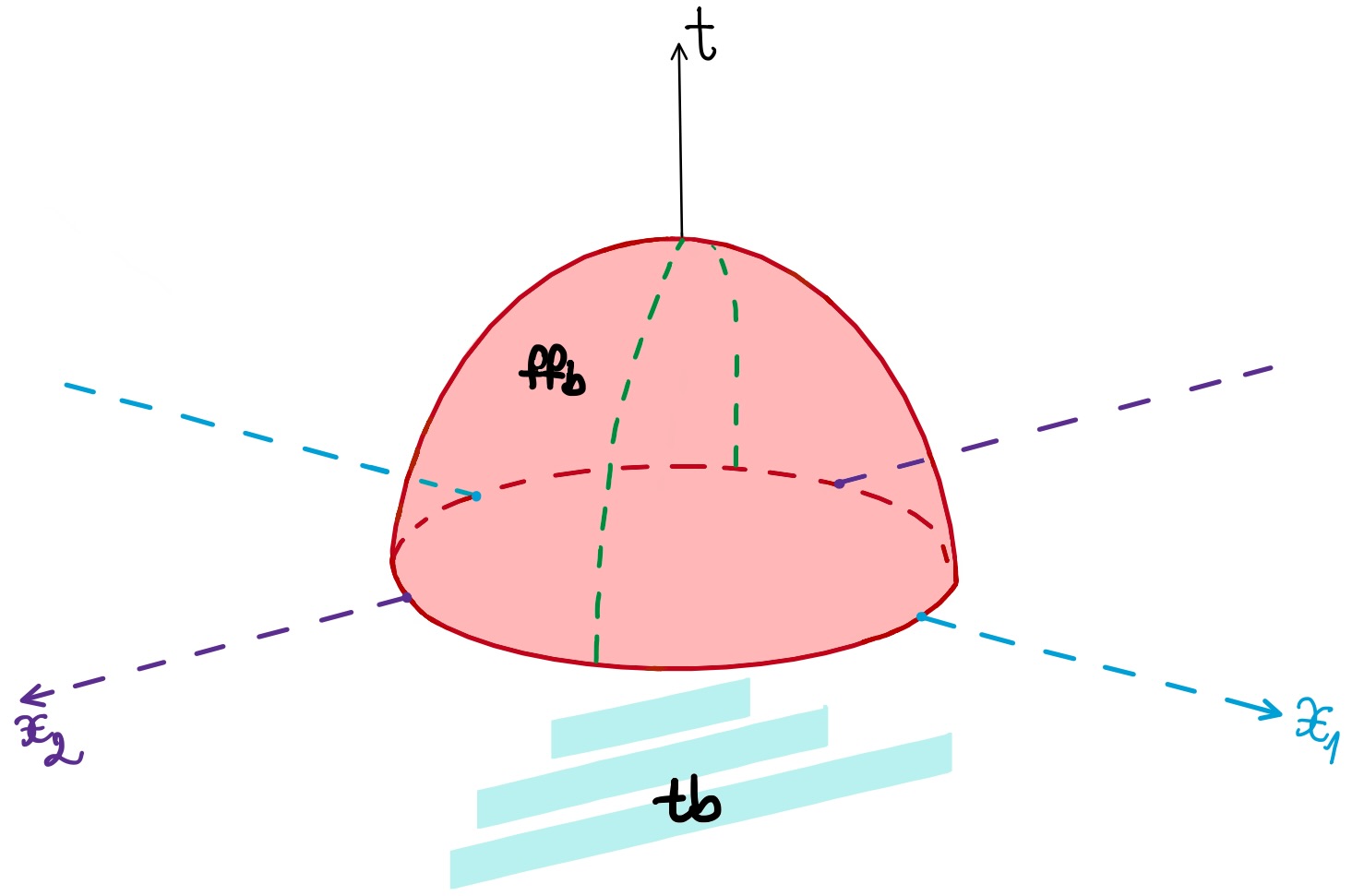}
\caption{The first blow-up in the \emph{b}-surgery double space}\label{bw1}
\end{center}
\end{figure}

Consider the \emph{cusp-surgery double space} $\Xd$ defined as
\begin{equation}\label{spdublu}
 \left[ [0, \infty) \times X^2 ; \{ t=0 \} \times \gamma_{(1)} \times \gamma_{(2)}; \{t=0\} \times \gamma_{(1)} ; \{t=0 \} \times \gamma_{(2)} ;  \lp \Diag \times [0, \infty) \rp \cap \ffb  \right] ,
\end{equation}
where $\Diag$ is the diagonal inside $\lp X \setminus \gamma \rp^2$, and we re-denoted by $\ffb$ the lift of the face $\ffb$ from $X^2_{\text{b}}:=\left[  [0, \infty) \times X^2 ; \{t=0 \} \times \gamma_{(1)} \times \gamma_{(2)} \right] $ (see Fig. \ref{bw1}) through the two blow-ups of the $p$-submanifolds $\{t=0\} \times \gamma_{(1)}$ and $\{t=0 \} \times \gamma_{(2)}$ (see Fig. \ref{doublesp}). Using the fact that $x$ defines $\gamma$ locally, we might abuse the notation and describe the double space as the following iterated blow-up
\begin{equation}\label{dublunerig}
 \Xd=\left[ [0, \infty) \times X^2 ; \{ t= x_1=x_2=0 \} ; \{t=x_1=0 \} ; \{t=x_2=0 \} ; \overline{ \{ \xi_1=\xi_2=t=0 \} } \right] .
\end{equation}

We will use the same notation $\ffb$ for the lift of the old $\ffb$ face on $\Xd$. Let $\tb$ be the temporal boundary of the double space $\Xd$, i.e., the lift of the old boundary $\{ t=0 \}$ in $\mathbb R_{+} \times X^2$ through the blow-down map
\begin{equation}\label{tildebeta}
\tilde{\beta} : \Xd \longrightarrow [0, \infty) \times X \times X.
\end{equation}

Furthermore, let $\ffc$ be the \emph{cusp front face} obtained through the last blow-up in \eqref{spdublu} and notice that $\ffc$ is a line bundle over $[0, \pi] \times \gamma \times \gamma$. Let $\Ba$ and respectively $\Bm$ be the new added faces through the blow-ups of $\{ t=0 \} \times \gamma_{(1)}$ and $\{ t=0 \} \times \gamma_{(1)}$ lifted to the double space $\Xd$.

\begin{figure}[H]
\begin{center}
\includegraphics[width=13cm, height=8cm]{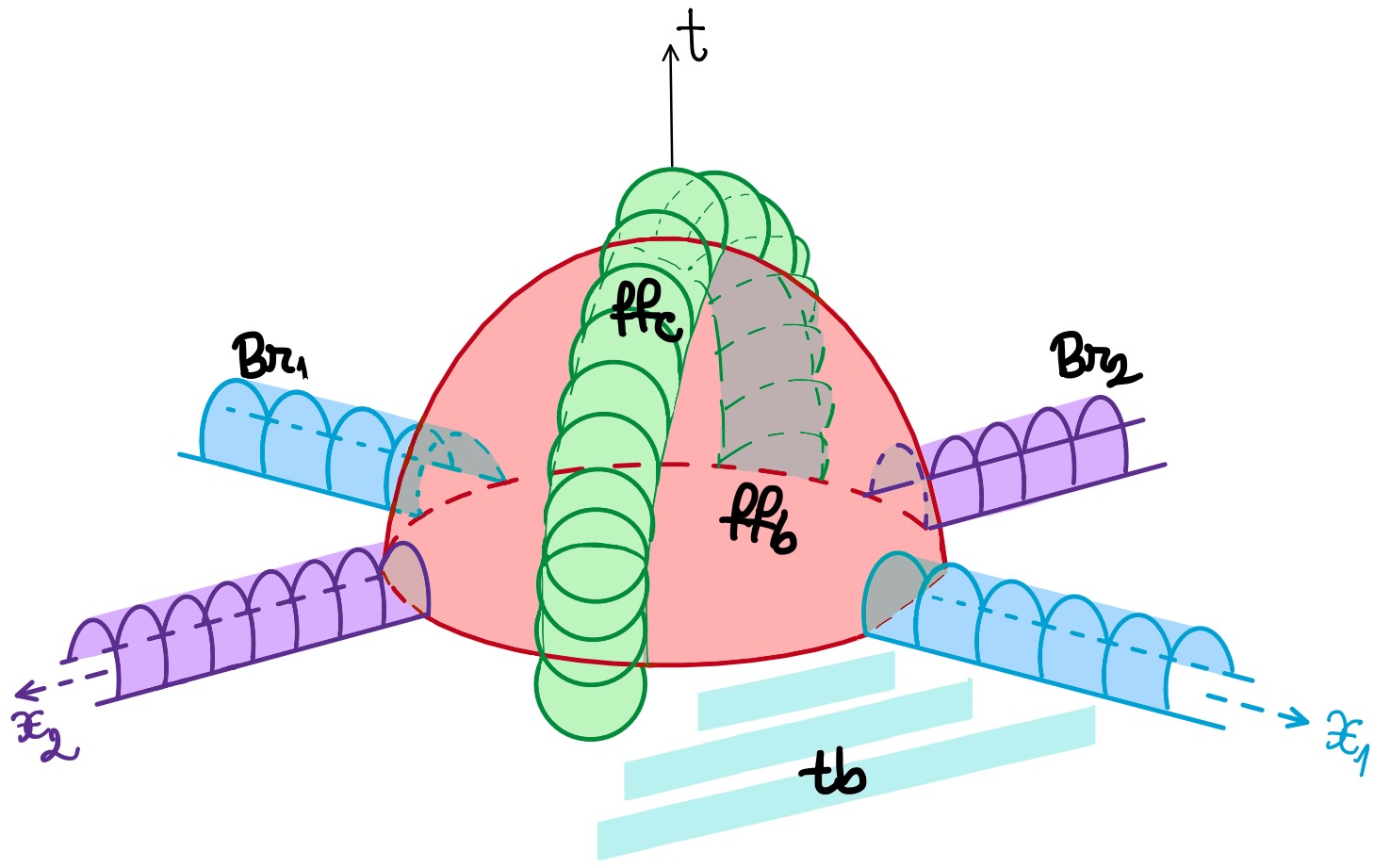}
\caption{The cusp-surgery double space with $y_1,y_2$ omitted}\label{doublesp}
\end{center}
\end{figure}

The reason for blowing-up $\{ t=0 \} \times \gamma_{(1)}$ and $\{ t=0 \} \times \gamma_{(2)}$ and producing the faces $\Ba$ and $\Bm$ is purely technical, because we need to obtain a \emph{b}-fibration over  the simple space $\Xs$. However, the definition of our cusp-surgery calculus will not involve these faces.

Now we define a ``projection" map $\pi^2_{(1)}$ between the double space $\Xd$ and the simple space $\Xs$ onto the first set of variables. Using Theorem \ref{comuteblowups}, since $ \{t= x_1=x_2=0 \} \subset \{t=x_1=0 \}$, we can commute the first two blow-ups in \eqref{spdublu}, thus $\Xd$ can be written as 
\[ \left[ [0, \infty)  \times X^2  , \{t=x_1=0 \} , \{t= x_1=x_2=0 \},   \{t=x_2=0 \},  \overline{ \{ \xi_1=\xi_2=t=0 \} } \right]. \]
Similar to \cite[Section~3.3]{mcdonald}, we define $\pi^2_{(1)}$ as the composition between the projection on the first factor $p_1 : \Xs \times X  \longrightarrow \Xs$ and the blow-down map $b$, as follows
\begin{equation}\label{mapx2cp}
\begin{tikzpicture}
  \matrix (m) [matrix of math nodes,row sep=3em,column sep=5em,minimum width=2em]
  {
       \Xs \times X =\left[ [0, \infty) \times X^2, \{t=x_1=0 \} \times X  \right] & \Xs\\
	   \Xd & {} \\  
  };    
  \path[->]
  (m-1-1) edge node [ above] {$p_1$} (m-1-2) 
  (m-2-1) edge node [left]{$b$} (m-1-1)
  (m-2-1) edge node [right]{$ \pi^2_{(1)}:=p_1 \circ b$} (m-1-2);
  
\end{tikzpicture}
\end{equation}
Remark that for $t>0$, $\pi^2_{(1)}$ is simply the projection onto the first factor.

\begin{proposition}\label{bfibspdublu}
The map $\pi^2_{(1)}: \Xd \longrightarrow \Xs$ is a \emph{b}-fibration.
\end{proposition}
\begin{proof}
Using Proposition \ref{bfib}, we will prove that $\pi^2_{(1)}$ is a \emph{b}-fibration locally near each point in $\Xd$ which has positive $x_1$ and $x_2$ coordinates. Similar computation work for all the other points in $\Xd$. From now on, we denote all the $x$-type variables in the sense of Section \ref{sem1} in bold. Furthermore, in all the projections that we will describe, we also have the coordinates $(y_1,y_2) \in \gamma^2$. They are just parameters in the blow-up, and the map $\pi^2_{(1)}$ acts on them as the projection onto the first factor $\pi^2_{(1)} : (y_1, y_2) \longmapsto y_1$. 
To simplify the notation, we will omit the coordinates along $\gamma$ in the computations below, since all of them are of $y$-type, and they are mapped into the  $y_1$ coordinates, which are again of $y$-type in the sense of Section \ref{sem1}.
We distinguish the following $15$ cases. 
\begin{figure}[H]
\begin{center}
\includegraphics[width=14cm, height=6cm]{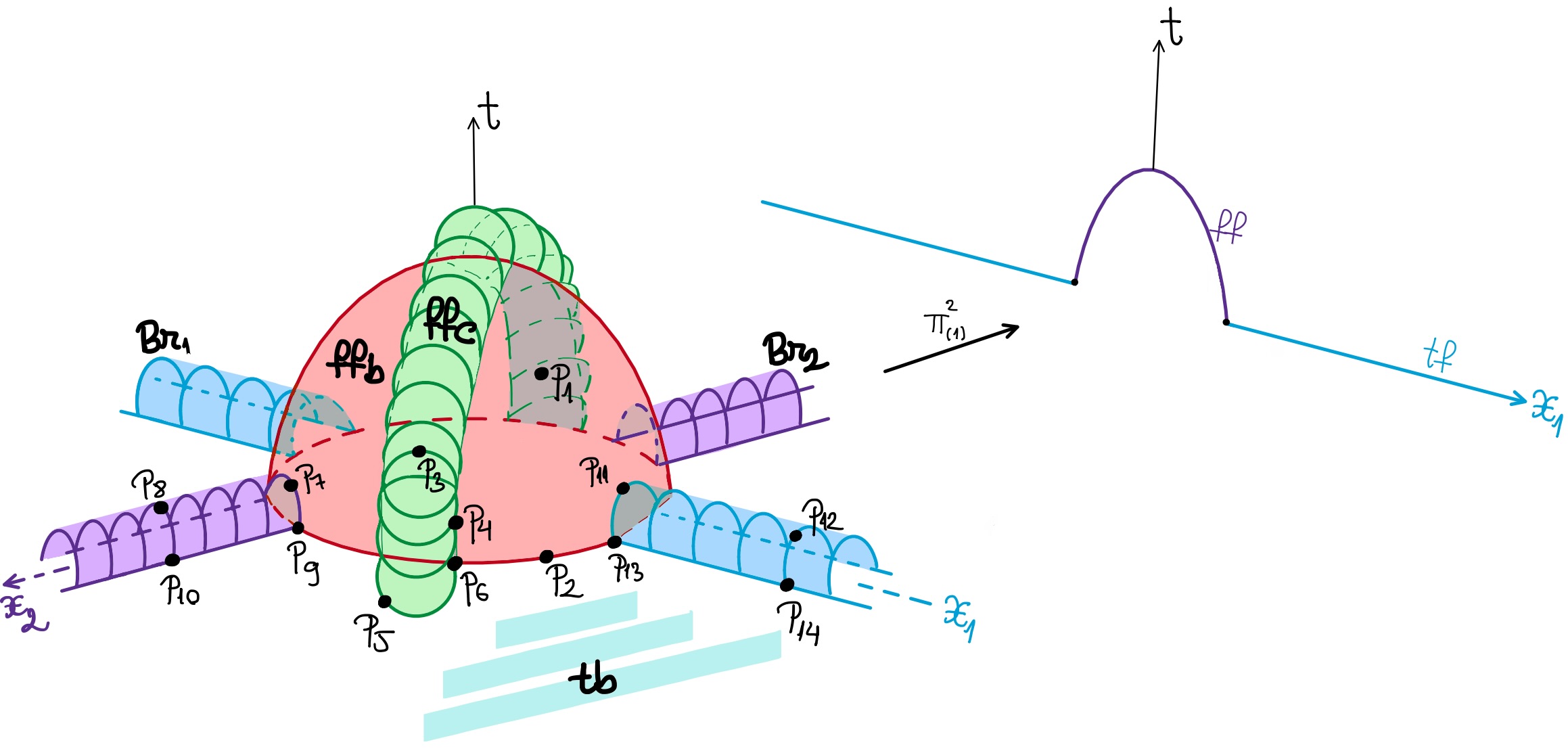}
\caption{The map $\pi^2_{(1)}$ is a \emph{b}-fibration}\label{bfibdublusimplu}
\end{center}
\end{figure}

\begin{enumerate}
\item A point $p_1 \in \Int(\ffb)$ (hence of codimension $1$, see Fig. \ref{bfibdublusimplu}) is mapped through $\pi^2_{(1)}$ on the interior of the front face of $\Xs$ and in local coordinates near $p_1$, the map $\pi^2_{(1)}$ is given by 
\begin{equation}\label{rosu1}
\lp \bm{ t},  \xi_1:=\frac{x_1}{t},\xi_2:=\frac{x_2}{t}  \rp  \longmapsto \lp \bm{ t} ,  \xi_1 \rp.
\end{equation}
\item A point $p_2 \in \Int(\ffb \cap \tb)$ is mapped through $\pi^2_{(1)}$ into one of the corners of codimension $2$ in $\Xs$, thus $\pi^2_{(1)}$ is locally given by
\begin{equation}\label{rosu2}
\lp  \bm{ \tau:=\frac{t}{x_1}}  , \bm {x_1} , \eta_2:=\frac{x_2}{x_1}  \rp \longmapsto \lp  \bm{ \tau}, \bm{ x_1} \rp.
\end{equation}

\item A point $p_3 \in \Int(\ffc)$ is mapped through $\pi^2_{(1)}$ on the front face of $\Xs$. We first change the coordinates on $\ffb$ in \eqref{rosu1} linearly to $ \lp t,  \xi_1,  u:=\xi_1-\xi_2 \rp$. Then the second blow-up of $\left\{ t=0,u=0 \right\}$ introduces near $p_3$ the coordinates $\lp { \bf t}, \xi_1, \frac{u}{t} \rp$ which are mapped through $\pi^2_{(1)}$ to $\lp {\bf t},  \xi_1 \rp$. 

\item Furthermore, a point $p_4 \in \Int(\ffc \cap \ffb)$ is mapped into the corners of $\Xs$, hence $\pi^2_{(1)}$ is locally given by
\[ \lp  \bm{ \frac{t}{u}} , \xi_1, \bm{ u} \rp \longmapsto \lp  \bm{ t=u \cdot \frac{t}{u}}, \xi_1 \rp.   \]

\item A point $p_5 \in \Int(\ffc \cap \tb)$ of codimension $2$ is mapped by $\pi^2_{(1)}$ into one of the corners of $\Xs$. Near $\tb$ the coordinates are $\lp \tau,  x_1, \eta_2 \rp$ as in \eqref{rosu2}, and after the second blow-up of $\left\{ x_1=0,  \eta_2=1 \right\}$, $\pi^2_{(1)}$ locally looks like
\[ \lp \bm{ \tau},  \bm{ x_1}, \frac{\eta_2-1}{x_1}   \rp \longmapsto \lp  \bm{ \tau},  \bm{ x_1} \rp.  \]

\item Furthermore, for $p_6 \in \Int(\ffb \cap \ffc \cap \tb)$, the map $\pi^2_{(1)}$ in suitable local coordinates near $p_6$ looks like
\[ \lp \bm{\tau},  \bm{\frac{x_1}{\eta_2-1}}, \bm{\eta_2-1}   \rp \longmapsto \lp  \bm{ \tau},  \bm{ x_1}  \rp.   \]

\item Let now $p_7 \in \Int(\Bm \cap \ffb)$  and remark that $\pi^2_{(1)}(p)$ belongs to the front face of $\Xs$. Since the coordinates $\lp \tau':=\frac{t}{x_2} , \mu_1:= \frac{x_1}{x_2}, x_2 \rp$ are suitable near $p_7$, after the second blow-up of $\left\{ \tau'=0,  \mu_1=0 \right\}$, $\pi^2_{(1)}$ is given by
\begin{equation}\label{movrosu}
 \lp  \bm{  \tau'},u:= \frac{\mu_1}{\tau'},  \bm{ x_2}  \rp  \longmapsto \lp \bm{ t=\tau' x_2}, \frac{x_1}{t}=u \rp. 
\end{equation}
\item The coordinates \eqref{movrosu} are also suitable for points $p_8 \in \Int(\Bm)$, the only difference is that the coordinate $x_2$ is then of $y$-type, so it can be absorbed into $\tau'$. 

\item Furthermore, near a point $p_9 \in \Int(\ffb \cap \Bm \cap \tb)$, $\pi^2_{(1)}$ locally looks like
\[  \lp  { \bf \frac{\tau'}{\mu_1}},  {\bf x_2},  {\bf \mu_1} \rp  \longmapsto \lp { \bf x_1= \mu_1 x_2 },  {\bf \frac{t}{x_1}= \frac{\tau'}{\mu_1}}  \rp.     \]

\item These coordinates also work for $p_{10} \in \Int(\Bm \cap \tb)$, the only difference in such a point is that $x_2$ is now a $y$-type variable, but it can be absorbed into $\mu_1$.

\item A point $p_{11} \in \Int(\Ba \cap \ffb)$ is mapped through $\pi^2_{(1)}$ to a corner of codimension $2$ in $\Xs$. Since the coordinates $\lp \tau, x_1, \eta_2 \rp$ from \eqref{rosu2} are suitable near $\ffb$, after the blow-up of $\left\{ \eta_2=0, \tau=0 \right\}$, $\pi^2_{(1)}$ looks near $p_{11}$ as
\begin{equation}\label{albastrurosu}
  \lp \bm{ \tau}, \bm{ x_1}, \frac{\eta_2}{\tau}   \rp \longmapsto \lp \bm{ \tau}, \bm{ x_1}   \rp.  
\end{equation}
\item The coordinates in \eqref{albastrurosu} also work for $p_{12} \in \Int(\Ba)$, the only difference is that $x_1$ is now of $y$-type. 

\item Furthermore, for $p_{13} \in \Int(\Ba \cap \ffb \cap \tb)$, the map $\pi^2_{(1)}$ is locally given by
\[    \lp \bm{ \frac{\tau}{\eta_2}} ,  \bm{ x_1}, \bm{ \eta_2}    \rp  \longmapsto \lp  \bm{ \tau} , \bm{ x_1}   \rp.  \]
\item These coordinates also work for $p_{14} \in \Int(\Ba \cap \tb)$, the only difference in such a point is that $x_1$ is now a $y$-type variable in both LHS and RHS. 
\item Finally, if $p_{15} \in \Int(\tb)$, then $\pi^2_{(1)}$ locally looks like
\[    \lp \bm{t} ,   x_1, x_2    \rp  \longmapsto \lp  \bm{ t} , { x_1}   \rp.    \]   
\end{enumerate} 
Using Proposition \ref{bfib} and the local forms of the map described above, we conclude that $\pi^2_{(1)}$ is indeed a \emph{b}-fibration.  
\end{proof} 

In this case, it was easy to construct the function $\pi^2_{(1)}$ in \eqref{mapx2cp} using a result of commutations of blow-ups. We remark that the existence of the function $\pi^2_{(1)}$ is also a consequence of the computations made in the above cases $(1)-(15)$. For constructing the map from the triple space to the double space, the strategy will be to simply use local computations as in this proof of Proposition \ref{bfibspdublu}. 

We note for later use that the ``projection" onto the second set of variables $\pi^2_{(2)}: \Xd \longrightarrow \Xs$ is, of course, also a \emph{b}-fibration.
Denote by $\mathcal A^{0, 0} \lp \Xd \rp$ the set of polyhomogeneous conormal sections on the double space which have index sets $\mathbb N$ towards the cusp front face $\ffc$ and towards the temporal boundary $\tb$. Moreover, we define
$ \mathcal A^{\alpha,\beta} \lp \Xd \rp:= \rho_{\ffc}^{\alpha} \rho_{\tb}^{\beta} \mathcal A^{0,0}  \lp  \Xd \rp.  $

\section{The cusp-surgery triple space}\label{sectiontriple}
In order to prove the composition theorem for our cusp-surgery calculus, we will need an auxiliary space, namely the \emph{cusp-surgery triple space} $\Xt$. We define it as an iterated blow-up of $p$-submanifolds in $[0, \infty) \times X^3$ in the following way. First we blow-up
\begin{itemize}
\item[$\bullet$] $\{ t=x_1=x_2=x_3=0  \}$; Denote by $\fffb$ the newly added face
\item[$\bullet$] $\{ t=x_2=x_3=0  \}$; $\{ t=x_1=x_3=0  \}$; $\{ t=x_1=x_2=0  \}$;  \\
Denote by $\Bu$, $\Bd$, $\Bt$ the newly added faces
\item[$\bullet$] $\{ t=x_1=0  \}$; $\{ t=x_2=0  \}$; $\{ t=x_3=0  \}$;  \\
Denote by $\Pu$, $\Pd$, $\Pt$ the new faces.
\end{itemize}

Until here we obtain McDonald's \emph{b}-surgery triple space from \cite[Section 4.1]{mcdonald} and \cite[Section 4.9]{melmaz}, let us denote it by $X_{\sbm}^3$. Since we want the three ``projections" from the triple space $\Xt$ to the double spaces $\Xd_{(12)}$, $\Xd_{(23)}$, and $\Xd_{(13)}$ to be \emph{b}-fibrations, we also need to blow-up the following $p$-submanifolds:
\begin{itemize}
\item[$\bullet$] $\{ x_1=x_2=x_3  \} \cap \fffb$; Denote by $\fffc$ the newly added face
\item[$\bullet$] $\{  x_2=x_3  \} \cap \fffb$; $\{  x_1=x_3 \} \cap \fffb $; $\{  x_1=x_2  \} \cap \fffb$;  \\
Denote by $\Cu$, $\Cd$, $\Ct$ the newly added faces
\item[$\bullet$]  $\{ x_2=x_3  \} \cap B_1 $; $\{ x_1=x_3  \} \cap B_2 $; $\{  x_1=x_2  \} \cap B_3 $;    \\
Denote by $\Tu$, $\Td$, $\Tt$ the new faces.
\end{itemize}
Notice that all the blow-ups occur at $\{ t=0 \}$ inside the space $[0,\infty) \times X^3$. 

\begin{figure}[H]
\begin{center}
\includegraphics[width=14.8cm, height=9cm]{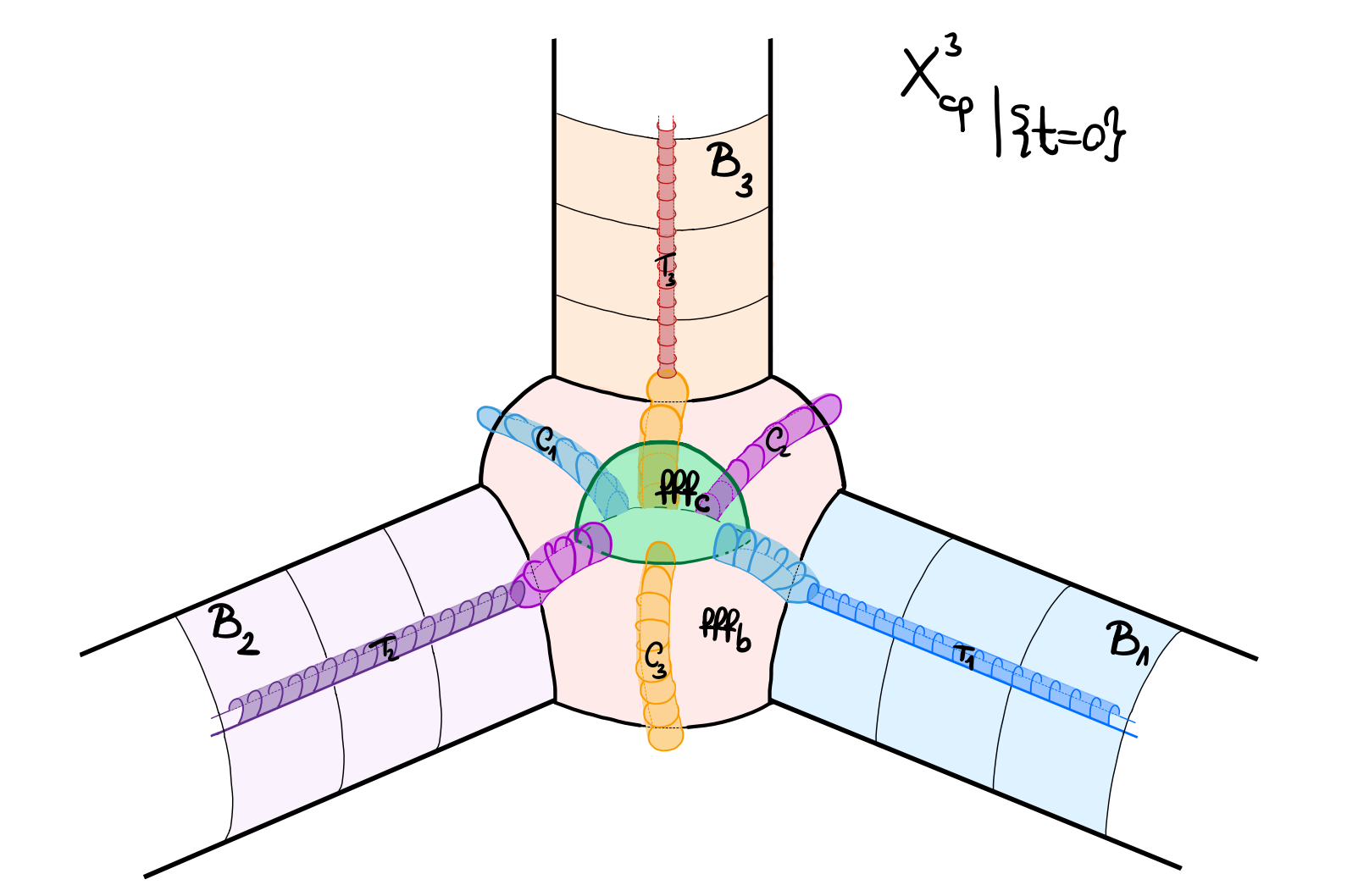}
\caption{The temporal boundary of the cusp-surgery triple space}\label{sptriplu}
\end{center}
\end{figure}

In Fig. \ref{sptriplu}, we represented the $\{ t=0 \}$ level in the cusp-surgery triple space $\Xt$. Notice that for simplicity, we draw just $1/8$ of the true figure (since the $x_1, x_2, x_3$ variables are of $y$-type, but we represented just the part of the figure having $x_1,x_2,x_3 \geq 0$).

\begin{proposition}\label{bfibtrip}
The projection $ \pi_{(12)}: (0,\infty) \times X^3 \longmapsto (0,\infty) \times X^2$ forgetting the last $X$ variable extends uniquely to a map from $\Xt$ to $\Xd$ which, furthermore, is a \emph{b}-fibration.
\end{proposition}
\begin{proof}
The plan is to look at all the interiors of boundary faces in $\Xt$. Near such a face $F$, we consider suitable projective coordinates in a neighborhood of $F$ intersected with the interior of $\Xt$, and also near the corresponding interior neighborhood of $\Xd$. Then we compute the local form of $ \pi_{(12)}$ in these coordinates. It will be evident from the computations that $ \pi_{(12)}$ \emph{extends} to the boundary face $F$ (such an extension is clearly unique by a density argument). Moreover, using Proposition \ref{bfib}, the map will also be a \emph{b}-fibration by inspection.

In all the cases below, we also have the variables $(y_1,y_2,y_3) \in \gamma^3$, they are just parameters in the blow-up. The interior projection acts on them as the projection onto the first two factors $\pi^3_{(12)} : (y_1,y_2,y_3) \longmapsto (y_1,y_2)$, and all of the coordinates are of $y$-type. To simplify the notation, we will omit to write the $(y_1,y_2,y_3)$ variables in the computations below.

We distinguish $19$ open faces for the \emph{b}-surgery triple space: the interiors of $\fffb$, $\Bu$, $\Bd$, $\Bt$, $\Pu$, $\Pd$, $\Pt$; the interiors of the $6$ intersections $B_i \cap P_j$, $i \neq j$ of codimension $2$; and $6$ faces of codimension $3$ arising as the intersection of $\fffb$ with the previous $6$ intersections of codimension $2$. Additionally, we have $38$ open faces of $\Xt$ which do not appear in $X_{\sbm}^3$, described as follows. First, we distinguish the interiors of $\fffc$ and $\fffc \cap \fffb$. Secondly, we have the $12$ faces arising from the ``vertical direction", i.e. the various interiors of intersections of $T_3$, $C_3$, $\Bt$, $\Pt$, $\fffb$ and $\fffc$, other than the ones already considered. Finally, we have $12$ open faces arising as above from the ``left direction" ($T_2$, $C_2$, $\Bd$, $\Pd$, $\fffb$, $\fffc$), and $12$ additional faces coming from the ``right" direction ($T_1$, $C_1$, $\Bu$, $\Pu$, $\fffb$, $\fffc$). 

Since we have a left-right symmetry, it is enough to check only $13$ faces corresponding to the b-surgery triple space, and $26$ cases for the faces of $\Xt$ which do not appear in $X_{\sbm}^3$, namely the two initial cases $\fffc$, $\fffc \cap \fffb$, and the $24$ more faces which arise from the ``vertical" and the ``right" direction. Note that the first $13$ cases corresponding to $X_{\sbm}^3$ are much easier to treat and some of them come as consequences of the faces corresponding to $\Xt$ treated below, see for example case $(8')$; for this reason, we omit them.

In the diagram \eqref{intersectiitriplu}, we highlight part of the intersection lattice of faces from the ``vertical" direction, i.e., among the hypersurfaces $\fffb$, $\fffc$, $\Ct$, $\Tt$, $\Bt$, $\Pt$. In each row, we list faces of corresponding codimension in $\Xt$. Each of them is the intersection of the closures of those faces from the previous line that are linked to it, and each segment represent an inclusion. 
\definecolor{black}{rgb}{0.0, 0.0, 0.0}
\definecolor{turquoise}{rgb}{0,0.7,0.5}
\definecolor{lightbrown}{rgb}{0.8,0.4,0.1}
\definecolor{lightgreen}{rgb}{0.5,1,0}
\definecolor{pinkish}{rgb}{0.8,0.1,0.5}
\definecolor{darkblue}{rgb}{0,0,0.5}
\definecolor{blue-green}{rgb}{0.0, 0.87, 0.87}
\definecolor{brightcerulean}{rgb}{0.11, 0.67, 0.84}
\definecolor{caribbeangreen}{rgb}{0.0, 0.8, 0.6}
\definecolor{carmine}{rgb}{0.59, 0.0, 0.09}
\definecolor{ceil}{rgb}{0.57, 0.63, 0.81}
\definecolor{darklavender}{rgb}{0.45, 0.31, 0.59}
\definecolor{darktangerine}{rgb}{1.0, 0.66, 0.07}
\begin{equation}
\label{intersectiitriplu}
\centering
\begin{split}
\begin{tikzpicture}
[x=1mm,y=1mm,scale=1.2]
\node (fffb) at (0,0) [fill,circle,inner sep=0.5mm] {};
\node (fffc) at (20,0) [fill,circle,inner sep=0.5mm] {};
\node (C3) at (40,0) [fill,circle,inner sep=0.5mm] {};
\node (T3) at (60,0) [fill,circle,inner sep=0.5mm] {};
\node (B3) at (80,0) [fill,circle,inner sep=0.5mm] {};
\node (P3) at (105,0) [fill,circle,inner sep=0.5mm] {};
\node at (fffb.north) [anchor=south] {$\fffb$};
\node at (fffc.north) [anchor=south] {$\fffc$};
\node at (C3.north) [anchor=south] {$\Ct$};
\node at (T3.north) [anchor=south] {$\Tt$};
\node at (B3.north) [anchor=south] {$\Bt$};
\node at (P3.north) [anchor=south] {$\Pt$};
\node (s2) at (10,-20) [fill,circle,inner sep=0.5mm] {};
\node (s5) at (20,-20) [fill,circle,inner sep=0.5mm] {};
\node (s4) at (30,-20) [fill,circle,inner sep=0.5mm] {};
\node (s1) at (40,-20) [fill,circle,inner sep=0.5mm] {};
\node (s9) at (50,-20) [fill,circle,inner sep=0.5mm] {};
\node (s13) at (60,-20) [fill,circle,inner sep=0.5mm] {};
\node (s12) at (70,-20) [fill,circle,inner sep=0.5mm] {};
\node (a2) at (84,-20) [fill,circle,inner sep=0.5mm] {};
\node (s7) at (95,-20) [fill,circle,inner sep=0.5mm] {};
\node at (s2.north) [anchor=south] {$F_2$};
\node at (s5.north) [anchor=south] {$F_5$};
\node at (s4.north) [anchor=south] {$F_4$};
\node at (s1.north) [anchor=south] {$S_1$};
\node at (s9.north) [anchor=south] {$F_9$};
\node at (s13.north) [anchor=south] {$F_{13}$};
\node at (s12.north) [anchor=south] {$F_{12}$};
\node at (s7.north) [anchor=south] {$F_7$};
\node at (a2.north) [anchor=south] {$S_2$};
\node (t6) at (20,-35) [fill,circle,inner sep=0.5mm] {};
\node (t14) at (40,-35) [fill,circle,inner sep=0.5mm] {};
\node (t11) at (60,-35) [fill,circle,inner sep=0.5mm] {};
\node (f8) at (80,-35) [fill,circle,inner sep=0.5mm] {};
\node at (t6.north) [anchor=south] {$F_6$};
\node at (t14.north) [anchor=south] {$F_{14}$};
\node at (t11.north) [anchor=south] {$F_{11}$};
\node at (f8.north) [anchor=south] {$F_8$};
\draw[color=turquoise,very thick] (fffb) to (s5);
\draw[color=turquoise,very thick] (C3) to (s5);
\draw[color=lightbrown,very thick] (C3) to (s13);
\draw[color=lightbrown,very thick] (B3) to (s13);
\draw[color=black,very thick] (fffb) to (s2);
\draw[color=black,very thick] (fffc) to (s2);
\draw[color=darkblue,very thick] (fffc) to (s4);
\draw[color=darkblue,very thick] (C3) to (s4);
\draw[color=caribbeangreen,very thick] (C3) to (s9);
\draw[color=caribbeangreen,very thick] (T3) to (s9);
\draw[color=blue-green,very thick] (T3) to (s12);
\draw[color=blue-green,very thick] (B3) to (s12);
\draw[color=pinkish,very thick] (C3) to (s7);
\draw[color=pinkish,very thick] (P3) to (s7);
\draw[color=brightcerulean,very thick] (fffb) to (s1);
\draw[color=brightcerulean, very thick] (B3) to (s1);
\draw[color=lightgreen, very thick] (fffb) to (a2);
\draw[color=lightgreen, very thick] (P3) to (a2);
\draw[color=carmine,very thick] (s2) to (t6);
\draw[color=carmine,very thick] (s4) to (t6);
\draw[color=carmine,very thick] (s5) to (t6);
\draw[color=ceil,very thick] (s5) to (f8);
\draw[color=ceil,very thick] (a2) to (f8);
\draw[color=ceil,very thick] (s7) to (f8);
\draw[color=darklavender,very thick] (s9) to (t11);
\draw[color=darklavender,very thick] (s12) to (t11);
\draw[color=darklavender,very thick] (s13) to (t11);
\draw[color=darktangerine,very thick] (s5) to (t14);
\draw[color=darktangerine,very thick] (s13) to (t14);
\draw[color=darktangerine,very thick] (s1) to (t14);
\end{tikzpicture}
\end{split}
\end{equation}
The open faces $F_{2},...,F_{14}$ are defined and studied below.
\begin{enumerate}
\item The face $F_1:= \Int(\fffc)$ of codimension $1$ (see Fig. \ref{sptriplu}) will be mapped by $\pi^3_{(12)}$  on the interior of the cusp front face $\ffc \subset \Xd_{(12)}$ (see Fig. \ref{doublesp}). The blow-up of $\{ t=x_1=x_2=x_3=0 \}$ introduces the projective coordinates $\lp t, \xi_1:=\tfrac{x_1}{t},  \xi_2:=\tfrac{x_2}{t},  \xi_3:=\tfrac{x_3}{t} \rp$. Set $u:=\xi_1-\xi_2$. The blow-up of $\{ t=0, \xi_1=\xi_2=\xi_3 \}$ produces near $F_1$ the coordinates $\lp t, \xi_1, \frac{u}{t}, \frac{\xi_1-\xi_3}{t} \rp$. Using the local coordinates near $\Int(\ffc)$ introduced in case $(3)$ from the proof of Proposition \ref{bfibspdublu}, the map $\pi_{(12)}$ is given in a neighborhood of $F_1$ by
\[  \lp {\bf t}, \xi_1, \frac{u}{t}, \frac{\xi_3- \xi_1}{t}   \rp \longmapsto \lp { \bf t}, \xi_1, \frac{u}{t} \rp ,  \]
where, again, we write in bold all the $x$-variables (here, the $t$-variable in LHS is a boundary defining function for $F_1$, while the $t$-variable in the RHS is a boundary defining function for $\ffc$). Clearly, $\pi_{(12)}$ extends on the interior of $F_1$ as a \emph{b}-fibration.

\item Let $F_2 := \Int( \fffc \cap \fffb)$, and let $v:=\xi_1-\xi_3$. Then suitable coordinates on the triple space in a neighborhood of $F_2$ are $\lp   \frac{t}{u} , \xi_1,u,  \nu:= \frac{v}{u} \rp$. Using the local coordinates in case $(4)$ from the proof of Proposition \ref{bfibspdublu} which are suitable near $\ffc \cap \ffb$, the map $\pi^3_{(12)}$ extends along $F_2$ as a \emph{b}-fibration by
\[\lp   \bm{\frac{t}{u}}   , \xi_1, \bm{u} ,   \frac{v}{u} \rp \longmapsto  \lp  \bm{\frac{t}{u}}, \xi_1, {\bf u} \rp.  \]
The $\frac{t}{u}$ variable in LHS is a boundary defining function for $\fffb$, while the one in the RHS is a boundary defining function for $\ffb$. Moreover, $u$ in the LHS defines $\fffc$, while $u$ in the RHS defines $\ffc$. 

Now we consider the faces along the ``vertical" direction in Fig. \ref{sptriplu}, i.e. those from diagram \eqref{intersectiitriplu}.

\item Consider the codimension $1$ face $F_3:=\Int(\Ct)$.  After the two blow-ups which produce the faces $\fffb$ and $\fffc$, we have the following local coordinates: $\lp T:=\frac{t}{v}, \xi_1, \mu:=\frac{u}{v}, v \rp$, where $T$ is a boundary defining function for $\fffb$ and $v$ is a boundary defining function for $\fffc$. To produce the face $\Ct$, we blow-up $\{ T=0, \mu=0  \}$, thus suitable coordinates on a neighborhood of $F_3$ are given by $(T, \xi_1, \frac{\mu}{T}, v  )$, where $T$ is now a boundary defining function for $\Ct$. Then the map $\pi_{(12)}$ looks in a neighborhood of  $F_3$ like
\begin{equation}\label{ecb3}
    \lp \bm{T}, \xi_1, \frac{\mu}{T} , v    \rp \longmapsto \lp \bm{t}, \xi_1, \frac{u}{t} \rp,
\end{equation}
and extends along $F_3$ by mapping it to $\Int (\ffc)$.
Notice that
\begin{equation}\label{not}
\begin{aligned}
\bm{t}= v \bm{T}, &&  \frac{u}{t}=\frac{\mu v}{Tv} = \frac{\mu}{T},
\end{aligned}
\end{equation}
thus the map $\pi^3_{(12)}$ extends as a \emph{b}-fibration near $F_3$. 

\item If $F_4:= \Int(\Ct \cap \fffc)$, we use the same coordinates as in \eqref{ecb3}, the only difference is that $v$ is of $x$-type defining $\fffc$:
\begin{equation}
    \lp \bm{T}, \xi_1, \frac{\mu}{T} , \bm{v}    \rp \longmapsto \lp \bm{t}, \xi_1, \frac{u}{t} \rp,
\end{equation}
thus $\pi_{(12)}$ extends along $F_4$ as a \emph{b}-fibration.

\item Let $F_5 := \Int (\Ct \cap \fffb)$. As in the above case $(3)$, after the first two blow-ups which produce the faces $\fffb$ and $ \fffc$, we work in the coordinates $\lp  T, \xi_1, \mu, v \rp$. In order to produce $\Ct$, we blow up $\{  T=0, \mu=0 \}$, and $\pi_{(12)}$ maps a neighborhood of $F_5$ to a neighborhood of $\ffb \cap \ffc$ as
\[       \lp \bm{ \frac{T}{\mu}}, \xi_1, \bm {\mu} , v    \rp \longmapsto \lp \bm{ \frac{t}{u} }, \xi_1, \bm{u} \rp .    \]
It clearly extends to the map $\pi^3_{(12)}$ along $F_5$, and by using \eqref{not} and the fact that $\bm{u}=\bm{\mu} v$, we deduce that $\pi^3_{(12)}$ is a \emph{b}-fibration near $F_5$.

\item Furthermore, if $F_6 := \Int (\Ct \cap \fffb \cap \fffc)$, then the coordinates from step $5)$ also work, the only difference is that $v$ is also of $x$-type in the LHS.

\item Let $F_7 := \Int (\Ct \cap \Pt)$. The blow-up which produces $\fffb$ introduces the projective coordinates $\lp \tau':= \frac{t}{x_2}, \mu_1:= \frac{x_1}{x_2},x_2, \mu_3:= \frac{x_3}{x_2}  \rp$. In order to obtain $\Pt$, we need to blow-up $\{ \tau'=0, \mu_3=0 \}$, and let us work in the coordinates $\lp \tau', \mu_1, x_2, \frac{\mu_3}{\tau'} \rp$, where $\tau'$ defines $\Pt$, and $x_2$ is a boundary defining function for $\fffb$. Now we blow-up $\{ \mu_1=1, x_2=0 \}$ to get $\Ct$, and the projective coordinates near $F_7$ are given by $\lp  \tau', \frac{\mu_1-1}{x_2}, x_2, \frac{\mu_3}{\tau'}  \rp$, where now $x_2$ defines $\Ct$. 

For the RHS, we use similar projective coordinates as in the case $(7)$ of Proposition \ref{bfibspdublu}. More precisely, we use the coordinates $(\tau', \mu_1, x_2)$ near $\ffb$, and then we blow-up $\{ x_2=0, \mu_1=1 \}$ to get $\ffc$. Then $\pi_{(12)}$ maps a neighborhood of $F_7$ to a neighborhood of $\ffc \cap \tb$ as
\[ \lp \bm{ \tau'}, \frac{\mu_1-1}{x_2}, \bm{x_2}, \frac{\mu_3}{\tau'}  \rp \longmapsto \lp \bm{\tau'}, \frac{\mu_1-1}{x_2}, \bm{x_2} \rp , \]
and it clearly extends as a b-fibration along $F_7$.

\item Furthermore, if $F_8 := \Int (\Ct \cap \Pt \cap \fffb)$, changing the coordinates from the previous step accordingly, we see that $\pi_{(12)}$ maps a neighborhood of $F_8$ to $\Int \lp \ffc \cap \tb \cap \fffb \rp$ as
\[ \lp \bm{ \tau'}, \bm{\mu_1-1}, \bm{\frac{x_2}{\mu_1-1}}, \frac{\mu_3}{\tau'}  \rp \longmapsto \lp \bm{\tau'}, \bm{\mu_1-1}, \bm{\frac{x_2}{\mu_1-1}} \rp, \]
and it extends along $F_8$ as a \emph{b}-fibration.

\item[${(8')}$] Let us see a computation corresponding to a face from $X_{\sbm}^3$. If $S_{8'}:=\Int ( \Pt \cap \fffb)$, then we can use the same coordinates as in case $(8)$, the only difference being that now the variables $\mu_1-1$ in the LHS and in the RHS are both of $y$-type.

\item Let $F_9 := \Int ( \Tt \cap \Ct)$. After the blow-up which produces $\fffb$, we use the projective coordinates $\lp t':=\frac{t}{x_3}, x_1':=\frac{x_1}{x_3}, x_2':= \frac{x_2}{x_3}, x_3  \rp$. To obtain $\Bt$, we need to blow-up $\{ t'=0, x_1'=0, x_2' =0 \}$. We get the projective coordinates $\lp \tau:= \frac{t'}{x_1'}, x_1', \eta_2:= \frac{x_2'}{x_1'}, x_3 \rp$, where $x_1'$ defines $\Bt$, and $x_3$ defines $\fffb$. Now we blow-up $\{ x_3=0, \eta_2=1 \}$ in order to obtain $\Ct$, and we use the projective coordinates $\lp \tau, x_1', \eta:=\frac{\eta_2-1}{x_3},x_3 \rp$, where $x_1'$ defines $\Bt$, and now $x_3$ defines $\Ct$. Finally, we blow-up $\{ x_1'=0, \eta=0 \}$ to obtain the boundary hypersurface $\Tt$, and $x_1'$ is a boundary defining function for $\Tt$, while $x_3$ defines $\Ct$. Then using the coordinates on $\ffc$ given in the case $5)$ of Proposition \ref{bfibspdublu}, $\pi_{(12)}$ looks in a neighborhood of $F_9$ as
\[  \lp \tau, \bm{x_1'}, \frac{\eta}{x_1'} , \bm{ x_3} \rp \longmapsto \lp  \tau, \bm{x_1=x_1' x_3}, \frac{\eta_2-1}{x_1}   \rp. \]
Notice that $\frac{\eta_2-1}{x_1}=\frac{\frac{x_2}{x_1}-1}{x_1}=\frac{\eta}{x_1'}$, which completes this case.

\item Furthermore, if $F_{10}:= \Int ( \Tt )$, the same coordinates as in the previous case work, the only difference is that the $x_3$ variable in the LHS is of $y$-type now. 

\item If $F_{11} := \Tt \cap \Ct \cap \Bt$, using similar coordinates as for $F_9$, $\pi_{(12)}$ is maps a neighborhood of $F_{11}$ to a neighborhood of $\Int \lp  \ffb \cap \ffc  \rp$ as
\[  \lp \tau, \bm{\frac{x_1'}{\eta}},\bm{\eta},\bm{ x_3} \rp \longmapsto \lp \tau, \bm{\frac{x_1}{\eta_2-1}=\frac{x_1'}{\eta}}, \bm{\eta_2-1 = \eta x_3} \rp ,  \]
which clearly extends along $F_{11}$ as a \emph{b}-fibration.

\item If $F_{12} := \Tt \cap \Bt$, then we use the same coordinates as in the previous case, the only difference is that the $x_3$ variable in the LHS is now of $y$-type. 

\item If $F_{13} := \Int(\Bt \cap \Ct)$, then we use the similar projective coordinates $\lp \tau, x_1', \eta_2-1, \frac{x_3}{\eta_2-1} \rp$ as those obtained in the above case $(9)$ near $\Ct$. Then $\pi_{(12)}$ maps a neighborhood of $F_{13}$ to a neighborhood of $\ffb \cap \ffc$ as
\[ \lp \tau, \bm{x_1'}, \bm{\eta_2-1}, \frac{x_3}{\eta_2-1} \rp \longmapsto \lp \tau, \bm{\frac{x_1}{\eta_2-1}} , \bm{\eta_2-1} \rp \in \Int (\ffb \cap \ffc).   \]
Since $\frac{x_1}{\eta_2-1} =x_1'   \frac{x_3}{\eta_2-1}$, the extension as a \emph{b}-fibration follows.

\item Furthermore, if $F_{14}:= \Int(\Bt \cap \Ct \cap \fffb)$, then the same coordinates as in the previous case $(13)$ work, the only difference is that the $\frac{x_3}{\eta_2-1}$ is of $x$-type because it defines $\fffb$.

Let us now consider the faces along the ``right" direction in Fig. \ref{sptriplu}, i.e., $\fffb$, $\fffc$, $\Cu$, $\Tu$, $\Bu$, $\Pu$. Their intersection structure is encoded in an identical diagram as in \eqref{intersectiitriplu} with different names of the faces.

\item Let $F_{15} := \Int (\Cu)$. After the blow-up which produces the face $\fffb$, we have the projective coordinates $(t, \xi_1,\xi_2,\xi_3)$ which change linearly to $(t,\xi_1,\xi_2,w:=\xi_3-\xi_2)$. To obtain the face $\Cu$, we need to blow-up $\{ t=0, w=0 \}$, thus $\pi_{(12)}$ maps a neighborhood of $F_{15}$ to a neighborhood of $\ffb$ as
\[  \lp \bm{t}, \xi_1, \xi_2 , \frac{w}{t}    \rp \longmapsto \lp \bm{t}, \xi_1, \xi_2 \rp, \]
where the $t$ variable in the LHS is a boundary defining function for $\Cu$, while the $t$ variable in the RHS defines $\fffb$.

\item Consider $F_{16} := \Int (\Cu \cap \fffb)$. Let $\lp t, u=\xi_1-\xi_2, \xi_2, w=\xi_3-\xi_2  \rp$ be coordinates after the blow-up of $\fffb$. To produce the face $\Cu$, we need to blow-up $\{ t=0, v'=0  \}$, and then $\pi_{(12)}$ maps a neighborhood of $F_{16}$ to a neighborhood of $\fffb$ as
\[  \lp \bm{ \frac{t}{w}}, u, \xi_2 ,  \bm{w}    \rp \longmapsto \lp \bm{t}, u , \xi_2 \rp ,\]
where the variables $\frac{t}{w}$ and $w$ in the LHS define $\fffb$ and $\Cu$, respectively, while the $t$ variable in the RHS defines $\ffb$. This map clearly extends along $F_{16}$ as a b-fibration.

\item Let $F_{17}:= \Int (\Cu \cap \fffc)$. As in the previous case $(16)$, we work in the projective coordinates $\lp t, u, \xi_2, w \rp$ along $\fffb$. We blow-up $\{ t=0, u=0, w=0 \}$ to produce $\fffc$ and get the projective coordinates $\lp T':=\tfrac{t}{u}, u, \xi_2, \nu':=\tfrac{w}{u} \rp$, where $T'$ defines $\fffb$, and $u$ is a boundary defining function for $\fffc$. In order to obtain $\Cu$, we blow-up $\{ T'=0, \nu'=0 \}$ and we use the projective coordinates $ \lp  T', u, \xi_2 , \frac{ \nu'}{T'}    \rp$, where $T'$ defines $\Cu$, and $u$ defines $\fffc$. 

On $\Xd$, we first have the projective coordinates $(t, u, \xi_2)$ along $\ffb$ which is given by the boundary defining function $t$. We blow-up $\{ t=0, u=0 \}$ to get $\ffc$, and the new coordinates are $\lp T'=\tfrac{t}{u}, u, \xi_2 \rp$, where $T'$ defines $\ffb$, and $u$ is a boundary defining function for $\ffc$. Now $\pi_{(12)}$ looks near $F_{17}$ as
\[  \lp \bm{ T'}, \bm{u}, \xi_2 , \frac{ \nu'}{T'}    \rp \longmapsto \lp \bm{T'}, \bm{u} , \xi_2 \rp , \]
and it clearly extends as a \emph{b}-fibration along $F_{17}$.

\item Furthermore, for $F_{18} :=\Int (\Cu \cap \fffc \cap \fffb)$, we use similar coordinates as in the previous case $(17)$, and $\pi_{(12)}$ looks near $F_{18}$ as
\[  \lp \bm{ \frac{T'}{\nu'}}, \bm{u}, \xi_2 , \bm{\nu'}  \rp \longmapsto \lp \bm{T'}, \bm{u} , \xi_2 \rp . \]

\item Let $F_{19}:= \Int (\Cu \cap  \Bu)$. Let $\lp \tau=\tfrac{t}{x_1}, x_1, \eta_2=\tfrac{x_2}{x_1} ,  \eta_3=\tfrac{x_3}{x_1} \rp$ be projective coordinates along $\fffb$. We blow-up $\{ \tau=0, \eta_2=0, \eta_3=0 \}$ to obtain $\Bu$, and we work in the projective coordinates $\lp \tau, x_1, \tfrac{\eta_2}{\tau}, \tfrac{\eta_3}{\tau} \rp$, where $\tau$ is a boundary defining function for $\Bu$, and $x_1$ defines $\fffb$. We change them linearly to $\lp \tau, x_1, \tfrac{\eta_2}{\tau}, N:= \tfrac{\eta_3}{\tau} -  \tfrac{\eta_2}{\tau} \rp$. To get $\Cu$, we need to blow-up $\{ x_1=0, N=0 \}$. Then $\pi_{(12)}$ maps a neighborhood of $F_{19}$ to a neighborhood of $\ffb \cap \Ba$ as
\[    \lp \bm{ \tau }, \bm{x_1}, \frac{\eta_2}{\tau} , \frac{N}{x_1}  \rp \longmapsto \lp \bm{\tau}, \bm{x_1} , \frac{\eta_2}{\tau} \rp ,  \]
where in the RHS we used the coordinates from case $(11)$ in the proof of Proposition \ref{bfibspdublu}.

\item If $F_{20} := \Int (\Cu \cap  \Bu \cap \fffb)$, then using similar projective coordinates as in the previous step, we remark that $\pi_{(12)}$ looks near $F_{18}$ as
\[   \lp \bm{ \tau }, \bm{\frac{x_1}{N}}, \frac{\eta_2}{\tau} , N \rp \longmapsto \lp \bm{\tau}, \bm{x_1} , \frac{\eta_2}{\tau} \rp ,   \]
thus it clearly extends as a \emph{b}-fibration near $F_{20}$.

\item Let $F_{21}:=\Int (\Cu \cap \Tu)$. After the blow-ups which produced the faces $\fffb$, $\Bu$ and $\Cu$, we have the coordinates  $\lp  \tau , x_1, \frac{\eta_2}{\tau} , n:=\frac{N}{x_1}  \rp$, as in the previous case $(19)$. In order to obtain $\Tu$, we need to blow-up $\{ \tau=0, n=0 \}$, thus  $\pi_{(12)}$ looks near $F_{21}$ as
\[    \lp \bm{ \tau }, \bm{x_1}, \frac{\eta_2}{\tau} , \frac{n}{\tau}  \rp \longmapsto \lp \bm{\tau}, \bm{x_1} , \frac{\eta_2}{\tau} \rp ,   \]
which extends as a \emph{b}-fibration along the face $F_{21}$.

\item Furthermore, if $F_{22}:= \Int ( \Tu )$, the same coordinates as in the previous case $(21)$ work, the only difference is that the $x_1$ variable is of $y$ type now in both LHS and RHS. 

\item If $F_{23} := \Int (\Cu \cap \Tu \cap \Bu)$, using similar coordinates as those introduced in the previous step, we see that $\pi_{(12)}$ maps a neighborhood of $F_{23}$ to a neighborhood of $\ffb \cap \Ba$ as
\[    \lp \bm{ \frac{\tau}{n} }, \bm{x_1}, \frac{\eta_2}{\tau} , \bm{n}  \rp \longmapsto \lp \bm{\tau}, \bm{x_1} , \frac{\eta_2}{\tau} \rp ,   \]
where in the LHS $\tfrac{\tau}{n}$ defines $\Bu$, $x_1$ defines $\Cu$, and $n$ defines $\Tu$.

\item Furthermore, if $F_{24} := \Int ( \Tu \cap \Bu)$, then the coordinates from the previous case $(23)$ also work, the only difference is that $x_1$ is of $y$ type in both LHS and RHS.

\item Let $F_{25} := \Int (\Cu \cap  \Pu)$. After the blow-up which produces $\fffb$, we work in the projective coordinates $\lp \tau'=\tfrac{t}{x_2},\mu_1=\frac{ x_1}{x_2},x_2 ,  \mu_3=\tfrac{x_3}{x_2} \rp$. In order to obtain $\Pu$, we blow-up $\{ \tau'=0, \mu_1=0 \}$, and we use the projective coordinates $\lp \tau',\frac{\mu_1}{\tau'} , x_2, \mu_3 \rp$. Now we blow-up $\{ x_2=0, \mu_3=1 \}$ in order to get the face $\Cu$. We get the coordinates  $\lp  \tau' ,\frac{\mu_1}{\tau'}, x_2 , \frac{\mu_3-1}{x_2}  \rp$, where $\tau'$ is a boundary defining function for $\Pu$, and $x_2$ defines $\Cu$. Then $\pi_{(12)}$ maps a neighborhood of $F_{25}$ to a neighborhood of $\ffb \cap \Bm$ as
\[    \lp \bm{ \tau' },\frac{\mu_1}{\tau'}, \bm{x_2} , \frac{\mu_3-1}{x_2}  \rp \longmapsto \lp \bm{\tau'}, \frac{\mu_1}{\tau'} , \bm{x_2} \rp  ,  \]
where in the RHS we used the coordinates of case $(7)$ in the proof of Proposition \ref{bfibdublusimplu}.

\item Finally, if $F_{26} := \Int (\Cu \cap  \Pu \cap \fffb)$, then using similar coordinates as in the previous step, $\pi_{(12)}$ looks near $F_{26}$ like
\[    \lp \bm{ \tau' },\frac{\mu_1}{\tau'}, \bm{\frac{x_2}{\mu_3-1}} , \bm{\mu_3-1}  \rp \longmapsto \lp \bm{\tau'}, \frac{\mu_1}{\tau'} , \bm{x_2} \rp,  \]  
which clearly extends on $F_{26}$ as a \emph{b}-fibration.   \qedhere
\end{enumerate} 
\end{proof}
By symmetry, the other two maps $\pi^3_{(23)}$ and $\pi^3_{(13)}$ are also \emph{b}-fibrations. 

We remark that the triple space of the (slightly) more complicated case of the $\varphi$-surgery calculus of Albin, Rochon, Sher \cite{ars} in \cite[Lemma C.2]{ars} can be treated essentially in the same way. Their original proof avoids explicit local computations, but relies instead on the exponent matrix computed by Mazzeo-Melrose in \cite{melmaz} for $X^3_{\text{sb}}$, Theorem \ref{comuteblowups} regarding commutations of blow-ups, and a technical lemma of ``lifting" \emph{b}-fibrations due to Hassel-Mazzeo-Melrose \cite{hasmazmel}. The proof of Proposition \ref{bfibtrip} shows by hand that all the exponents $e(G,H)$ from Definition \ref{bmap} appearing in $\pi^3_{(12)}$ belong to $\{ 0 ,1 \}$, as already remarked in \cite{ars} and \cite{melmaz}.

\section{The cusp-surgery calculus}
Let $\Ss$ be a smooth vector bundle over the simple space $\Xs$. In order to define our cusp-surgery calculus, we  first need to fix a convention regarding the densities. Denote by $\Omega(X)$ the $1$-density bundle over $X$. The metric $h_t$ is not defined at $t=0$ along the hypersurface $\gamma$. The family of corresponding densities is smooth outside $\gamma$ and it is constant in $t>0$ in a neighborhood of $\gamma$. In particular, the family extends along $\gamma \times \{ 0 \}$ and forms a smooth section in the pull-back of $\Omega(X)$ over the space $X \times [0,\infty)$. We denote this family density by $\mu_X$.

We will work with operators having a density factor in the second set of variables coming from $X$. More precisely, we will regard our operators as distributional kernels on $\Xd$ with values in the density bundle $\lp  \pi_2 \circ \tilde{\beta}\rp^* \Omega \lp X \rp$, the pull-back of the bundle $\Omega(X)$ from $X_{(2)}$  through the map $\lp  \pi_2 \circ \tilde{\beta}\rp$ as in the following diagram (see Remark \ref{notatie} for the notation $X_{(2)}$):

\begin{center}
\begin{tikzpicture}
[x=1mm,y=1mm]
\node (tl) at (0,15) {$(\pi_2 \circ \tilde{\beta})^*\Omega(X)$};
\node (tr) at (40,15) {$\Omega(X)$};
\node (bl) at (0,0) {$\Xd$};
\node (bm) at (20,0) {$X^2 \times \mathbb{R}_+$};
\node (br) at (40,0) {$X_{(2)}.$};
\draw[->] (bl) to node[above,font=\small]{$\tilde{\beta}$} (bm);
\draw[->] (bm) to node[above,font=\small]{$\pi_2$} (br);
\draw[->] (tr) to (br);
\draw[->,densely dashed] (tl) to (tr);
\draw[->,densely dashed] (tl) to (bl);
\end{tikzpicture}
\end{center}
The operators in the calculus will take values into the sections of a certain vector bundle that we are going to introduce now. More precisely, let us denote by 
\[   \Ss \boxtimes \Ss^*:= \lp \pi^2_{(1)}\rp^* \Ss \otimes \lp \pi^2_{(2)} \rp^* \Ss^*, \]
the tensor product of the pull-back of the bundle $\Ss$ through the map $\pi^2_{(1)}$ and the pull-back of the dual bundle $\Ss^*$ through $\pi^2_{(2)}$ as in the following diagram:
\begin{equation}
\label{eqn:label-bla-bla}
\centering
\begin{split}
\begin{tikzpicture}
[x=1mm,y=1mm]
\node (tl) at (0,15) {$\Ss$};
\node (tm) at (30,15) {$\Ss \boxtimes \Ss^*$};
\node (tr) at (60,15) {$\Ss^*$};
\node (bl) at (0,0) {$\Xs_{(1)}$};
\node (bm) at (30,0) {$X^2_{\cp}$};
\node (br) at (60,0) {$\Xs_{(2)}.$};
\draw[->] (tl) to (bl);
\draw[->] (tm) to (bm);
\draw[->] (tr) to (br);
\draw[->] (bm) to node[above,font=\small]{$\pi^2_{(1)}$} (bl);
\draw[->] (bm) to node[above,font=\small]{$\pi^2_{(2)}$} (br);
\end{tikzpicture}
\end{split}
\end{equation}
Let us consider the diagonal ``plane" $(0, \infty) \times \Diag$ inside $[0, \infty) \times X^2$, where $\Diag:= \{ (p,p): \ p\in X \}$. Furthermore, let $\Delta$ be the closure of the lift through $\tilde{\beta}$ of this ``plane"
\begin{equation}\label{planuldecon}
 \Delta:=  \overline{ \tilde{\beta}^{-1} \lp {\Diag} \times (0, \infty)  \rp }  \subset \Xd. 
\end{equation} 
Remark that $\Delta$ is diffeomorphic to the simple space $\Xs$ (see Fig. \ref{defop}).

\begin{definition}\label{distrconcalc}
Denote by
\[ \II^{m,0,0} (\Xd):= \II^{m,0,0} \lp \Xd, \Delta , \Ss \boxtimes \Ss^* \otimes \lp  \pi_2 \circ \tilde{\beta}\rp^* \Omega \lp X \rp    \rp \]
the set of distributions of order $m$ on the surgery-cusp double space $\Xd$ which
\begin{itemize}
\item[$i)$] are conormal of order $m$ to the ``plane" $\Delta $;
\item[$ii)$] have index sets $\mathbb N$ to the front face $\ffc$ and to the temporal boundary $\tb$ (these are extendible distributions across these faces);
\item[$iii)$] have empty index sets towards all the other boundary hypersurfaces (meaning that they vanish at infinite order there);
\item[$iv)$] contain a density factor in the second set of variables coming from $X_{(2)}$. 
\end{itemize} Furthermore, let
$ \II^{m, \alpha,\beta} \lp \Xd \rp := \rho_{\ffc}^{\alpha} \rho_{\tb}^{\beta}  \II^{m,0,0} (\Xd).  $
\end{definition}
The operators in the \emph{surgery-cusp calculus} are classical pseudodifferential operators on $X$ which depend smoothly on the time-parameter $t$ for $t>0$, and satisfy certain properties towards $\{t=0\}$.

\begin{definition}\label{distrconcalc1}
Let $\Psi^{m,0,0}_{\cp}(X)$ be the set of \emph{cusp-surgery pseudodifferential operators}  corresponding to the distributional kernels in $\II ^{m,-2,0} \lp \Xd \rp $, and furthermore, let
\[\Psi_{\cp}^{m, \alpha, \beta} (X):=\rho_{\ffc}^{- \alpha} \rho_{\tb}^{-\beta} \Psi^{m,0,0}_{\cp}(\Xd) = \II^{m,-(\alpha+2),-\beta}(\Xd). \]
\end{definition}
Note that $\Psi^{*,*,*}_{\cp}(X)$ and $\II^{*,*,*}(\Xd)$ differ by a $-2$ in the second index. This difference appears because of the density factor from the second set of variables.

\begin{figure}[H]
\begin{center}
\includegraphics[width=12cm, height=10.5cm]{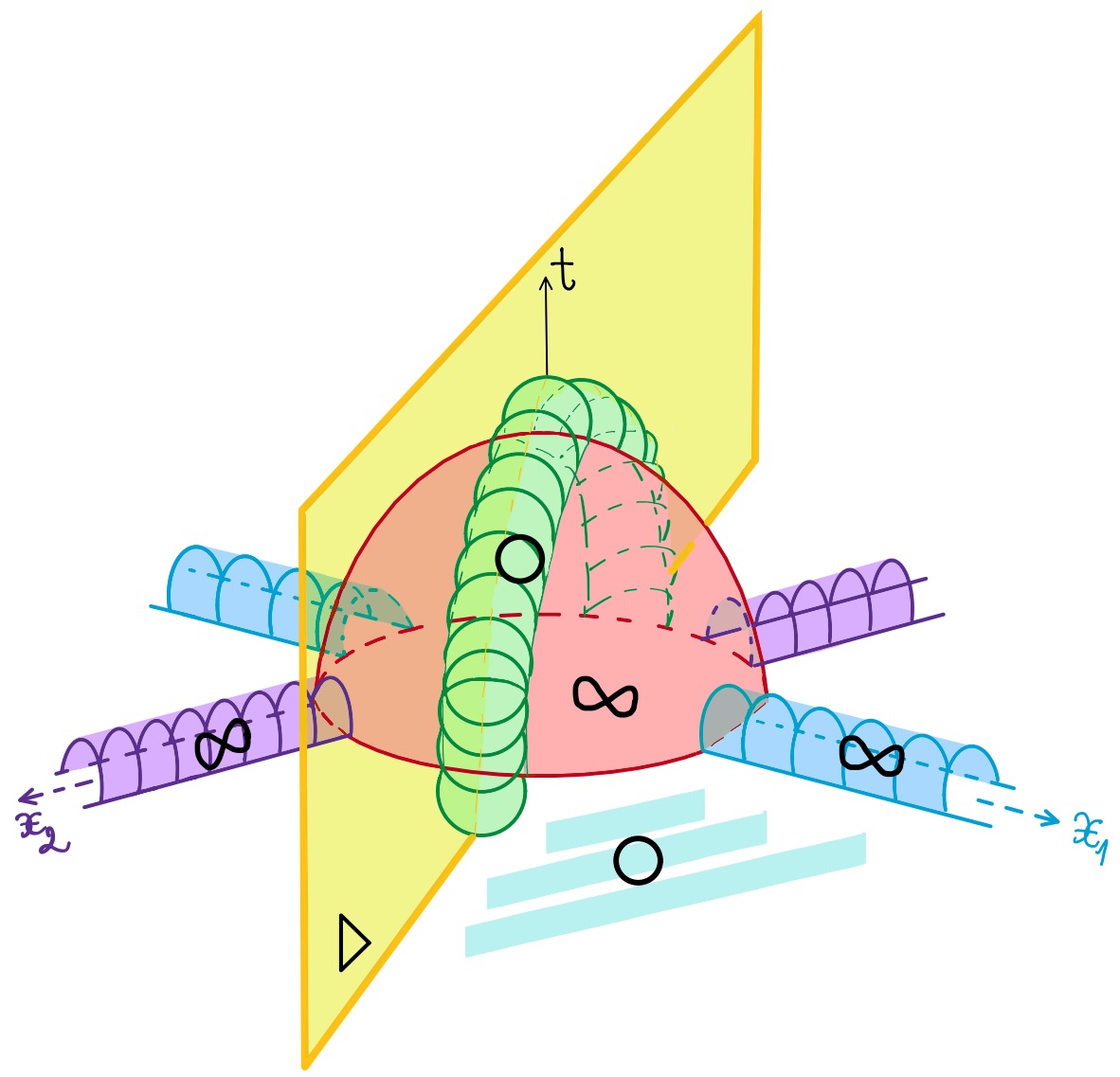}
\caption{Index sets and conormality of the cusp-surgery $\Psi$DO operators}\label{defop}
\end{center}
\end{figure}

We need to define and describe the mapping properties of these operators. Let $A \in \Psi^{m,\alpha,\beta}_{\cp}(X)$  be a cusp-surgery operator and let  $f \in \mathcal A^{\alpha',\beta'} \lp \Xs_{(2)} , \Ss \rp$ be a polyhomogeneous section with values in $\Ss$ (see Remark \ref{notatie}). Notice that we already have in $k_A$ a partial density from $\Xd$. The strategy to apply the operator $A$ on the section $f$ is the following:
\begin{itemize}
\item[$i)$] Take \emph{any} \emph{b}-density $ \tilde{\omega} \in \Omega^b \lp  \Xs_{(1)} \rp$ and pull it back through $\pi^2_{(1)}$ to get a partial density on the double space $\Xd$;
\item[$ii)$] Then $k_A \cdot \lp \pi^2_{(2)} \rp^* f  \cdot  {\pi^2_{(1)}}^*\tilde{\omega}$ is a section in the bundle $\Ss \otimes \Omega^b \lp \Xd \rp$ (see Fig. \ref{aplicare});
\item[$iii)$] We push-forward the \emph{b}-density appearing in $ii)$ through the map ${\pi^2_{(1)}}$ to obtain a section in $\Gamma \lp  \Xs, \Ss \otimes \Omega^b(\Xs) \rp$ (see \eqref{proiectii1}). 
\begin{equation}
\label{proiectii1}
\centering
\begin{split}
\begin{tikzpicture}
[x=1mm,y=1mm]
\node (tl) at (0,15) {$\Xd$};
\node (tm) at (30,15) {$X^2 \times \mathbb{R}_+$};
\node (tr) at (60,15) {$X_{(2)}$};
\node (bl) at (0,0) {$\Xs$};
\node (bm) at (30,0) {$X \times \mathbb{R}_+$.};
\draw[->] (tl) to node[above,font=\small]{$\widetilde{\beta}$} (tm);
\draw[->] (tm) to node[above,font=\small]{$\pi_2$} (tr);
\draw[->] (bl) to node[above,font=\small]{$\beta$} (bm);
\draw[->] (tl) to node[right,font=\small]{$\pi^2_{(1)}$} (bl);
\draw[->] (tm) to node[right,font=\small]{$\pi_{1,+}$} (bm);
\end{tikzpicture}
\end{split}
\end{equation}
\item[$iv)$] The final step is to ``divide" by $\tilde{\omega}$. 
\end{itemize}
It is easy to remark that if we start with another \emph{b}-density $h \tilde{\omega}$, where $h$ is a polyhomogeneous conormal function on $\Xs_{(1)}$, the result of step $iii)$ is $h$ times the former result. Therefore, in the next definition we will take $\tilde{\omega}$ to be the standard \emph{b}-density on $\Xs_{(1)}$, which is $\beta^* \lp \mu_{X_{(1)}} \otimes \tfrac{dt}{t} \rp$, where $\mu_X \in \Omega(X)$ is the density factor given by the volume form of the metric $h_t$, which is independent of $t$.
\begin{definition}\label{apl}
The action of a cusp-surgery pseudodifferential operator $A \in  \Psi^{m, \alpha,\beta}_{\cp}(X)$ on $f \in \mathcal A^{\alpha',\beta'} \lp \Xs_{(2)} , \Ss \rp$ is given by
\[   A(f):=  \frac{{\pi^2_{(1)}}_* \left[   \lp  \pi^2_{(2)} \rp^* f \cdot k_A   \cdot  {\pi^2_{(1)}}^*\tilde{\omega}   \right]    }{\tilde{\omega}} , \]
where $\omega:=\mu_{X_{(1)}} \otimes \tfrac{dt}{t} \in \Omega^b \lp X \times \mathbb R_+ \rp$ and $\tilde{\omega}:=\beta^* \omega \in \Omega^b \lp \Xs_{(1)}  \rp$. 
\end{definition}

\begin{proposition}
Let $A \in  \Psi^{m,\alpha,\beta}_{\cp}(X)$ and let $f \in \mathcal A^{\alpha',\beta'} \lp \Xs , \Ss \rp$. Then 
\[Af \in \mathcal A^{-\alpha+\alpha',-\beta+\beta'}(\Xs, \Ss). \]
\end{proposition}
\begin{proof}
Let us separate the distributional and the density parts in $k_A$ by writing it as
\[  k_A=u \cdot \lp \pi_2  \circ \tilde{\beta} \rp^* \mu_{X_{(2)}},  \]
where $u \in \II^{m,-\alpha-2,-\beta}  \lp \Xd, \Delta , \Ss \otimes \Ss^* \rp$. Remark that using Proposition \ref{ridbdens},
\[  \tilde{\omega}=\beta^* \omega =\beta^* \lp  \mu_{X_{(1)}} \otimes \frac{dt}{t} \rp \in \rho_{\ff} \cdot \Omega^b \lp \Xs \rp , \]
and 
\begin{equation}\label{dens}
\begin{aligned}
 \lp \pi^2_{(1)} \rp^* \tilde{\omega} \otimes k_A \cdot f ={}& u \cdot f \  \tilde{\beta}^*  \lp  \pi_{1,+}^* \lp \mu_{X_{(1)}} \otimes \tfrac{dt}{t} \rp \otimes \pi_2^* \mu_{X_{(2)}}   \rp \\
={}& u \cdot f \  \tilde{\beta}^* \lp  \mu_{X_{(1)}} \otimes \mu_{X_{(2)}} \otimes \tfrac{dt}{t} \rp,
\end{aligned}
\end{equation}
where $\pi_{1,+}$ and $\pi_2$ are the projections from \eqref{proiectii1}. We use again Proposition \ref{ridbdens} to study the density factor in \eqref{dens}. The blow-up of $\{  t=x_1=x_2=0 \}$ in \eqref{dublunerig} (the one which produces $\ffb$ in the double space $\Xd$) is locally modeled by $[\RR^3_1; \{0 \}]$. Hence the \emph{b}-density $  \mu_{X_{(1)}} \otimes \mu_{X_{(2)}} \otimes \tfrac{dt}{t} \in \Omega^b \lp X^2 \times \mathbb R_{+}  \rp$ lifts through this blow-up to a \emph{b}-density multiplied by a factor of $\rho_{\ffb}^2$. 

The last blow-up in \eqref{spdublu} (the one which produces the cusp front face $\ffc$) is of type $[\RR^2_1 ;\{ 0 \}]$, thus when we lift the \emph{b}-density through this one, we get another factor of $\rho_{\ffc}$. Furthermore, the previous factor $\rho_{\ffb}^2$ lifts through this final blow-up to $\rho_{\ffb}^2\rho_{\ffc}^2$. We obtain that
\begin{equation}\label{tildebeta*}
  \tilde{\beta}^* \lp  \mu_{X_{(1)}} \otimes \mu_{X_{(2)}} \otimes \tfrac{dt}{t} \rp \in \rho_{\ffb}^2 \rho_{\ffc}^3 \Omega^b (\Xd), 
\end{equation}  
thus  $\lp \pi^2_{(1)} \rp^* \tilde{\omega} \otimes k_A \cdot f \in \rho_{\ffc}^{-\alpha-2+3 +\alpha'} \rho_{\tb}^{-\beta+\beta'} \Omega^b \lp \Xd  \rp $. 

By the Pull-Back Theorem \ref{pullbackphg}, we can read the index sets of $\lp \pi^2_{(2)} \rp^* f$ (see Fig. \ref{aplicare}).
\begin{figure}[H]
\begin{center}
\includegraphics[width=14.5cm, height=4.5cm]{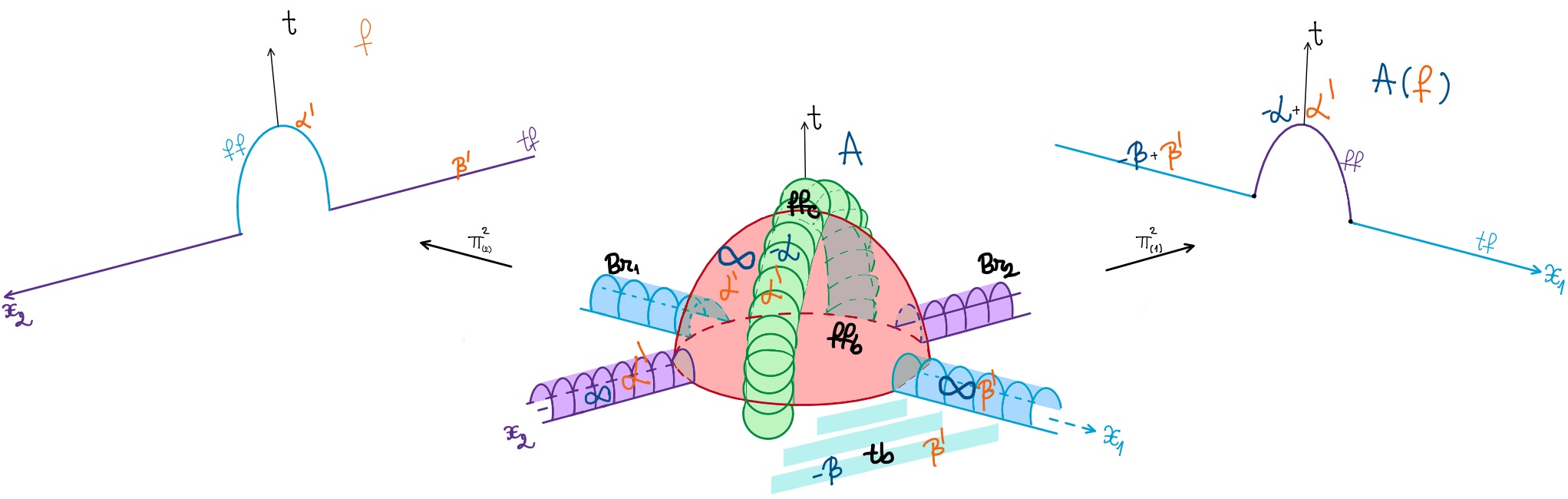}
\caption{Mapping properties of the cusp-surgery pseudodifferential operators}\label{aplicare}
\end{center}
\end{figure}

We use the fact that $\pi^2_{(1)}$ is a \emph{b}-fibration (see Proposition \ref{bfibspdublu}) in order to apply the Push-Forward Theorem \ref{pftc}. Remark that the conormality is integrated away, since it is transversal to the fibers in which we integrate. In order to find the index sets towards the boundary hypersurfaces of $\Xs$, notice that the integrability condition in the Push-Forward Theorem \ref{pft} is empty in our case.

The boundary hypersurfaces $\ffc$, $\ffb$ and $\Bm$ in $\Xd$ project through $\pi^2_{(1)}$ over the front face $\ff$ in $\Xs_{(1)}$ (see Fig. \ref{aplicare}), thus the index set towards $\ff$ is $-\alpha+\alpha'+1$. The boundary hypersurfaces $\Ba$ and $\tb$  project through $\pi^2_{(1)}$ over the temporal face $\tf$ in $\Xs_{(1)}$, thus the index set towards $\tf$ is $-\beta+\beta'$. Therefore
 \[    {\pi^2_{(1)}}_* \left[   \lp \pi^2_{(1)} \rp^* \tilde{\omega} \otimes k_A \cdot f  \right]  \in \rho_{\ff}^{-\alpha+\alpha'+1} \rho_{\tf}^{-\beta+ \beta'} \Omega^b \lp \Xs_{(1)} \rp. \]

By Definition \ref{apl}, it follows that
\[ A(f) \in \rho_{\ff}^{-\alpha+\alpha'} \rho_{\tf}^{-\beta+\beta'} \mathcal A^{0,0} \lp \Xs_{(1)} \rp.    \qedhere  \]
\end{proof}

\section{The cusp-surgery symbol}
Denote by $N^*\Delta$ the conormal bundle to the ``plane" $\Delta$ (see \eqref{planuldecon}), and let $SN^*$ be the sphere bundle in $N^* \Delta$ over $\Xs$. The principal symbol map for conormal distributions induces a surjective homomorphism 
\[ \sigma_{\cp}  : \Psi^{m,\alpha,\beta}_{\cp} \lp X \rp \longrightarrow \mathcal C^{\infty} \lp \Xs, SN^*\Delta  \rp, \]
which is called the \emph{principal cusp surgery symbol}. Remark that the kernel of $\lp \sigma_{\cp} \rp_{m}$ is precisely the set of cusp-surgery pseudodifferential operators of order $m-1$, i.e.,
\[  \ker \lp \sigma_{\cp} \rp_{m} =  \Psi^{m-1,\alpha,\beta}_{\cp} (X). \]
\section{The normal operator and the indicial family}
We will define the normal operator of a cusp-surgery pseudodifferential operator $A$ as the ``restriction" of the distributional kernel of $A$ to the cusp front face $\ffc$. In order to do this, we need to explain the restriction of the density bundle. First, we have the following diagram:
\begin{equation}
\label{diagr1}
\centering
\begin{split}
\begin{tikzpicture}
[x=1mm,y=1mm]
\node (tl) at (0,15) {$(\pi_2 \circ \widetilde{\beta} \circ \iota)^*\Omega(X)$};
\node (tr) at (60,15) {$\Omega(X)$};
\node (bl) at (0,0) {$\ffc$};
\node (bm1) at (20,0) {$\Xd$};
\node (bm2) at (40,0) {$X^2 \times \mathbb{R}_+$};
\node (br) at (60,0) {$X_{(2)}$,};
\draw[->,densely dashed] (tl) to (bl);
\draw[->] (tr) to (br);
\draw[<-,>=right hook] (bl) to ($ (bl)!0.5!(bm1) $) node[above,font=\small]{$\mathrm{\iota}$};
\draw[->] ($ (bl)!0.5!(bm1) $) to (bm1);
\draw[->] (bm1) to node[above,font=\small]{$\tilde{\beta}$} (bm2);
\draw[->] (bm2) to node[above,font=\small]{$\pi_2$} (br);
\end{tikzpicture}
\end{split}
\end{equation}
where $\iota$ is the inclusion of the cusp front face into the double space.
\begin{definition}
If $A \in \Psi^{m,\alpha,\beta}_{\cp} \lp X \rp$ is a cusp-surgery pseudodifferential operator, we define \emph{the normal operator} of $A$ as 
\[   \mathcal N_{\alpha} (A):={\lp \rho_{\ffc}^{\alpha} k_A \rp}_{\vert_{\ffc}} ,  \]
where we regard the density factor in $k_A$ as a section in $\lp \pi_2 \circ \tilde{\beta} \circ \iota \rp^* \Omega(X)$ (see \eqref{diagr1}).
\end{definition}

If we use the normal operator in a context in which the order $\alpha$ is clear, then we might simply write $\NN$ instead of $\NN_\alpha$. However, notice that we get another density bundle $p^* \Omega \lp  \mathbb R \times \gamma  \rp$ over the cusp front face $\ffc$ (see \eqref{diagr2}), where $p:[0,\pi] \times \mathbb R \times \gamma^2 \longrightarrow \mathbb R \times \gamma$ is the projection onto the first factor of $\gamma$. Therefore we have two density bundles $ \lp  \pi_2 \circ \tilde{\beta} \rp^* \Omega(X)$ and $\Omega \lp \mathbb R \times \gamma \rp$ over $\ffc$.
\begin{equation}
\label{diagr2}
\centering
\begin{split}
\begin{tikzpicture}
[x=1mm,y=1mm]
\node (tl) at (0,15) {$p^*\Omega(\mathbb{R} \times \gamma)$};
\node (tr) at (50,15) {$\Omega(\mathbb{R} \times \gamma)$};
\node (bl) at (0,0) {$\ffc$};
\node (bm) at (20,0) {$[0,\pi] \times \mathbb{R} \times \gamma^2$};
\node at ($ (bl.east)!0.5!(bm.west) $) {$=$};
\node (br) at (50,0) {$\mathbb{R} \times \gamma$};
\draw[->,densely dashed] (tl) to (bl);
\draw[->] (bm) to node[above,font=\small]{$p$} (br);
\draw[->] (tr) to (br);
\end{tikzpicture}
\end{split}
\end{equation}
One can easily check the following result.
\begin{lemma}\label{densffc}
The densities bundles $ \lp  \pi_2 \circ \tilde{\beta} \rp^* \Omega(X)$ and $\Omega \lp \mathbb R \times \gamma \rp$ over the cusp front face $\ffc$ are related by a square of $\rho_{\ffc}$:
\[\rho_{\ffc}^{-2} \lp  \pi_2 \circ \tilde{\beta} \rp^* \Omega(X) \simeq p^* \Omega \lp \mathbb R \times \gamma \rp.   \]
\end{lemma} 

Remark that the normal operator takes values in a family of operators on $X$. More precisely, if $A \in \Psi^{m,\alpha,\beta}_{\cp} \lp X \rp$ is a surgery-cusp pseudodifferential operator, then $\NN(A)$ is a family 
\[   \NN(A):[0, \pi] \longrightarrow \Psi^m_{\sus} \lp \gamma \rp, \]
where 
 \[ \Psi^m_{\sus}(\gamma):= \II^m \lp \gamma \times \gamma \times \mathbb R_u, \Diag_{\gamma} \times \{ u=0 \} \rp \]
is the set of \emph{suspended operators} on $\gamma$ (see e.g. \cite[Section 4]{melmaz98} or \cite[Section 2]{moroktheory}), i.e., translation-invariant conormal distributions which decrease rapidly as $u>>0$. We denoted by $\Diag_{\gamma}:=\{ (y,y): \ y \in \gamma \}$ the diagonal inside $\gamma^2$.

To simplify the notation, we will denote the maps from $[0, \pi] \longrightarrow \Psi^m_{\sus} \lp \gamma \rp$ by $\Psi_{[0, \pi], \sus} \lp \gamma \rp$, thus we regard the normal operator as
\[ \NN : \Psi^{m, \alpha, \beta}_{\cp} (X) \longrightarrow  \Psi_{[0, \pi], \sus} \lp \gamma \rp. \]
\begin{definition}
The \emph{indicial family} $\mathcal I (A) (\xi)_{\xi \in \mathbb C}$ of an operator $A \in \Psi^{m, \alpha, \beta}_{\cp} \lp X \rp$ is given by the Fourier transform in the fiber of the cusp front face $\ffc$ of the normal operator $\mathcal N (A)$.
\end{definition}

\section{The temporal operator}
Since the surgery-calculus that we are constructing contains a parameter in time, it is useful to consider another ``symbol operator" (as in \cite[Section 5.2]{mcdonald} or \cite[Section 3.5]{ars}) defined by restricting cusp-surgery distributional kernels to the temporal face $\tb$, playing the role of the \emph{surgery normal homomorphism} in \cite{melmaz}. We will restrict the density section from $\Omega \lp  X_{(2)} \rp$ using the pull-back of this bundle over the temporal boundary, as follows:
\begin{equation}
\label{diagr3}
\centering
\begin{split}
\begin{tikzpicture}
[x=1mm,y=1mm]
\node (tl) at (0,15) {$(\pi_2 \circ \widetilde{\beta} \circ \mathrm{j})^*\Omega(X)$};
\node (tr) at (60,15) {$\Omega(X)$};
\node (bl) at (0,0) {$\tb$};
\node (bm1) at (20,0) {$\Xd$};
\node (bm2) at (40,0) {$X^2 \times \mathbb{R}_+$};
\node (br) at (60,0) {$X_{(2)}$,};
\draw[->,densely dashed] (tl) to (bl);
\draw[->] (tr) to (br);
\draw[<-,>=right hook] (bl) to ($ (bl)!0.5!(bm1) $) node[above,font=\small]{$\mathrm{j}$};
\draw[->] ($ (bl)!0.5!(bm1) $) to (bm1);
\draw[->] (bm1) to node[above,font=\small]{$\tilde{\beta}$} (bm2);
\draw[->] (bm2) to node[above,font=\small]{$\pi_2$} (br);
\end{tikzpicture}
\end{split}
\end{equation}
where $\mathrm{j}$ is the inclusion of the temporal boundary hypersurface into the double space.
\begin{definition}
If $A \in \Psi^{m,\alpha,\beta}_{\cp} \lp X \rp$ is a cusp-surgery pseudodifferential operator, we define \emph{the temporal symbol} as
\[   \mathcal T (A):= \lp {\rho^{\beta}_{\tb}} {k_A} \rp_{\vert_{\tb}} \in \Psi^m_{\cc} \lp X \setminus \gamma \rp ,  \]
where we restrict the density factor in $k_A$ to be a section in $\lp \pi_2 \circ \tilde{\beta} \circ \mathrm{j} \rp^* \Omega \lp X \rp$ (see \eqref{diagr3}).
\end{definition}

\section{The Composition Theorem}\label{compunere}
\begin{theorem}\label{thcomp}
Let $A \in \Psi^{m,\alpha,\beta}_{\cp} \lp X \rp$ and $B \in \Psi^{m',\alpha',\beta'}_{\cp} \lp X \rp$ be two cusp-surgery pseudodifferential operators. Then their composition $AB$ is also a cusp-surgery pseudodifferential operator which belongs to $\Psi^{m+m', \alpha+\alpha',\beta+\beta'}_{\cp} \lp X \rp. $
\end{theorem}
\begin{proof}
Melrose's principle for composing pseudodifferential operators is to define the kernel of $AB$ as the push-forward through the map ${\pi^3_{(13)}}  $ of the product of the two pull-backs of the kernels of $A$ and $B$ to the triple space $\Xt$ (see e.g. \cite{melrose}):
\[ k_{AB}={\pi^3_{(13)}}_{*}   \lp  {\pi^3_{(12)}}^{*} k_A \cdot  {\pi^3_{(23)}}^{*} k_B   \rp .\]
We have to be careful with the density factors. Let us summarize the maps involved in the process of the composition in the following diagram:

\begin{equation}
\label{diagrtriplu}
\centering
\begin{split}
\begin{tikzpicture}
[x=1mm,y=1mm]
\node (tl) at (0,15) {$\Xd_{(12)}$};
\node (tm) at (40,15) {$\Xt$};
\node (tr) at (80,15) {$\Xd_{(23)}$};
\node (bl) at (0,0) {$X_{(2)}$};
\node (bm) at (40,0) {$X^3 \times [0,\infty)$};
\node (br) at (80,0) {$X_{(3)}$};
\node (et) at (25,-10) {$\Xd_{(13)}$};
\node (em) at (25,-25) {$\Xs$};
\node (eb) at (25,-40) {$X \times \mathbb{R}_+$.};
\draw[->] (tl) to node[left,font=\small]{$p_2$} (bl);
\draw[->] (tm) to node[right,font=\small]{$\tilde{\tilde{\beta}}$} (bm);
\draw[->] (tr) to node[right,font=\small]{$p_3$} (br);
\draw[->] (tm) to node[above,font=\small]{$\pi^3_{(12)}$} (tl);
\draw[->] (tm) to node[above,font=\small]{$\pi^3_{(23)}$} (tr);
\draw[->] (bm) to node[above,font=\small]{$\pi_2$} (bl);
\draw[->] (bm) to node[above,font=\small]{$\pi_3$} (br);
\draw[->] (et) to node[left,font=\small]{$\pi^2_{(1)}$} (em);
\draw[->] (em) to node[right,font=\small]{$\beta$} (eb);
\draw[->] (tm) to[out=210,in=90] node[left,font=\small]{$\pi^3_{(13)}$} (et);
\draw[->] (bm) to[out=-90,in=30] node[right,font=\small]{$\pi_{1,+}$} (eb);
\end{tikzpicture}
\end{split}
\end{equation}
For simplicity, let us first suppose that  $A,B \in \Psi^{-\infty,0,0}_{\cp} \lp X \rp$, hence
\begin{align*}
k_A=h_A \  p_2^* \mu_{X_{(2)}},
&&
k_B=h_B \  p_3^* \mu_{X_{(3)}},
\end{align*}
where $h_A$ and $h_B$ are smooth functions on the two double spaces $\Xd_{(12)}$ and $\Xd_{(23)}$. The strategy is the following:
\begin{itemize}
\item[$i)$] Consider the pull-backs of the kernels $k_A$ and $k_B$ on the triple space $\Xt$;  \\ Remark that now we have a \emph{partial} density in the variables of $\Xt$;
\item[$ii)$] We pull-back the missing \emph{b}-density $\mu_{X_{(1)}} \otimes \tfrac{dt}{t} \in \Omega^b \lp  X \times \mathbb R_{+} \rp$ through the map $\beta \circ \pi^2_{(1)} \circ \pi^3_{(13)}$ and we tensor it with the previous partial density (see \eqref{diagrtriplu});
\item[$iii)$] We push-forward the result through $\pi^3_{(13)}$ to get a (full) \emph{b}-density $\nu$ on $\Xd_{(13)}$;
\item[$iv)$] We have to ``divide" $\nu$ by $\mu_{X_{(1)}} \otimes \tfrac{dt}{t}$. 
\end{itemize}
Let us study first what happens with the density factors in the process of composition. 

{\bf Step $ \bf{i)}$.} The partial density on $\Xt$ is the following
\begin{equation}\label{d1}
\begin{aligned}
\lp {\pi^3_{(12)}}^* p_2^* \mu_{X_{(2)}} \rp   \otimes \lp {\pi^3_{(23)}}^* p_3^* \mu_{X_{(3)}} \rp ={}&  \lp \pi_2 \circ \tilde{\tilde{\beta}} \rp^* \mu_{X_{(2)}} \otimes  \lp \pi_3 \circ \tilde{\tilde{\beta}} \rp^* \mu_{X_{(3)}}  \\
={}&  \tilde{\tilde{\beta}}^* \lp \pi_2^* \mu_{X_{(2)}}  \otimes  \pi_3^* \mu_{X_{(3)}} \rp. 
\end{aligned}
\end{equation}

{\bf Step $ \bf{ii)}$.} Using the previous step \eqref{d1}, the complete density on $\Xt$ is given by
\begin{equation}\label{d2}
\begin{aligned}
{}&\tilde{\tilde{\beta}}^* \lp \pi_2^* \mu_{X_{(2)}} \otimes   \pi_3^* \mu_{X_{(3)}} \rp \otimes \lp \beta \circ \pi^2_{(1)} \circ \pi^3_{(13)}\rp^* \lp \mu_{X_{(1)}} \otimes \tfrac{dt}{t} \rp \\
={}& \tilde{\tilde{\beta}}^* \lp \pi_2^* \mu_{X_{(2)}} \otimes  \pi_3^* \mu_{X_{(3)}} \rp \otimes \lp \pi_{1,+} \circ \tilde{\tilde{\beta}} \rp^* \lp \mu_{X_{(1)}} \otimes \frac{dt}{t} \rp \\
={}&   \tilde{\tilde{\beta}}^* \lp  \pi_2^* \mu_{X_{(2)}} \otimes  \pi_3^* \mu_{X_{(3)}}  \otimes  \pi_{1,+}^*  \lp \mu_{X_{(1)}} \otimes \frac{dt}{t} \rp      \rp \\
={}& \tilde{\tilde{\beta}}^* \lp \mu_{X_{(1)}} \otimes \mu_{X_{(2)}} \otimes \mu_{X_{(3)}} \otimes \frac{dt}{t}
\rp.
\end{aligned}
\end{equation}
Now we use Proposition \ref{ridbdens} to compute the exponents of the boundary defining functions appearing in the density factor from \eqref{d2}. The blow-up of $\{ t=x_1=x_2=x_3=0 \}$ in the definition of $\Xt$ (see Section \ref{sectiontriple}) produces $\fffb$, and the local model for this procedure is $[ \RR^4_1  ; \{ 0 \}]$. Therefore the b-density $\mu_{X_{(1)}} \otimes \mu_{X_{(2)}} \otimes \mu_{X_{(3)}} \otimes \frac{dt}{t} \in \Omega^b \lp X^3 \times \mathbb R_{+} \rp$ lifts through this blow-up to a \emph{b}-density having a factor of $\rho_{\fffb}^3$. 

The local model for the blow-up which produces $\fffc$ in the triple space $\Xt$ is $[ \RR^4_2 ; \{ 0 \} ]$, thus when we lift a \emph{b}-density through this blow-up, we get a factor of $\rho_{\fffc}^2$. Furthermore, the previous factor $\rho_{\fffb}^3$ lifts to $\rho_{\fffb}^3 \rho_{\fffc}^3$. In total, we obtain that
\begin{equation}\label{bdenspetrip}
\tilde{\tilde{\beta}}^* \lp \mu_{X_{(1)}} \otimes \mu_{X_{(2)}} \otimes \mu_{X_{(3)}} \otimes \frac{dt}{t} \rp \in \rho_{\fffb}^3 \rho_{\fffc}^5 \Omega^b \lp\Xt   \rp .
\end{equation}

{\bf Step $ \bf{iii)}$.} Now using the Push-Forward Theorem \ref{pft}, we get that the push-forward $\nu$ of the density in \eqref{bdenspetrip} to $\Xd_{(13)}$ satisfies
\begin{equation}
\nu:={{\pi^3}_{(13)}}_* \lp \tilde{\tilde{\beta}}^* \lp \mu_{X_{(1)}} \otimes \mu_{X_{(2)}} \otimes \mu_{X_{(3)}} \otimes \frac{dt}{t} \rp \rp \in \rho_{\ffb}^3 \rho_{\ffc}^5 \Omega^b \lp\Xd_{(13)}  \rp.
\end{equation}

{\bf Step $ \bf{iv)}$.} For the final step, let us first remark that, as seen in \eqref{tildebeta*},
\[ \eta:= \tilde{\beta}^*   \lp   \mu_{X_{(1)}} \otimes \mu_{X_{(3)}} \otimes \tfrac{dt}{t} \rp \in \rho_{\ffb}^2 \rho_{\ffc}^3 \Omega^b \lp   \Xd_{(13)}  \rp , \]
and hence we get
$   \frac{\nu}{\eta} \in \rho_{\ffc}^2 \rho_{\ffb} \mu_{X_{(3)}},  $ 
which is the density factor that we wanted to obtain.

Let us now study what happens with the polyhomogeneous conormal functions $h_A$ and $h_B$. In Fig. \ref{sptriplu}, we represented $1/8$ of the temporal boundary of the triple space $\Xt$, and notice that this face is exactly the usual cusp triple space (see \cite[Section 5]{melmaz98}). One can visualize each actual boundary hypersurface of $\Xt$ as a family indexed by $[0, \infty)$ of the boundary hypersurfaces of the usual cusp triple space. For example, $\fffc$ is actually a family of half-spheres $S^2$. 

Using the Pull-Back Theorem \ref{pullbackphg}, the only non-empty index sets of 
$ \lp \pi^3_{(12)} \rp^* h_A \cdot  \lp \pi^3_{(23)} \rp^* h_B   $
occur towards the faces $\fffc$ (where the index set is $0$), and towards the temporal boundary of the triple space (partially-represented in Fig. \ref{sptriplu}). The index sets corresponding to all the other boundary hypersurfaces of $\Xt$ are empty, since either $\lp \pi^3_{(12)} \rp^* h_A$ or $ \lp \pi^3_{(23)} \rp^* h_B$ vanish infinitely fast to each of them.

By Proposition \ref{bfibtrip}, the map $\pi^3_{(13)}$ is a \emph{b}-fibration, and we can apply the Push-Forward Theorem \ref{pft}. Notice that the boundary hypersurfaces $\fffc, \Cd, \Td$ in $\Xt$ are the only ones which project over $\ffc \subset \Xd_{(13)}$ (one can check this from Proposition \ref{bfibtrip}). Hence the index set of
\[  {\pi^3_{(13)}}_* \left[ \  \lp \pi^3_{(12)} \rp^* h_A \cdot  \lp \pi^3_{(23)} \rp^* h_B \ \right]  \]
towards the cusp front face in $\Xd_{(13)}$ is given by
$ 0 \ \overline{\cup} \ \infty \ \overline{\cup} \ \infty = 0$.  
Furthermore, the temporal face of $\Xt$ is the only one projecting over the temporal face of $\Xd$ (and this projection is exactly the usual one from the cusp triple space towards the cusp double space). Therefore the index towards $\tb \subset \Xd_{(13)}$ is $0$. Moreover, we obtain empty index sets towards all the other boundary hypersurfaces in $\Xd_{(13)}$, hence the composition belongs to the calculus
\[  AB \in \Psi^{-\infty, 0, 0}_{\cp}(X). \] 

Notice that if $A$ and $B$ have non-zero index sets towards the respective cusp front faces and temporal boundaries, their behavior is transferred as expected to the composition. More precisely, if $A \in \Psi^{-\infty, \alpha, \beta}_{\cp} (X)$ and $B \in \Psi^{-\infty, \alpha', \beta'}_{\cp} (X)$, then
$AB \in \Psi^{-\infty, \alpha+\alpha', \beta + \beta'}_{\cp}(X). $

Furthermore, if $A$ and $B$ are given by conormal distributions of orders $m$ and $m'$, the statement follows using the Pull-Back and the Push-forward Theorems \ref{pbtc}  and \ref{pftc}. More precisely, $\lp \pi^3_{(12)} \rp^* k_A$ is a distributional density on $\Xt$ conormal to $\lp \pi^3_{(12)} \rp^{-1} \Delta_{(12)}$, while  $\lp \pi^3_{(23)} \rp^* k_B$ is a distribution on $\Xt$ conormal to $\lp \pi^3_{(12)} \rp^{-1} \Delta_{(23)}$. Thus their product is a distribution singular along the union of the double diagonals. Using a result on the push-forward of multiplication of conormal distributions (see \cite[Theorem~B7.20]{epstein91}), we obtain that
\[  {\pi^3_{(13)}}_* \left[ \  \lp \pi^3_{(12)} \rp^* k_A \cdot  \lp \pi^3_{(23)} \rp^* k_B \ \right]  \]
is a conormal distribution of order $m+m'$ to the ``plane" $\Delta_{(13)}$, and that the principal symbol of the composition satisfies 
\[ \sigma_{\cp}(AB)=\sigma_{\cp}(A) \sigma_{\cp}(B). \qedhere \] 
\end{proof}
The proof of the Composition Theorem in the more complicated case of the surgery $\varphi$-calculus can be found in \cite[Appendix C]{ars}.

\section{Composition of normal operators}\label{multiplnormali}
\begin{proposition}\label{compunereopnorm}
Let $A \in \Psi^{m,\alpha,\beta}_{\cp} \lp X \rp$, $B \in \Psi^{m,\alpha',\beta'}_{\cp} \lp X \rp$ be two cusp-surgery operators. Then the normal operator of the composition $AB$ is obtained by the convolution of the two normal operators of $A$ and $B$:
\[ \NN(AB)= \NN(A) \ast \NN(B). \]
\end{proposition}
\begin{proof}
Let us start with two operators $k_A \in \Psi^{m,0,0}_{\cp} \lp X_{(12)} \rp$ and  $k_B \in \Psi^{m,0,0}_{\cp} \lp X_{(23)} \rp$ and let
\begin{align*}
 k_A= \rho_{\ffc_{(12)} }^{-2} h_A \cdot \lp \pi^2  \circ \tilde{\beta} \rp^* \mu_{X_{(2)}},
 &&
 k_B=\rho_{\ffc_{(23)}}^{-2} h_B \cdot \lp \pi^2  \circ \tilde{\beta} \rp^* \mu_{X_{(3)}},
\end{align*}
where $  h_A$ and $ h_B$ are now conormal distributions in $ \II^{m,0,0} \lp  \Xd, \Ss \boxtimes \Ss^*  \rp$ as in Definition \ref{distrconcalc1}. We choose suitable coordinates near all four cusp front faces of $\Xd_{(12)}$, $\Xd_{(23)}$, $\Xd_{(13)}$ and $\Xt$.

Let $(t,x_1,x_2,y_1,y_2)$ be local coordinates near $[0, \infty ) \times \gamma \times \gamma$ on $[0,\infty) \times X^2_{(12)}$. After the first blow-up of $\{ t=x_1=x_2=0\}$, suitable coordinates away from the temporal boundary $\tb$ are $\lp t, \xi_1:=\tfrac{x_1}{t}, \xi_2:=\tfrac{x_2}{t} , y_1, y_2\rp$. Now we blow-up $\{ t=0; \xi_2-\xi_1=0  \}$ and the coordinates
\begin{equation}\label{coord12}
  \lp  t, u:=\frac{\xi_2-\xi_1}{t},\xi_2  \rp   
\end{equation}
together with $(y_1,y_2)$ are suitable on the interior of the cusp front face $\ffc \subset\Xd_{(12)}$ (away from $\ffc \cap \ffb$ and $\ffc \cap \tb$). Remark that $t$ is a boundary defining function for $\ffc$, $u$ is the variable along the $\mathbb R$ fiber of the cusp face, while $\xi_2$ is the variable along $[0, \pi]$ (see Fig. \ref{doublesp}). 

Let $(t,x_1,x_3,y_1,y_3)$ be local coordinates on $[0,\infty) \times X^2_{(13)}$. After the first blow-up of $\{ t=x_1=x_3=0\}$, suitable coordinates away from the temporal boundary $\tb$ are $\lp t, \xi_1, \xi_3:=\tfrac{x_3}{t} ,y_1, y_3 \rp$. In order to obtain the cusp front face, we blow-up $\{ t=0; \xi_3-\xi_1=0  \}$, hence the coordinates
\[  \lp  t, \xi_1,  v:=\frac{\xi_3-\xi_1}{t}  \rp   \]
together with $(y_1,y_3)$ are suitable on the interior of the cusp front face $\ffc \subset \Xd_{(13)}$. Moreover, $t$ is a boundary defining function for $\ffc$, $v$ is the variable along the fiber $\mathbb R$ cusp face, while $\xi_1$ is the variable along $[0, \pi]$. 

Now let $(t,x_2,x_3,y_2,y_3)$ be the coordinates on $[0,\infty) \times X^2_{(23)}$. After the first blow-up of $\{ t=x_2=x_3=0\}$, suitable coordinates away from the temporal boundary $\tb$ are $\lp t, \xi_2:=\tfrac{x_2}{t}, \xi_3 ,y_2,y_3 \rp$. Now we blow-up $\{ t=0; \xi_3-\xi_2=0  \}$ and the coordinates
\begin{equation}\label{cord}
 \lp  t, \xi_2, w:=\frac{\xi_3-\xi_2}{t}=v-u  \rp   
\end{equation} 
together with $(y_2,y_3)$ are suitable on the interior of the cusp front face $\ffc \subset \Xd_{(23)}$ and furthermore, $t$ is a boundary defining function for $\ffc$, $w$ is the variable along the fiber $\mathbb R$ cusp face, while $\xi_2$ is the variable along $[0, \pi]$. 

Using the results in Section \ref{compunere}, the only boundary hypersurface which projects onto the $\ffc$ in $\Xd_{(13)}$ is the cusp front face $\fffc$ in the triple space $\Xt$, thus we have to carefully choose coordinates near this face. If $(t,x_1,x_2,x_3,y_1,y_2,y_3)$ are the initial local coordinates on $[0,\infty) \times X^3$ near $\gamma^3$, after the first blow-up of $\{ t=x_1=x_2=x_3=0\}$, suitable coordinates near $\fffb$ are $\lp t,\xi_1, \xi_2, \xi_3, y_1,y_2,y_3 \rp$. Then we blow-up $\{ t=0; \xi_1=\xi_2=\xi_3  \}$ and the functions
\[  \lp t, u= \frac{\xi_2-\xi_1}{t} , \xi_2, v=\frac{\xi_3-\xi_1}{t}  \rp   \]
together with $(y_1,y_2, y_3)$ form a system of coordinates on the interior of the triple cusp front face $\fffc \subset \Xt$. Using the Composition Theorem \ref{thcomp} and Lemma \ref{densffc}, it follows that
\begin{equation}\label{comptemp}
\begin{aligned}
 h_{AB} \lp t, \xi_1, v  , y_1,y_3 \rp  dv dy_3 ={}& \int_{(u,y_2) \in \mathbb R \times \gamma} h_A \lp t,u,\xi_2 , y_1, y_2 \rp   du dy_2  \cdot  h_B \lp   t, \xi_2, w, y_2, y_3 \rp dw dy_3   \\
 ={}& \int_{(u,y_2) \in \mathbb R \times \gamma} h_A \lp t,u,\xi_2 y_1, y_2\rp     \cdot  h_B \lp   t, \xi_2, v-u, y_2, y_3 \rp  du dy_2 dv dy_3.
 \end{aligned}
\end{equation}

Since we are interested in the normal operators, we need to restrict equation \eqref{comptemp} to the cusp front faces, which are all given in the chosen coordinates by $t=0$. Using \eqref{coord12}, let us first remark that
$     \xi_1=  \xi_2-ut, $
thus ${\xi_1}_{\vert_{t=0}}=\xi_2$, and then equation \eqref{comptemp} implies at $\{ t=0 \}$ that
\[  \NN(AB) \lp \xi_2,v, y_1, y_3 \rp =   \int_{(u,y_2) \in \mathbb R \times \gamma} h_A \lp u,\xi_2 , y_1, y_2 \rp     \cdot  h_B \lp    \xi_2, v-u, y_2, y_3 \rp  du dy_2 dv dy_3,  \]
hence
\begin{equation}\label{compnorm}
  \NN(AB) (\cdot, \xi_2, y_1, y_3) = \NN(A) (\cdot , \xi_2, y_1, y_2) \ast \NN(B) (\cdot, \xi_2, y_2, y_3).  
\end{equation}
Remark that if we start with the cusp-surgery operators $A \in \Psi^{m,\alpha,\beta}_{\cp} (X)$,  $B \in \Psi^{m',\alpha',\beta'}_{\cp} (X)$, due to the Composition Theorem \ref{compunere}, the operator $AB$ contains a factor of $\rho_{\ffc}^{-(\alpha+\beta)}$, thus \eqref{compnorm} remains true. Therefore, once we fix the angle $\xi_2 \in [0, \pi]$, the normal operator of the composition of two cusp-surgery operators $A,B$ restricted to the $\RR$-fiber at angle $\xi_2$ is the convolution of the corresponding normal operators $\NN(A)$ and $\NN(B)$ in the corresponding $\RR$-fibers of the same angle. 

By continuity, this also holds true at the intersection of the temporal boundary $\tb$ with the cusp front face $\ffc$, even though the coordinates used above are not suitable on the temporal boundary. We will see another computational argument for this statement in the next section. 

As a consequence of this section, we also remark that the indicial family behaves well under composition, being multiplicative. 
\end{proof}

\section{Composition of temporal operators}\label{multipltemp}
\begin{proposition}\label{compunereoptemp}
Let $A \in \Psi^{m,\alpha,\beta}_{\cp} \lp X \rp$ and $B \in \Psi^{m,\alpha',\beta'}_{\cp} \lp X \rp$. Then the temporal operator of the composition 
$\Te (AB) \in \Psi^{m+m'}_{\cc}(X)$
is obtained as the composition of the cusps operators $\Te(A) \in \Psi^m_{\cc}(X)$ and $\Te(B) \in \Psi^{m'}_{\cc}(X)$.
\end{proposition}

\begin{proof}
Let  us start with the simpler case when $k_A \in \Psi^{m,0,0}_{\cp} \lp X_{(12)} \rp$. Denote again by
\begin{align*}
 k_A= \rho_{\ffc_{(12)} }^{-2} h_A \cdot \lp \pi^2  \circ \tilde{\beta} \rp^* \mu_{X_{(2)}},
 &&
 k_B=\rho_{\ffc_{(23)}}^{-2} h_B \cdot \lp \pi^2  \circ \tilde{\beta} \rp^* \mu_{X_{(3)}},
\end{align*}
where $ \rho_{\ffc_{(12)} }^{-2} h_A$ and $\rho_{\ffc_{(23)}}^{-2} h_B$ are now conormal distributions on $\Xd$, as in Definition \ref{distrconcalc}. We describe suitable coordinates near all the four temporal faces of $\Xd_{(12)}$, $\Xd_{(23)}$, $\Xd_{(13)}$ and $\Xt$.

Let $(t,x_1,x_2,y_1,y_2)$ be local coordinates near $[0, \infty) \times \gamma^2$ in $[0,\infty) \times X^2_{(12)}$. After the first blow-up of $\{ t=x_1=x_2=0\}$, the coordinates $\lp \tau:=\tfrac{t}{x_1}, x_1, \eta_2:=\tfrac{x_2}{x_1} , y_1, y_2 \rp$ are suitable near the temporal boundary $\tb$. Now we blow-up $\{ x_1=0; \eta_2=1  \}$ and the coordinates
\begin{equation}\
  \lp  \tau, x_1, u:=\frac{\eta_2-1}{x_1}  \rp   
\end{equation}
together with $y_1, y_2$ are suitable near $\tb$, away from $\{ x_1=0\}$ and $\ffb$ inside $\Xd_{(12)}$. Remark that $x_1$ is a boundary defining function for $\ffc$, $u$ is the variable along the $\mathbb R$ fiber of the cusp face, while $\tau$ is the variable along $[0, \pi]$. 

Let $(t,x_1,x_3, y_1, y_3)$ be the coordinates near $[0, \infty ) \times \gamma^2$ in $[0,\infty) \times X^2_{(13)}$. After the first blow-up of $\{ t=x_1=x_3=0\}$, we choose the projective coordinates $\lp \tau, x_1, \eta_3:=\tfrac{x_3}{x_1} ,y_1, y_3 \rp$. In order to obtain the cusp front face, we blow-up $\{ x_1=0; \eta_3=1  \}$, hence the coordinates
\[  \lp  \tau , x_1,  v:=\frac{\eta_3-1}{x_1}  \rp   \]
together with $y_1, y_3$ are suitable on $ \Xd_{(13)}$, away from $\{ x_1=0 \}$ and $\ffb$. Moreover, $x_1$ is a boundary defining function for $\ffc$, $v$ is the variable along the fiber $\mathbb R$ in the cusp face, while $\tau$ is the variable along $[0, \pi]$.

Now let $(t,x_2,x_3, y_2, y_3)$ be the coordinates near $[0,\infty) \times \gamma_{(2)} \times \gamma_{(3)}$. After the first blow-up of $\{ t=x_2=x_3=0\}$, the projective coordinates $\lp \tau'=\tfrac{t}{x_2}, x_2, \mu_3:=\tfrac{x_3}{x_2} ,y_2, y_3 \rp$ are suitable near $\tb$, but away from $\{  x_2=0 \}$ and $\ffb$. Now we blow-up $\{ x_2=0; \mu_3=1  \}$ and the coordinates
\begin{equation}
 \lp  \tau', x_2, w:=\frac{\mu_3-1}{x_2}  \rp   
\end{equation} 
together with $y_2, y_3$ are suitable near $\tb$, but away from $\{ x_2=0  \} $ and $\ffb$. Furthermore, $x_2$ is a boundary defining function for $\ffc$, $w$ is the variable along the fiber $\mathbb R$ cusp face, while $\tau'$ is the variable along $[0, \pi]$. 

Now we carefully choose coordinates in the triple space $\Xt$. If $(t,x_1,x_2,x_3,y_1, y_2, y_3)$ are the initial coordinates near on $[0,\infty) \times \gamma_{(1)} \times \gamma_{(2)} \times \gamma_{(3)}$, after the first blow-up of $\{ t=x_1=x_2=x_3=0\}$, we consider the projective coordinates $\lp \tau,x_1, \eta_2, \eta_3 , y_1, y_2, y_3 \rp$. Then we blow-up $\{ x_1=0; \eta_2=\eta_3=1  \}$ and we use the coordinates
\[  \lp \tau, x_1, u= \frac{\eta_2-1}{x_1} , v=\frac{\eta_3- 1}{x_1}  \rp   \]
together with $y_1, y_2, y_3$.
Using the Composition Theorem \ref{thcomp}, it follows that
\begin{equation}\label{comptemp1}
\begin{aligned}
 h_{AB} \lp \tau, x_1, v  \rp  dv dy_3 ={}& \int_{(u,y_2) \in \mathbb R \times \gamma} h_A \lp \tau,x_1,u \rp   du dy_2  \cdot  h_B \lp   \tau', x_2, w \rp dw dy_3 .
 \end{aligned}
\end{equation}
In order to restrict the equation \eqref{comptemp1} to the temporal boundary $\tb$, we need to set $\tau=\tau'=0$ (or $\tau=\tau'=\pi$), and then we get
\[   \Te (AB) \lp x_1, v  \rp  = \int_{(u,y_2) \in \mathbb R \times \gamma} \Te( A) \lp x_1, u \rp     \cdot  \Te(B) \lp   x_2, w \rp  d u d y_2 \]
which is exactly the composition of cusp operators (see e.g. \cite{melmaz98}). Furthermore, let us remark that
\begin{align*}
x_2=x_1 \eta_2 = x_1 \lp ux_1 +1  \rp, &&
x_3=x_1 \eta_3=x_1 \lp  v x_1 +1 \rp,
\end{align*}
therefore we obtain that
\[  w=\frac{\mu_3-1}{x_2} = \frac{\frac{x_2}{x_3}-1}{x_1 (u x_1 +1)} = \frac{v-u}{(u x_1 +1)^2} , \]
in particular, it follows that $w_{\vert_{x_1=0}}=v-u$. If we restrict again \eqref{comptemp} to the cusp front faces $\ffc$ which are given by $ \{ x_1=0=x_2 \}$, we remark again that the composition of the normal operators is made through a convolution. Remark that the two restrictions (the normal and the temporal operators) are then compatible.

If we start with the cusp-surgery operators $A \in \Psi^{m,\alpha,\beta}_{\cp} (X)$,  $B \in \Psi^{m', \alpha',\beta'}_{\cp} (X)$, due to the Composition Theorem \ref{compunere}, the operator $AB$ contains a factor of $\rho_{\tb}^{-(\beta+\beta')}$, thus the results remains true. Therefore, the temporal operator of the composition of two cusp-surgery operators $A,B$ is simply the composition of the cusp operators $\Te(A)$ and $\Te (B)$.
\end{proof}

We summarize Sections \ref{compunere}, \ref{multiplnormali} and \ref{multipltemp} in the following statement.
\begin{theorem}[The Composition Theorem]\label{compositiontheorem}
Consider two cusp-surgery pseudodifferential operators $A \in \Psi^{m,\alpha,\beta}_{\cp} \lp X \rp$ and $B \in \Psi^{m',\alpha',\beta'}_{\cp} \lp X \rp$. Then the composition
 \[ AB \in \Psi^{m+m', \alpha+\alpha',\beta+\beta'}_{\cp} \lp X \rp, \]
and furthermore, all the three symbols are multiplicative
\begin{align*}
\sigma_{\cp}(AB)=\sigma_{\cp}( A)  \sigma_{\cp} (B), && \NN(AB)=\NN ( A) \ast  \NN (B), && \Te(AB)=\Te ( A)  \circ  \Te (B).
\end{align*} 
\end{theorem}

\section{Spin structures and the family of Dirac operators}
Let us first recall some notions on the spinor bundle over a spin Riemannian manifold following \cite[Chapter 2]{spinorial} and \cite[Chapter 3]{berline}. Consider $V$ to be a real vector space of dimension $n$ equipped with a scalar product $g$. The \emph{Clifford algebra} $\Cl_n(V)$ is the tensor algebra $T(V)= \oplus_{k\geq 0} V^{\otimes k}$ factorised by  the ideal generated by the set $\lbrace v\otimes w+w\otimes v+ 2 g(v,w): v,w \in V  \rbrace$. Furthermore, the \emph{spin group} is the subgroup of $\Cl_n(V)$ generated by the set of even products of unit vectors
\[ \Spin(V):= \langle v_1...v_{2k} : \Vert v_i \Vert=1, \ v_i \in V \rangle   .\]
One can prove that 
\begin{equation}\label{spinso}
\Spin(V) \longrightarrow \SO(n)
\end{equation} 
is a double covering, where $\SO(n) \subset \mathcal M_n (\mathbb R)$ is the special orthogonal group (see e.g. \cite[Proposition 1.20]{spinorial}). Now let $n$ be even, and consider $V=\mathbb R^n$ with the standard basis $\{e_1,...,e_n \}$ endowed with the standard scalar product. Consider $P$ the complex vector space generated by \[ \left\{ \tfrac{1}{\sqrt 2} (e_1-i e_2),..., \tfrac{1}{\sqrt 2} (e_{n-1}-i e_{n}) \right\}. \] 
Then the \emph{spinor module} is the $2^{n/2}$ dimensional complex vector space $\Sigma_n:= \Lambda^* P$. 
The \emph{Clifford representation} $\cl : V \longrightarrow \End \Sigma_n$ is defined in the following way
\[ \cl(v)=\frac{1}{\sqrt 2} \lp v-iJv \rp \wedge  \cdot - \frac{1}{\sqrt 2} \lp v+iJv \rp  \lnot  \ \cdot ,\]
where $J$ is the standard almost complex structure on $\mathbb R^n$ mapping $e_1 \mapsto e_2, \ e_3 \mapsto e_4,...$, and clearly $\cl$ extends to the whole complexified Clifford algebra $\Cl_n (\mathbb R^n) \otimes_{\mathbb R} \mathbb C$. From now on, we denote by $\cl$ the restriction of the Clifford action to the spin group $\Spin(n):=\Spin (\mathbb R^n) \subset \Cl_n (\mathbb R^n)$,
\begin{equation}\label{cliff}
 \cl : \Spin(n) \longrightarrow \End \Sigma_n. 
\end{equation}

A \emph{spin structure} on an oriented Riemannian manifold $(X,g)$ of dimension $n$ is a principal bundle $\Psp X$ with group $\Spin(n)$ which is a $2:1$ covering over the principal frame bundle $\PSOn X$, and this covering is compatible with the group covering in \eqref{spinso}.

We say that $X$ is a \emph{spin manifold} if it admits a spin strucure. Once we fixed a spin structure on an even dimensional spin Riemannian manifold $(X,g)$, we define the \emph{spinor bundle} as 
\[ \Ss:=\Psp X \times_{\Spin(n)} \Sigma_n.  \]
Remark that for even $n$, the spinor bundle admits a splitting 
$\Ss = \Ss^+ \oplus \Ss^{-} ,$
where the \emph{postive} and \emph{negative spinors} are the ones corresponding to the splitting of $\Sigma_n= \Lambda^* P$ into even and odd forms, respectively.

The vector bundle $\Ss$ is endowed with a natural Levi-Civita connection $\nabla$ obtained in the following manner: we lift the Levi-Civita connection of the metric $g$ from the frame bundle $\PSOn X$ to the spin structure $\Psp X$, and then we consider the connection induced on the associated vector bundle $\Ss$.
\begin{definition}
The Dirac operator $\Dd$ of a spin structure $\Psp X$ on a spin Riemannian manifold $(X,g)$ is defined as the composition of the Clifford action and the Levi-Civita connection
\begin{center}
\begin{tikzpicture}
[x=1mm,y=1mm]

\node (t1) []   at (-40,0)   {$\Gamma\left(X,\Ss \right)$};
\node (t2) []   at (0,0)   {$\Gamma\left(X, T^{*}X \otimes \Ss \right)$};
\node (t3) []   at (50,0)   {$\Gamma\left(X,\Ss \right).$};

\draw[->]             (t1) to node[yshift=4mm]{$\nabla$} (t2);
\draw[-> ]             (t2) to node[yshift=4mm]{$\cl$} (t3);
\draw[->, dashed]  [in=-165,out=-15]           (t1) to node[yshift=-4mm]{$\Dd:=\cl \circ \nabla$} (t3);
\end{tikzpicture}
\end{center}
\end{definition}
Locally, if $\lbrace dx_i \rbrace$ is a local basis in $T^*X$,  then $\nabla=\sum_{i}  dx_i \otimes \nabla_{\partial_i}$ and it follows that the Dirac operator is given by $\Dd=\sum_{i} \cl(dx_i)\nabla_{\partial_i}$.

\begin{example}\label{strspins1}
Spin structures are also defined for odd dimensional manifolds. As an example, let $X= S^1$, whose orthonormal frame is identified with $S^1$. There are only two spin structures over the circle, corresponding to the $2:1$ coverings of the orthonormal frame $\{ \partial_y \}$. The trivial spin structure has two connected components $S^1\sqcup S^1$. The non-trivial spin structure is $S^1$, and it covers the orthonormal frame through the map $z \longmapsto z^2$. In both cases, the spinor bundle over $X$ is trivial of rank $1$. The two spin structures induce the corresponding two Dirac operators on $X=S^1$:
\begin{align*}
\Ddd_1 = i \partial_y, && \Ddd_2 = \frac{1}{2} + i \partial_y.
\end{align*}
Notice that the Dirac operator $\Ddd_2$ induced by the non-trivial spin structure is invertible.
\end{example}

Let us now go back to our $2$-dimensional context. We want to study the behavior of the spectrum of the Dirac operator in the pinching process of the surface $(X,g_t)$ along a simple closed curve $\gamma$. The key point is to add the following:
\begin{hypothesis}\label{hypot} 
Let $\Pspd X$ be a \emph{non-trivial} spin structure on $X$ with respect to $\gamma$, i.e., a spin structure on $X$ which restricted to the geodesic is non-trivial, meaning that the induced Dirac operator on the geodesic is invertible (see Example \ref{strspins1}):
\[ \Dd_{\vert \gamma} = \frac{1}{2} + i \partial_y. \]
This is equivalent to saying that the lift $\tilde{p}$ of the orthonormal frame $p$ along the geodesic (formed by the tangent vector $\dot{\gamma}$ and its orthogonal) does not close continuously in $\Pspd \gamma$.

Bär \cite[Theorem 1]{Bar} proved that on a hyperbolic surface of finite area equipped with a spin structure which is non-trivial along all cusps, the spectrum of the Dirac operator is discrete. Moroianu \cite{moroweyl} generalised this result in any dimension for a wider class of metrics conformal to \emph{exact cusp metrics}. In our case, the metric $g_0$ in \eqref{metr} is conformal to a cusp metric as in \cite[Theorem 2]{moroweyl}, ensuring us that the spectrum of the Dirac operator on the limit surface $(X \setminus \gamma, g_0)$ is discrete.
\end{hypothesis}
For a fixed spin structure on $X$ non-trivial along $\gamma$ as above, we have so far a family of Dirac operators $\Dd_t$ corresponding to the metric $g_t$ at any time $t >0$, and we also have the Dirac operator $\Dd_0$ at time $t=0$ on $X \setminus \gamma$. In the next Section \ref{extspin}, we will rigorously describe this family and we will regard it as a cusp-surgery differential operator in our calculus $\Psi_{\cp}^{*,*,*}(X)$ constructed over the surface $X$. 

\section{The extension of the spinor bundle up to the front face}\label{extspin}
We need to define the spinor bundle $\Ss$ on $\Xs$, up to the front face $\ff$, together with a natural Levi-Civita connection. We first extend the family of tangent bundles $(TX)_{t > 0}$ and the cusp-tangent bundle $T^{\cc}X$ at time $\{ t=0 \}$ up to the front face $\ff$ in the simple space $\Xs$. Recall that we have a family of metrics $(g_t)_{t \geq 0}$ on the surface $X \setminus \gamma$ which locally near the curve $\gamma$ has the form
\begin{equation}\label{gtGt}
g_t=\lp x^2+t^2 \rp \lp  \frac{dx^2}{\lp x^2+t^2 \rp^2} + dy^2 \rp =: \lp x^2+t^2 \rp  G_t ,
\end{equation}
where we remark that $G_0=\tfrac{dx^2}{x^4} + dy^2$ is the cusp metric corresponding to $p=1$ in \cite[$(4)$]{moroweyl}. Consider now the vector fields
\begin{align*}
W:=(x^2+t^2) \partial_x,  &&  Y:=\partial_y,
\end{align*}
which form an orthonormal basis near $\gamma$ for the metric $G_t$. Let
\begin{align*}
\rho:=\sqrt{x^2+t^2}, && \theta:=\frac{t}{x}
\end{align*}
be polar coordinates on the simple space $\Xs$, and notice that $\rho$ is a boundary defining function for the front face $\ff$. Let us lift the vector fields $\left\{ W,Y \right\}$ to the simple space through the blow-down map $\beta : \Xs \longrightarrow \mathbb R_{+} \times X$ (see \eqref{beta}). We obtain that
\begin{equation}\label{betaW}
\beta^* W = \beta^* \lp \rho^2 \partial_x \rp = \rho^2 \lp \frac{\partial \rho}{\partial x} \partial_{\rho}  +  \frac{\partial \theta}{\partial x} \partial_{\theta  } \rp= \rho \lp \rho \cos \theta \partial_{\rho } -  \sin \theta \partial_{\theta}  \rp,
\end{equation}
and $\beta^*Y= Y$ (since the blow-up in $\Xs$ does not affect the variables on the geodesic $\gamma$). Motivated by $\eqref{betaW}$, let us consider the following family of vector fields on the simple space
\[ F:=\left\{ V \in \mathcal V \lp \Xs \rp : \ V(t)=0, \ V(\rho) \in \rho^2 \mathcal C^{\infty} \lp \Xs \rp, V (\theta) \in \rho \mathcal C^{\infty} \lp \Xs \rp \right\}, \]
and remark that $F$ forms a Lie algebra. Furthermore, there exists a smooth bundle $\mathcal F$ having as sections the vector fields in $F$. Indeed, we obtain the vector bundle $\mathcal F$ by glueing the following three vector bundles:
\begin{itemize}
\item[$1)$] The family of tangent bundles for $t>0$ over $X \times (0, \infty)$, i.e., the pull-back of $TX$ through the projection on the first factor in the following diagram
\begin{equation}
\centering
\begin{split}
\begin{tikzpicture}
[x=1mm,y=1mm]
\node (tl) at (0,15) {$p^*_1 TX$};
\node (tr) at (30,15) {$TX$};
\node (bl) at (0,0) {$X \times (0,\infty)$};
\node (br) at (30,0) {$X.$};
\draw[->,densely dashed] (tl) to (bl);
\draw[->] (tr) to (br);
\draw[->] (bl) to node[above,font=\small]{$p_1$} (br);
\end{tikzpicture}
\end{split}
\end{equation}
\item[$2)$]
The bundle $\pi_1^* \lp TX_{\vert_{\lp X \setminus \gamma \rp}} \rp$ over $\lp X \setminus \gamma \rp \times [0, \infty)$, i.e., the pull-back of the tangent bundle restricted to the complement of the geodesic $\gamma$ in the surface $X$ through the projection onto the first factor:
\begin{equation}
\centering
\begin{split}
\begin{tikzpicture}
[x=1mm,y=1mm]
\node (tl) at (0,15) {$\pi_1^* TX_{\vert_{\lp X \setminus \gamma \rp}} $};
\node (tr) at (30,15) {$ TX_{\vert_{\lp X \setminus \gamma \rp}} $};
\node (bl) at (0,0) {$\lp X \setminus \gamma \rp \times [0,\infty)$};
\node (br) at (30,0) {$X \setminus \gamma.$};
\draw[->,densely dashed] (tl) to (bl);
\draw[->] (tr) to (br);
\draw[->] (bl) to node[above,font=\small]{$\pi_1$} (br);
\end{tikzpicture}
\end{split}
\end{equation}
Clearly this bundle agrees with the one described in $1)$ on the intersection $(X \setminus \gamma) \times (0, \infty)$.
\item[$3)$]
A trivial bundle $T$ of rank $2$ over a tubular neighborhood $U$ of the front face:
\begin{equation}
\centering
\begin{split}
\begin{tikzpicture}
[x=1mm,y=1mm]
\node (tl) at (0,15) {$\widetilde{\pi}^* \pi^* T\gamma$};
\node (tm) at (30,15) {$\pi^* T\gamma$};
\node (tr) at (60,15) {$T\gamma$};
\node (bl) at (0,0) {$U=\ff \times [0, \epsilon)$};
\node (bm) at (30,0) {$\ff = \gamma \times [0,\pi]$};
\node (br) at (60,0) {$\gamma$,};
\draw[->,densely dashed] (tl) to (bl);
\draw[->,densely dashed] (tm) to (bm);
\draw[->] (tr) to (br);
\draw[->] (bl) to node[above,font=\small]{$\widetilde{\pi}$} (bm);
\draw[->] (bm) to node[above,font=\small]{$\pi$} (br);
\end{tikzpicture}
\end{split}
\end{equation}
where $\pi$ and $\widetilde{\pi}$ are the projections onto the first factor on the corresponding spaces. We glue the trivial bundle $T$ with the two bundles in $1)$ and $2)$ by identifying the two spanning sections in $T$ to the vector fields $\beta^* W$ and $Y$.
\end{itemize}
Therefore we obtain the vector bundle $\mathcal F$ over $\Xs$, which will play the role of the tangent bundle in constructing the spinor bundle up to the front face. Notice that the restriction of $\mathcal F$ to the boundary hypersurfaces $\ff \cup \tf$ is the cusp tangent bundle ${}^c TX$.

Now we introduce a ``global" spin structure $\Pspd \mathcal F$ over the simple space $\Xs$. We will do this be discussing the frame bundle and the spin structure on each of the three open sets described in the items $1)$, $2)$ and $3)$ above. Remark that at each time $t>0$, we have the frame bundle $\PSO X_t$ and the spin structure $\Pspd X_t$ satisfying Hypothesis \ref{hypot}, corresponding to the metric $G_t$. We also have them at time $t=0$ on $X \setminus \gamma$. We claim that we can extend the family of spin structures up to the front face $\ff$. 

In order to identify the frame bundles (and afterwords the spinor bundles) at different time moments, we use the Bourguignon--Gauduchon method \cite{bgaud}. More precisely, consider the adapted metric
\[  \widetilde{G_t}:= G_t + dt^2 \]
on $\RR_+ \times X$, and remark that the tangent spaces $TX_{t_0}$ and $TX_{t_1}$ at different positive time moments $t_0, t_1 >0$ (with the different metrics $G_{t_0}$ and $G_{t_1}$) are isometric by the parallel transport of the metric $\widetilde{G_t}$ along the vertical geodesics (the integral curves of $\partial_t$). It follows that we have an identification between the frame bundles $\PSO X_{t_0}$ and $\PSO X_{t_1}$ at time $t_0, t_1$, which provides an identification between the spin structures $\Pspd X_{t_0}$ and $\Pspd X_{t_1}$. 

Thus on the first open set $X \times (0, \infty)$, we pull-back the existing spin structure from $(X, G_{t_0})$, for any positive time $t_0 >0$, and the choice is independent of $t_0$:
\begin{equation}
\label{spinstr1}
\centering
\begin{split}
\begin{tikzpicture}
[x=1mm,y=1mm]
\node (tl) at (0,15) {$p^*_1 \Pspd X_{t_0}$};
\node (tr) at (30,15) {$\Pspd X_{t_0}$};
\node (bl) at (0,0) {$X \times (0,\infty)$};
\node (br) at (30,0) {$\lp X, G_{t_0} \rp.$};
\draw[->,densely dashed] (tl) to (bl);
\draw[->] (tr) to (br);
\draw[->] (bl) to node[above,font=\small]{$p_1$} (br);
\end{tikzpicture}
\end{split}
\end{equation}
The same Bourguignon--Gauduchon method of parallel transport also provides a spin structure over the second open set $\lp X \setminus \gamma \rp \times [0, \infty)$, since the metric $G_t$ is non-degenerate there up to time $t=0$. This spin structure and the one defined in \eqref{spinstr1} agree on the intersection $\lp  X \setminus \gamma \rp \times (0, \infty)$. 

Finally, on the third open set $U$, the tubular neighborhood of the front face $\ff$, the tangent bundle $\mathcal F$ is trivial, generated by the orthonormal frame 
\[  p:= \{ \beta^*W, Y \}. \]
We define the spin structure on $U$ such that the lift of $p$ along $\gamma$ does \emph{not} close continuously, according to Hypothesis \ref{hypot}. We therefore obtain a well-defined spin structure
\begin{equation}
\label{defspistr}
\centering
\begin{split}
\begin{tikzpicture}
[x=1mm,y=1mm]
\node (tr) at (30,15) {$\Pspd \mathcal F$};
\node (br) at (30,0) {$\Xs,$};
\draw[->] (tr) to (br);
\end{tikzpicture}
\end{split}
\end{equation}
over the simple space, up to the front face $\ff$, which agrees with the spin structures on $(X,G_t)$, for each $t>0$. It is a principal bundle covering $2:1$ the frame bundle of $\mathcal F$ with respect to the metric $\widetilde{G_t}$
\[  \Pspd \mathcal F \longrightarrow \PSO \mathcal F. \]

\begin{definition}\label{fibratulspinorilor}
The \emph{extension of the spinor bundle} over the simple space $\Xs$ (including the front face) is defined as
\[ \Ss:=\Pspd \mathcal F \times_{\Spin(2)} \Sigma_2 = \Ss^+ \oplus \Ss^{-}.  \]
\end{definition}

Now we need to define a natural connection in the tangent bundle $\mathcal F$. Remark that the bundle $\mathcal F$ is horizontal, in the sense that 
\[ \widetilde{G_t} \lp V, \partial_t \rp =0, \] 
for any $V \in F$. Moreover,
\[ 0= \partial_t  \widetilde{G_t} \lp V, \partial_t \rp = \widetilde{G_t} \lp \nabla_{\partial_t} V, \partial_t \rp +  \widetilde{G_t} \lp V, \nabla_{\partial_t} \partial_t \rp,  \]
and since $\nabla_{\partial_t} \partial_t =0$ (because the vertical lines are geodesics in $\Xs \setminus \ff$), it follows that 
\[  \widetilde{G_t} \lp \nabla_{\partial_t} V , \partial_t \rp=0, \]
thus the vector field $\nabla_{\partial_t} V$ is also horizontal. 

\begin{definition}\label{fconex}
Let $\mathcal F$ be a smooth vector bundle over a smooth manifold $M$. We say that 
\[ \nabla: \Gamma(\mathcal F) \times \Gamma(\mathcal F)  \longrightarrow  \Gamma(\mathcal F)  \]
is an \emph{$\mathcal F$-connection} if it is $\mathcal C^{\infty}(X)$ linear in the first argument, and it verifies the Leibniz rule in the second argument.
\end{definition}

The family of Levi-Civita connections $\nabla:=\lp \nabla^{\LC}_t \rp_{t \geq 0}$ on the tangent spaces of $(X,G_t)$ up to $\{ t=0 \}$ in $\mathbb R_+ \times X$ induces an $\mathcal F$-connection (see Definition \ref{fconex}) on the bundle $\mathcal F$ away from the front face $\ff$. 

The metric $\widetilde{G_t}$ is flat near the geodesic $\gamma$, thus
\begin{equation}\label{nbl}
\nabla_{\beta^* W} \beta^* W =\nabla_{Y} Y=\nabla_{\beta^* W} Y= \nabla_{Y} \beta^* W =0,
\end{equation}
therefore $\nabla$ trivially extends up to the front face $\ff$.

In this manner, we obtain an $\mathcal F$-connection in the bundle $\mathcal F$ over the simple space $\Xs$, which induces an $\mathcal F$-connection in the spinor bundle $\Ss$ (see Definition \ref{fibratulspinorilor}). From now on, we will denote this $\mathcal F$-connection acting on the spinor bundle $\Ss$ by $\nabla$.


\section{The Dirac operator as a cusp-surgery differential operator}
Now we define the Dirac operator of the metric $G_t$ as in the general case, i.e., the composition of the Clifford action (see \eqref{cliff}) and the $\mathcal F$-connection on the extended spinor bundle $\Ss$ 
\[ \Ddd:= \cl \circ \nabla. \]
The key point is that we merged the family of Dirac operators $\lp \Ddd_t \rp_{t>0}$ on the spin Riemannian surfaces $(X,G_t)$, together with the family of Dirac operators $\lp \Ddd_t \rp_{t \geq 0}$ on the family of Riemannian surfaces $(X \setminus \gamma, G_t)$, and we extended it up to the front face of $\Xs$ to obtain  a cusp-surgery differential operator $\Ddd \in \Psi^{1,0,0}_{\cp}(X)$.

Using the splitting of the spinor bundle (see Definition \ref{fibratulspinorilor}), we regard the Dirac operator as a $2 \times 2$ matrix acting as follows
\begin{equation}
\label{diagrdirac}
\Ddd= \left[ \begin{matrix}
0 & \ \Ddd^- \\
\ \Ddd^+ & 0
\end{matrix} \right]
\colon \phantom{+}
\begin{matrix}
\mathcal{S}^+ \\ \oplus \phantom{+} \\ \mathcal{S}^-
\end{matrix}
\longrightarrow \phantom{+}
\begin{matrix}
\mathcal{S}^+ \\ \oplus \phantom{+} \\ \mathcal{S}^-.
\end{matrix}
\end{equation}
We want to prove that the normal operator of $\Ddd$ is invertible. In order to obtain this result, we first need some local computations. Recall from Section \ref{extspin} that the orthonormal frame
\[ p:=\{ \beta^*W,Y  \} \in \PSO \mathcal F \]
is parallel along the geodesic $\gamma$. If we fix 
$\tilde{p} \in \Pspd \mathcal F$
a lift of $p$, it is not continuous due to our Hypothesis \ref{hypot}, but remark instead that \[e^{\tfrac{iy}{2}} \tilde{p} = \left\{ e^{\tfrac{iy}{2}} \beta^* W ,   e^{\tfrac{iy}{2}} Y \right\} \] 
\emph{is} continuous along the geodesic $\gamma$. Therefore a local smooth basis of spinors near the geodesic $\gamma$ is given by
\begin{align*}\label{fi+}
\varphi^+:=e^{\tfrac{iy}{2}} \left[ \tilde{p}, 1 \right], && \varphi^-:=e^{\tfrac{iy}{2}} \left[ \tilde{p}, \tfrac{1}{\sqrt{2}} \lp e_1-i e_2 \rp \right].
\end{align*}
Using \eqref{nbl}, let us remark that 
\begin{equation}\label{calculenabla}
\begin{aligned}
\nabla_{\beta^* W} \varphi^+ = 0, && \nabla_{\beta^* W} \varphi^- = 0 && \nabla_{Y} \varphi^+ = \frac{i}{2} \varphi^+, && \nabla_{Y} \varphi^- = \frac{i}{2}\varphi^-.  
\end{aligned}
\end{equation}

\section{The invertibility of the indicial family of the Dirac operator}
Our aim in this section is to prove that under the Hypothesis \ref{hypot} regarding the invertibility of the spin structure on the two cusps, the indicial family of the Dirac operator is invertible. First, we compute the normal operator of  
\[\Ddd^{-}\Ddd^+ : \Ss^+ \longrightarrow \Ss^+ \]
which acts on the sections of a fiber bundle of rank $1$ for which the positive spinor $\varphi^+=e^{\tfrac{iy}{2}} s^+$ is a local basis near the geodesic $\gamma$. Remark that the distributional kernel of $\Ddd^{-}\Ddd^{+}$ is given by
\begin{equation}\label{kerdir}
k_{\Ddd^{-} \Ddd^{+}} = \Ddd^{-} \Ddd^{+}_{(1)} \lp   k_{\id_{\Ss^+}} \rp,
\end{equation}
where 
\[     k_{\id_{\Ss^+}}  = \delta \lp x_1 - x_2 \rp \delta \lp y_1 - y_2 \rp \varphi_y^+ \otimes \varphi_x^+   \]
is the kernel of the identity operator acting on $\Ss^+$, $x=(x_1,y_1)$, $y=(x_2,y_2)$ are two points near the geodesic $\gamma$, and we denoted by $\Ddd^- \Ddd^+_{(1)}$ the action of the operator in the first set of variables. Using the Lichnerowicz formula (see e.g. \cite[Theorem~3.52]{berline}), we have that
\[ \Ddd^{-} \Ddd^{+} = \nabla^*\nabla + \frac{1}{4} \scal,    \]
where $\nabla$ is the $\mathcal{F}$-connection that we extended up to the front face $\ff \subset \Xs$ in Section \ref{extspin}, and $\scal$ is the scalar curvature of $(X,G_t)$ which is zero near the geodesic $\gamma$ (since the metric $G_t$ is flat there). Notice that locally in the orthonormal basis $\{ \beta^*W, \partial_{y} \}$ of the bundle $\mathcal F$, the connection $\nabla$ is given by
\begin{equation}\label{nabla}
\nabla ( \cdot ) = \lp \beta^* W \rp^{\sharp} \otimes \nabla_{\beta^* W} (\cdot) + dy \otimes \nabla_{Y} (\cdot).
\end{equation}
In order to compute the adjoint $\nabla^*$, consider the $1$-form $\alpha:=g_1 \beta^*W + g_2 Y$, and let $u:=f \varphi^+$ be another spinor. Using the relations $\nabla_{\beta^*W} \varphi^+=0$ and $\nabla_{Y} \varphi^+=\tfrac{i}{2} \varphi^+$ (see \eqref{calculenabla}), we get 
\begin{align*}
\langle \nabla^* \lp \alpha \otimes \varphi^+  \rp , u  \rangle_{L^2}={}& \langle \alpha \otimes \varphi^+, \nabla u   \rangle_{L^2} = \int_{X} \lp g_1 \overline{\beta^*W f} + g_2 \overline{\lp \tfrac{i}{2} + Y \rp f} \rp \dvol_{G_t} \\
={}& \int_X \lp \lp -\tfrac{i}{2} - Y \rp g_2f - \partial_x g_1 f  \rp \dvol_{G_t} \\
={}&\left\langle \lp \lp -\tfrac{i}{2} -Y \rp g_2 - \beta^*W g_1  \rp \varphi^+, u \right\rangle_{L^2(X,G_t)},
\end{align*}
and therefore the adjoint of $\nabla$ is obtained as
\begin{equation}\label{adjnabla}
\nabla^* \lp g_1 \beta^* W \otimes \varphi^+ + g_2 d y \otimes \varphi^+ \rp = \lp \lp - \tfrac{i}{2} - Y \rp g_2 - \beta^*W g_1  \rp \varphi^+.
\end{equation}
Furthermore, remark that using \eqref{nabla}, we have
\begin{align*}
\nabla u ={}& \nabla \lp f \varphi^+ \rp = d f \otimes \varphi^+ + f \nabla \varphi^+ =  d f \otimes \varphi^+ + f d y \otimes \tfrac{i}{2} \varphi^+  \\
={}&\left[ \beta^* W(f) (\beta^*W)^{\sharp} + \lp \partial_y f + \tfrac{i}{2} f  \rp d y \right] \otimes \varphi^+,
\end{align*}
and by \eqref{adjnabla}, it follows that
\begin{equation}\label{nablastnabla}
\begin{aligned}
\nabla^* \nabla \lp f \varphi^+ \rp ={}&  \lp -(\beta^* W)^2 f + \lp -\tfrac{i}{2} - \partial_y \rp \lp \partial_y f + \tfrac{i}{2} \rp  \rp \varphi^+ \\
={}& - \lp (\beta^* W)^2 + \lp \partial_y^2 + \tfrac{i}{2}  \rp^2 f  \rp \varphi^+.
\end{aligned}
\end{equation}
If we denote by $A:= \partial_y + \tfrac{i}{2}$, equation \eqref{nablastnabla} implies that
\[ \nabla^* \nabla \lp f \varphi^+ \rp = \lp (\beta^*W)^* (\beta^*W) + A^*A \rp f \otimes \varphi^+ .  \]
Let us pull-back the vector field 
\begin{equation}\label{w}
W=(x_1^2+t^2) \partial_{x_1}
\end{equation}
(acting in the first set of variables due to \eqref{kerdir}) through the blow-down map $\tilde{\beta}$ (see \eqref{tildebeta}). We will work in polar coordinates $(\rho, \sigma, \theta)$, which are suitable over all the cusp front face. We first linearly change the variables $(t,x_1,x_2,y_1,y_2)$ to 
\begin{equation}\label{cod1}
\lp t,    v_1:=\tfrac{1}{\sqrt{2}} (x_1+x_2) ,  v_2:=\tfrac{1}{\sqrt{2}} (x_1-x_2) , y_1, y_2   \rp.
\end{equation} 
After the blow up of $\{ t=v_1=v_2=0 \}$ in \eqref{dublunerig} (which produces $\ffb$), we introduce polar coordinates
\begin{align*}
v_1= \rho \cos \sigma \cos \theta, && v_2=\rho \sin \sigma, && t=\rho \cos \sigma \sin \theta,
\end{align*}
where $\theta \in [0, \pi)$ and $\sigma \in \left[ -\tfrac{\pi}{2}, \tfrac{\pi}{2}  \right]$. Hence the new coordinates are:
\begin{equation}\label{cod2}
 \lp \rho:= \sqrt{t^2 + v_1^2 + v_2^2} , \  \sigma:= \arcsin \tfrac{v_2}{\rho}, \
  \theta:= \arctan \tfrac{t}{v_1}  , \ y_1,  \ y_2 \rp.
\end{equation}
Now we blow-up $\{ \rho=0, \sigma=0 \}$ to obtain coordinates along the cups front face $\ffc$
\begin{equation}\label{cod3}
 \lp  r:=\rho, u:=\tfrac{\sigma}{\rho}, \theta, y_1, y_2  \rp,
 \end{equation}
and notice that $r$ is a boundary defining function for the cusp front face inside $\Xd$. 

One can easily check that in the coordinates \eqref{cod1}, the vector field $W=(x_1^2+t^2) \partial_{x_1}$ from \eqref{w} is given by
\begin{equation}\label{w1}
W= \tfrac{1}{\sqrt 2} \lp  \tfrac{1}{2} (v_1 + v_2)^2 + t^2 \rp \lp \partial_{v_1} + \partial_{v_2} \rp.
\end{equation}
Since
\begin{align*}
{}&\partial_{v_1} =\cos \sigma \cos \theta \partial_{\rho} - \frac{\sin \theta}{\rho \cos \sigma} \partial_{\theta} - \frac{\sin \sigma \cos \theta}{\rho} \partial_{\sigma}, && \partial_{v_2}= \sin \sigma \partial_{\rho} + \frac{\cos \sigma}{\rho} \partial_{\sigma},
\end{align*}
it follows that the lift of $W$ through the first blow-up (the one which produces $\ffb$, see Fig. \ref{bw1}) is given by
\begin{equation}\label{wbw1}
\begin{aligned}
{}&\frac{\rho^2}{\sqrt{2}} \lp \frac{1}{2} \lp \cos \sigma \cos \theta + \sin \sigma \rp^2 + \cos^2 \sigma \sin^2 \theta  \rp \cdot \\
{}&\lp  \lp \cos \sigma \cos \theta + \sin \sigma  \rp \partial_{\rho} - \frac{\sin \theta}{\rho \cos \sigma} \partial_{\theta} + \frac{1}{\rho} \lp \cos \sigma - \sin \sigma \cos \theta \rp \partial_{\sigma}  \rp .
\end{aligned}
\end{equation}
Notice that 
\begin{align*}
\partial_{\rho}= \partial_{r} - \frac{u}{r} \partial_u, && \partial_{\sigma}=\frac{1}{r} \partial_{u},
\end{align*}
thus the lift of $W$ in \eqref{w} through the blow-down map $\tilde{\beta}: \Xd \longrightarrow [0, \infty) \times X \times X$ in the coordinates \eqref{cod3} is
\begin{equation}\label{wbw2}
\begin{aligned}
\tilde{\beta}^* W ={}& \frac{r^2}{\sqrt{2}} \lp \frac{1}{2} \lp \cos(ru) \cos \theta + \sin (ru)  \rp^2  + \cos^2 (ru) \sin^2 \theta \rp \cdot \\ 
{}&\left[ \lp \cos(ru) \cos \theta + \sin (ru)  \rp \partial_r  - \frac{u}{r} \lp \cos(ru) \cos \theta + \sin (ru) \rp \partial_u -  \right.  \\
{}&\left. \frac{\sin \theta}{r \cos(ru)} \partial_{\theta} + \frac{1}{r^2} \lp \cos(ru) - \sin (ru) \cos \theta  \rp \partial_u   \right].
\end{aligned}
\end{equation}
The normal operator is given by the restriction of the operator to the cusp front face, thus we restrict \eqref{wbw2} to $\{ r=0 \}$:
\begin{align*}
\tilde{\beta}^* W_{\vert_{\ffc}} = \frac{1}{\sqrt{2}} \lp \frac{1}{2} \cos^2 \theta + \sin^2 \theta  \rp \partial_u = \frac{1}{\sqrt{2}} \lp \frac{1}{2} + \frac{\sin^2 \theta}{2} \rp \partial_u.
\end{align*}
If we denote by $\xi$ the cotangent variable of $u$ in the $\mathbb R$-fiber of the cusp front face at a fixed angle $\theta$, then the indicial family of the Dirac operator is given by
\begin{align*}
\II(\Ddd)(\xi) = \Dd_{\vert_{\gamma}} + i \xi \frac{1}{2 \sqrt{2}} \lp 1 + \sin^2 \theta  \rp,
\end{align*} 
thus 
\[ \II(\Ddd)(\xi)^* \II(\Ddd)(\xi) = \Dd^2_{\vert \gamma} + \frac{\xi^2}{8} \lp 1 + \sin^2 \theta \rp^2 \geq \frac{1}{4},  \]
for any $\xi \in \mathbb R$. Using our Hypothesis \ref{hypot}, we obtained that the indicial family is invertible, and therefore the normal operator of the Dirac operator $\Ddd$ is invertible.

\section{The Dirac operator for the metric \texorpdfstring{$g_t$}{gt}}

We managed to compute the normal operator of the Dirac operator $\Ddd$ corresponding to the family of metrics $(G_t)_{t \geq 0}$ in \eqref{gtGt}. Now we want to go back to the Dirac operator $\Dd$ which interest us, the one corresponding to the family of metrics $(g_t)_{t \geq 0}$, where
\[ g_t=(x^2+t^2) G_t \]
near $\gamma$. The two families of metrics $(g_t)_{t \geq 0}$ and $(G_t)_{t \geq 0}$ differ through a conformal change, namely the multiplication by the square of the boundary defining function for the front face $\rho_{\ff}$.
Recall the following classical result (see for instance \cite[Proposition~2.31]{spinorial} or \cite[Proposition~1.3.10]{ginoux}).
\begin{proposition}\label{confdirac}
Let $M$ be an $n$-dimensional spin manifold with a fixed spin structure, and let $g, \bar{g}$ be two conformal metrics on $M$ such that $\bar{g}=e^{2u}g$, where $u \in \mathcal C^{\infty}(M)$. Then there exists an isomorphism between the spinor bundles 
\begin{align*}
{}&\Ss_{(M,g)} \longrightarrow \Ss_{(M, \bar{g})}, \\
{}& \ \ \ \ \ s \longmapsto \bar{s}
\end{align*}
and the two Dirac operators $\Ddd$, $\overline{\Ddd}$ corresponding to the metrics $g$, $\bar{g}$ respectively, are related by the formula
\[     \overline{\Ddd} \lp e^{{-\tfrac{n-1}{2}} u} \ \bar{s} \rp=  e^{{-\tfrac{n+1}{2}} u} \ \overline{\Ddd s}. \] 
\end{proposition}
In our case, 
\[ g_t = \rho_{\ff}^2 G_t = e^{2 \log \rho_{\ff}} G_t, \]
thus by Proposition \ref{confdirac}, we get
\begin{align*}
{}&\Dd \lp e^{-\tfrac{1}{2} \log \rho_{\ff}} \  s \rp = e^{-\tfrac{3}{2}  \log \rho_{\ff} } 
\ \Ddd s \ \Longleftrightarrow  \ \Dd \lp \rho_{\ff}^{-1/2 } \ s \rp = \rho_{\ff}^{-3/2} \Ddd s,
\end{align*}
where we denoted by $\Dd$ the Dirac operator of the metric $g_t$. It follows that
\begin{equation}\label{dirpl}
\Dd^+ \lp f \varphi^+ \rp = \rho_{\ff}^{-3/2}  \Ddd^+ \lp \rho_{\ff}^{1/2} \ f \varphi^+ \rp.
\end{equation}
Since $\Ddd^+= \lp \cl (W) \nabla_W + \cl (Y) \nabla_Y \rp^+$ and $W \lp \rho_{\ff}^{1/2} \rp = \tfrac{x}{2} \rho_{\ff}^{1/2} $, we get
\begin{align*}
 \Dd^+ \lp \rho_{\ff}^{1/2} \ f \varphi^+ \rp  {}&= \rho_{\ff}^{1/2} \Ddd^+ \lp f \varphi^+ \rp + \cl (W) W \lp \rho_{\ff}^{1/2}  \rp  f \varphi^+  = \rho_{\ff}^{1/2} \Ddd^+ \lp f \varphi^+ \rp  + \tfrac{x}{2} \rho_{\ff}^{1/2} \cl(W) \lp f \varphi^+ \rp.
\end{align*}
Therefore equation \eqref{dirpl} becomes
\begin{equation}\label{dirplus}
\Dd^+ \lp f \varphi^+ \rp  = \rho_{\ff}^{-1} \lp  \Ddd^+  (f \varphi^+ ) +  \tfrac{x}{2}  \cl(W) \lp f \varphi^+ \rp  \rp,
\end{equation}
and we obtain a similar expression for $\Dd^-$
\begin{equation}\label{dirminus}
\Dd^- \lp f \varphi^- \rp  = \rho_{\ff}^{-1} \lp \Ddd^-  (f \varphi^- ) +  \tfrac{x}{2}  \cl(W) \lp f \varphi^- \rp  \rp.
\end{equation}
Using \eqref{dirplus} and \eqref{dirminus}, we obtain the connection between the Dirac operator ${\Dd}$ of the family of metrics $(g_t)_{t \geq 0}$ and the Dirac operator $\Ddd$ of the family of metrics $(G_t)_{t \geq 0}$
\begin{equation}\label{legaturadiraci}
 \Dd= \rho_{\ff}^{-1} \ \lp  \Ddd + \tfrac{x}{2} \cl(W) \rp.
\end{equation}
Thus the Dirac operator $\Dd$ of the family of metrics $(g_t)_{t \geq 0}$ is a cusp-surgery differential operator belonging to
\[ \Dd \in \Psi^{1,1,0}_{\cp}(X) .\] 
Notice that the normal operator does not detect the term $\tfrac{x}{2} \cl (W)$ in \eqref{legaturadiraci},  thereby proving the following result.
\begin{proposition}\label{norminv}
The normal operator of the Dirac operator $\Dd$ of the metric $g_t$ is equal to the normal operator of the operator $\Ddd$:
\[ \NN_1 (\Dd) = \NN_0 (\Ddd).  \] 
Furthermore, the normal operator $\NN_1(\Dd)$ is invertible.
\end{proposition}

Let $\lambda \in \RR$ such that the cusp differential operator $\Dd_0 - \lambda \in \Psi^{1,0}_{\cc}(X)$ is invertible, where we denoted by $\Dd_0$ the Dirac operator at time $\{ t=0 \}$ in the pinching process. Remark that the temporal operator of the cusp-surgery differential operator $\Dd$ is 
\[  \Te (\Dd) = \Dd_0, \]
thus $\Te (\Dd - \lambda) = \Dd_0 - \lambda$ is invertible. Furthermore, $\lambda \in \Psi^{0,0,0}_{\cp}(X)$ is not ``seen" by the normal operator $\NN_1$, hence
\[ \NN_1 (\Dd)= \NN_1 (\Dd - \lambda), \]
which together with Proposition \ref{norminv} implies that the normal operator of $\Dd - \lambda$ is also invertible. Finally, the principal cusp-surgery symbol $\sigma_{\cp} \lp \Dd-\lambda \rp$ is also invertible, as is usually the case with Dirac operators, and we state these results in the following proposition.
\begin{proposition}\label{D-lambda}
If $\lambda \in \RR$ such that $\Dd_0 - \lambda \in \Psi^{1,0}_{\cc}(X)$ is invertible, then all the three symbols  $\sigma_{\cp}$, $\NN$, and $\Te$ of the cusp surgery differential operator $\Dd-\lambda$ are invertible.
\end{proposition}

\section{The construction of the parametrix}
As for the usual pseudodifferential calculus (see for instance \cite[Chapter 4]{spinorial}), the \emph{b}-calculus (see e.g., \cite[Section 4.16]{melrose}) or the fibred cusp calculus (see \cite[Section 3 and Proposition 5]{melmaz98}), we have the following results.
\begin{proposition}\label{sirscurtexact}
There exists three natural short exact sequences:
\begin{itemize}
\item[$i)$] $0 \longrightarrow \Psi^{m-1,\alpha, \beta}_{\cp}(X) \longrightarrow \Psi^{m,\alpha, \beta}_{\cp}(X) \xrightarrow{ \ \sigma_{\cp} \ } \mathcal C^{\infty} \lp S^*N(\Delta) \rp \longrightarrow 0    $.
\item[$ii)$] $0 \longrightarrow \Psi^{m,\alpha-1, \beta}_{\cp}(X) \longrightarrow \Psi^{m,\alpha, \beta}_{\cp}(X) \xrightarrow{ \ \mathcal N \ } \Psi^m_{[0, \pi], \sus}(\gamma)  \longrightarrow 0    $.
\item[$iii)$] $0 \longrightarrow \Psi^{m,\alpha, \beta-1}_{\cp}(X) \longrightarrow \Psi^{m,\alpha, \beta}_{\cp}(X)  \xrightarrow{\ \Te \ } \Psi^m_{\cc}(X)   \longrightarrow 0    $.
\end{itemize}
\end{proposition}

\begin{proposition}\label{asimptsum}
Let $m, \alpha, \beta \in \RR$ and consider a sequence of cusp surgery pseudodifferential operators $(A_j)_{j  \in \mathbb N}$, where
\[ A_j \in \Psi^{m-j, \alpha-j, \beta-j}_{\cp}(X).  \]
Then there exists an \emph{asymptotic sum} $A=\sum_{j \in \mathbb N} A_j \in \Psi^{m,\alpha,\beta}_{\cp}(X)$, in the sense that for any natural number $k$
\[ A - \sum_{j=0}^{k-1} A_j \in \Psi^{m-k, \alpha-k, \beta-k}. \]
\end{proposition}

Now let us prove that whenever we start with an operator having all the three symbols invertible, we can find an inverse modulo residual operators.
\begin{proposition}[The existence of the parametrix]\label{parametrix}
If $A \in  \Psi^{m,\alpha, \beta}_{\cp}(X) $ is a cusp-surgery pseudodifferential operator having all the three symbols invertible, then there exists an operator $ Q \in \Psi^{-m,-\alpha, -\beta}_{\cp}(X)$ such that 
\begin{align*}
{}& AQ=1-R, && QA=1-R',
\end{align*}
where $R,R' \in \Psi^{-\infty, -\infty, -\infty}_{\cp}(X)$ are residual operators, and $Q$ is called a \emph{parametrix} for $A$.
\end{proposition}
\begin{proof}
We first arrange the short exact sequences from Proposition \ref{sirscurtexact} in a $3$-dimensional commutative diagram, as follows:
\begin{equation}
\label{parametrix3d}
\centering
\begin{split}
\begin{tikzpicture}
[x=1mm,y=1mm]
\node (111) at (0,0) {$\Psi^{m-1,\alpha-1,\beta-1}_{\cp}(X)$};
\node (211) at (40,0) {$\Psi^{m-1,\alpha,\beta-1}_{\cp}(X)$};
\node (311) at (80,0) {$\rho_{\tb}\Psi^{m-1}_{[0, \pi],\sus}(\gamma)$};
\node (121) at (0,30) {$\Psi^{m,\alpha-1,\beta-1}_{\cp}(X)$};
\node (221) at (40,30) {$\Psi^{m,\alpha,\beta-1}_{\cp}(X)$};
\node (321) at (80,30) {$\rho_{\tb}\Psi^{m}_{[0, \pi],\sus}(\gamma)$};
\node (131) at (0,60) {$C^\infty(S^* N(\Delta))$};
\node (231) at (40,60) {$C^\infty(S^* N(\Delta))$};
\node (331) at (80,60) {$C^\infty \lp S^*_{\sus}(\gamma) \rp$};
\node (112) at (10,10) {$\Psi^{m-1,\alpha-1,\beta}_{\cp}(X)$};
\node (212) at (50,10) {$\Psi^{m-1,\alpha,\beta}_{\cp}(X)$};
\node (312) at (90,10) {$\Psi^{m-1}_{\sus}(\gamma)$};
\node (122) at (10,40) {$\Psi^{m,\alpha-1,\beta}_{\cp}(X)$};
\node (222) at (50,40) {$\bm{\Psi^{m,\alpha,\beta}_{\cp}(X)}$};
\node (322) at (90,40) {{\color{blue}$\Psi^{m}_{[0,\pi],\sus}(\gamma)$}};
\node (132) at (10,70) {$C^\infty(S^* N(\Delta))$};
\node (232) at (50,70) {{\color{purple}$C^\infty(S^* N(\Delta))$}};
\node (332) at (90,70) {$C^\infty \lp S^*_{\sus}(\gamma) \rp$};
\node (113) at (20,20) {$\Psi^{m,\alpha-1,\beta-1}_{\cc}(X)$};
\node (213) at (60,20) {$\Psi^{m-1}_{\cc}(X)$};
\node (313) at (100,20) {$\Psi^{m-1}_{\{0, \pi\},\sus}(\gamma)$};
\node (123) at (20,50) {$\Psi^{m,-1}_{\cc}(X)$};
\node (223) at (60,50) {{\color{orange} $\Psi^{m}_{\cc}(X)$}};
\node (323) at (100,50) {$\Psi^{m}_{\{0, \pi \}\sus}(\gamma)$};
\node (133) at (20,80) {$C^\infty(S^{*c}X)$};
\node (233) at (60,80) {$C^\infty(T^{*c}X)$};
\node (333) at (100,80) {$C^\infty \lp S^*_{\{0, \pi \},\sus}(\gamma) \rp$};
\draw[->] (111) to (211);
\draw[->] (211) to (311);
\draw[->] (121) to (221);
\draw[->] (221) to (321);
\draw[->] (131) to (231);
\draw[->] (231) to (331);
\draw[->,densely dashed] (112) to (212);
\draw[->,densely dashed] (212) to (312);
\draw[->,densely dashed] (122) to (222);
\draw[->,densely dashed] (222) to (322);
\draw[->] (132) to (232);
\draw[->] (232) to (332);
\draw[->,densely dashed] (113) to (213);
\draw[->,densely dashed] (213) to (313);
\draw[->,densely dashed] (123) to (223);
\draw[->,densely dashed] (223) to (323);
\draw[->] (133) to (233);
\draw[->] (233) to (333);
\draw[->] (111) to (121);
\draw[->] (121) to (131);
\draw[->] (211) to (221);
\draw[->] (221) to (231);
\draw[->] (311) to (321);
\draw[->] (321) to (331);
\draw[->,densely dashed] (112) to (122);
\draw[->,densely dashed] (122) to (132);
\draw[->,densely dashed] (212) to (222);
\draw[->,densely dashed] (222) to (232);
\draw[->] (312) to (322);
\draw[->] (322) to (332);
\draw[->,densely dashed] (113) to (123);
\draw[->,densely dashed] (123) to (133);
\draw[->,densely dashed] (213) to (223);
\draw[->,densely dashed] (223) to (233);
\draw[->] (313) to (323);
\draw[->] (323) to (333);
\draw[->,densely dashed] (111) to (112);
\draw[->,densely dashed] (112) to (113);
\draw[->,densely dashed] (211) to (212);
\draw[->,densely dashed] (212) to (213);
\draw[->] (311) to (312);
\draw[->] (312) to (313);
\draw[->,densely dashed] (121) to (122);
\draw[->,densely dashed] (122) to (123);
\draw[->,densely dashed] (221) to (222);
\draw[->,densely dashed] (222) to (223);
\draw[->] (321) to (322);
\draw[->] (322) to (323);
\draw[->] (131) to (132);
\draw[->] (132) to (133);
\draw[->] (231) to (232);
\draw[->] (232) to (233);
\draw[->] (331) to (332);
\draw[->] (332) to (333);
\end{tikzpicture}
\end{split}
\end{equation}
where the maps are:
\begin{equation}
\label{directii}
\centering
\begin{split}
\begin{tikzpicture}
[x=1mm,y=1mm]
\node (tl) at (0,30) {$\bullet$};
\node (tr) at (20,16) {$\bullet$};
\node (ml) at (0,15) {$\bullet$};
\node (mm) at (10,8) {$\bullet$};
\node (bl) at (0,0) {$\bullet$};
\node (bm) at (20,0) {$\bullet$};
\node (br) at (40,0) {$\bullet$.};
\node at (tr.east) [anchor=west] {};
\node at (br.south) [anchor=north] {};
\draw[->] (ml) to node[left,font=\small]{$\sigma_{\cp}$} (tl);
\draw[->,densely dashed] (mm) to node[above left,font=\small]{$\Te$} (tr);
\draw[->] (bm) to node[above,font=\small]{$\mathcal{N}$} (br);
\draw[<-,>=right hook] (bl) to ($ (bl)!0.5!(ml) $) node[left,font=\small]{$i$}; \draw[->] ($ (bl)!0.5!(ml) $) to (ml);
\draw[<-,>=right hook,densely dashed] (bl) to ($ (bl)!0.5!(mm) $) node[above left,font=\small]{$i$}; \draw[->,densely dashed] ($ (bl)!0.5!(mm) $) to (mm);
\draw[<-,>=right hook] (bl) to ($ (bl)!0.5!(bm) $) node[above,font=\small]{$i$}; \draw[->] ($ (bl)!0.5!(bm) $) to (bm);
\end{tikzpicture}
\end{split}
\end{equation}
For each $i,j,k \in \{ 0,1,2 \}$, we denote the map from the node $(i,j,k)$ to $(i+1,j, k)$ by $\dd_{\NN}^{i,j,k}$, the map from $(i,j,k)$ to $(i,j+1, k)$ by $\dd_{\Te}^{i,j,k}$, and the map from the node $(i,j,k)$ to $(i,j,k+1)$ by $\dd_{\sigma_{\cp}}^{i,j,k}$, in correspondence with the three directions in \eqref{directii}. 

Notice that if $m=\alpha=\beta=0$, \eqref{parametrix3d} is a diagram of $\mathbb C$-algebras, and otherwise, it is a diagram of additive groups. Furthermore, remark that each square diagram in \eqref{parametrix3d} is commutative. We first prove the following diagram chasing result.

\begin{lemma}\label{param}
If $\eta \in {\color{blue} \Psi^m_{[0,\pi],\sus} (\gamma)}$, $\tau \in  {\color{orange} \Psi^m_{\cc} (X)}$, and $\sigma \in {\color{purple} \mathcal C^{\infty} \lp S^*N (\Delta) \rp}$ satisfying the compatibility conditions
\begin{align*}
\dd_{\Te}^{2,1,1} \eta = \dd_{\NN}^{1,2,1} \tau, &&
\dd_{\sigma_{\cp}}^{2,1,1} \eta = \dd_{\NN}^{1,1,2} \sigma, &&
\dd_{\sigma_{\cp}}^{1,2,1} \tau = \dd_{\Te}^{1,1,2} \sigma, 
\end{align*}
then there exists an operator $Q_0 \in \Psi^{m,\alpha,\beta}_{\cp}(X)$ such that 
\begin{align*}
\dd_{\NN}^{1,1,1} \lp Q_0 \rp = \eta, && \dd_{\Te}^{1,1,1} \lp Q_0 \rp = \tau, && \dd_{\sigma_{\cp}}^{1,1,1} \lp Q_0 \rp = \sigma.
\end{align*}
\end{lemma}
\begin{proof}
The first step is to find an operator $X_{\eta, \tau} \in \Psi^{m, \alpha, \beta}_{\cp} (X)$ such that $\dd_{\NN}^{1,1,1} (X_{\eta, \tau}) = \eta$ and $\dd_{\Te}^{1,1,1} (X_{\eta, \tau}) = \tau$. Since $\dd_{\Te}^{1,1,1}$ is surjective, there exists an operator $X_{\tau} \in \Psi^{m, \alpha, \beta}_{\cp} (X)$ which is mapped through $\dd_{\Te}^{1,1,1}$ into $\tau$. Using the commutativity of the diagram and the first compatibility condition, we have that
\[ \dd_{\Te}^{2,1,1} \lp \dd_{\NN}^{1,1,1} X_{\tau} - \eta \rp =0,   \]
thus there exists an operator $P' \in \rho_{\tb} \Psi^m_{[0, \pi], \sus}(\gamma)$ such that
\[ \dd_{\Te}^{2,0,1} (P')=\dd_{\NN}^{1,1,1} X_{\tau} - \eta .\] 
Since $\dd_{\NN}^{1,0,1}$ is surjective, take $P \in \Psi^{m,\alpha,\beta-1}_{\cp} (X)$ such that $\dd_{\NN}^{1,0,1} (P)=P'$. One can easily check that
\[ X_{\eta,\tau}:= -\dd_{\Te}^{1,0,1} (P)+X_{\tau}   \]
satisfies the conditions that we want, i.e.,
\begin{align*}
\dd_{\NN}^{1,1,1} (X_{\eta, \tau}) = \eta, && \dd_{\Te}^{1,1,1} (X_{\eta, \tau}) = \tau.
\end{align*}

We notice that $\dd_{\sigma_{\cp}}^{1,1,1} \lp X_{\eta, \tau} \rp$ is not necessarily equal to $\sigma$, hence we need to modify $X_{\eta, \tau}$. Remark that
\[ \dd_{\Te}^{112} \lp  \sigma- \dd_{\sigma_{\cp}}^{1,1,1} X_{\eta,\tau}  \rp = \dd_{\Te}^{1,1,2} \sigma - \dd_{\sigma_{\cp}}^{1,2,1} \dd_{\Te}^{1,1,1} X_{\eta, \tau} =
\dd_{\Te}^{1,1,2} \sigma -  \dd_{\sigma_{\cp}}^{1,2,1}  \tau =0,
\]
thus there exists $Z \in \mathcal C^{\infty} \lp S^*N (\Delta) \rp$ such that 
\begin{equation}\label{Z}
 \dd_{\Te}^{1,0,2} Z =   \sigma- \dd_{\sigma_{\cp}}^{1,1,1} X_{\eta,\tau} ,   
 \end{equation}
and moreover, since $\dd_{\sigma_{\cp}}^{1,0,1}$ is surjective, it follows that there exists $T \in \Psi^{m,\alpha,\beta-1}_{\cp}(X)$ verifying
\begin{equation}\label{T}
  \dd_{\sigma_{\cp}}^{1,0,1} (T)=Z.
  \end{equation}
Using \eqref{Z} and the second compatibility condition, we get
\begin{align*}
\dd_{\NN}^{1,1,2} \dd_{\Te}^{1,0,2} Z {}&= \dd_{\NN}^{1,1,2} \lp \sigma- \dd_{\sigma_{\cp}}^{1,1,1}  X_{\eta,\tau}  \rp= \dd_{\NN}^{1,1,2} \sigma- \dd_{\sigma_{\cp}}^{2,1,1} d_{\NN}^{1,1,1} X_{\eta, \tau} = \dd_{\NN}^{1,1,2} \sigma - \dd_{\sigma_{\cp}}^{2,1,1} \eta=0,
\end{align*}
which implies that $\dd_{\NN}^{1,0,2} Z =0$, hence there exists $Z' \in \mathcal C^{\infty} \lp S^*N \Delta \rp$ such that
\begin{equation}\label{Z'}
Z= \dd_{\NN}^{0,0,2} Z' .
\end{equation}  
Moreover, using the fact that the map $\dd_{\sigma_{\cp}}^{0,0,1}$ is surjective, we take $T'' \in \Psi^{m,\alpha-1, \beta-1}_{\cp}(X)$ such that
\begin{equation}\label{T''}
Z'= \dd_{\sigma_{\cp}}^{0,0,1} T''.
\end{equation}
Notice that using relations \eqref{T}, \eqref{Z'} and \eqref{T''} and the commutativity of the diagram \eqref{parametrix3d}, we get
\begin{align*}
\dd_{\sigma_{\cp}}^{1,0,1} \lp  -T + \dd_{\NN}^{0,0,1} T''   \rp = -Z + \dd_{\NN}^{0,0,2} \dd_{\sigma_{\cp}}^{0,0,1} T'' =-Z+  \dd_{\NN}^{0,0,2} Z'= -Z + Z=0,
\end{align*}
meaning that $ -T + \dd_{\NN}^{0,0,1} T'' \in \Ker  \dd_{\sigma_{\cp}}^{1,0,1} = \Imag \dd_{\sigma_{\cp}}^{1,0,0}$, hence there exists $S \in \Psi^{m-1, \alpha, \beta-1}_{\cp}(X)$ such that
\begin{equation}\label{alfa}
\dd_{\sigma_{\cp}}^{1,0,0} S=-T + \dd_{\NN}^{0,0,1} T''.
\end{equation}
It follows that
\begin{equation}\label{da}
\dd_{\sigma_{\cp}}^{1,0,1} \lp T+ \dd_{\sigma_{\cp}}^{1,0,0} S \rp=\dd_{\sigma_{\cp}}^{1,0,1} \dd_{\NN}^{0,0,1} T''= \dd_{\NN}^{0,0,2} \dd_{\sigma_{\cp}}^{0,0,1} T'' =   \dd_{\NN}^{0,0,2} Z' =
 Z,
\end{equation}
and let us denote by $\delta := T+ \dd_{\sigma_{\cp}}^{1,0,0} S$. We claim that 
\[Q_0:=X_{\eta,\tau} + \dd_{\Te}^{1,0,1} \delta \in \Psi^{m,\alpha,\beta}_{\cp}(X)\] 
is sent through the maps $\dd_{\Te}^{1,1,1}$, $\dd_{\sigma_{\cp}}^{1,1,1}$ and $\dd_{\NN}^{1,1,1}$ to $\tau$, $\sigma$ and $\eta$. Indeed, we first remark that
\begin{align*}
{}& \dd_{\Te}^{1,1,1} Q_0= \dd_{\Te}^{1,1,1} \lp X_{\eta,\tau} + \dd_{\Te}^{1,0,1} \delta  \rp = \tau + 0 = \tau.
\end{align*} 
Now using \eqref{da} and \eqref{Z} and the commutativity of the diagram \eqref{parametrix3d}, we get:
\begin{align*}
\dd_{\sigma_{\cp}}^{1,1,1} Q_0 {}&= \dd_{\sigma_{\cp}}^{1,1,1} \lp X_{\eta,\tau} + \dd_{\Te}^{1,0,1} \delta  \rp =  \dd_{\sigma_{\cp}}^{1,1,1} X_{\eta,\tau} + \dd_{\Te}^{1,0,2} \dd_{\sigma_{\cp}}^{1,0,1} \delta \\
{}&=  \dd_{\sigma_{\cp}}^{1,1,1} X_{\eta,\tau} + \dd_{\Te}^{1,0,2} Z = \dd_{\sigma_{\cp}}^{1,1,1} X_{\eta,\tau} + \sigma-\dd_{\sigma_{\cp}}^{1,1,1} X_{\eta,\tau} =\sigma. 
\end{align*}
Finally, using the definition of $\delta$ and \eqref{alfa}, remark that
\[ \delta=T + \dd_{\sigma_{\cp}}^{1,0,0} S = T-T+\dd_{\NN}^{0,0,1} T'' = \dd_{\NN}^{0,0,1} T'',  \]
thus $\delta \in \Imag  \dd_{\NN}^{0,0,1} = \Ker \dd_{\NN}^{1,0,1}$, which by the commutativity of the diagram \eqref{parametrix3d} implies that 
\[ \dd_{\NN}^{1,1,1} \dd_{\Te}^{1,0,1} \delta = \dd_{\Te}^{2,0,1} \dd_{\NN}^{1,0,1} \delta =0, \]
and furthermore
\begin{align*}
{}&\dd_{\NN}^{1,1,1} Q_0 = \dd_{\NN}^{1,1,1} \lp X_{\eta,\tau} + \dd_{\Te}^{1,0,1} \delta  \rp = \eta + 0 =\eta,
\end{align*}
which ends the proof of the Lemma.
\end{proof}
Let us now prove the existence of the parametrix. Clearly the three symbols of $A$ satisfy the compatibility conditions in the hypothesis of Lemma \ref{param}, therefore the inverses of the three symbols also satisfy them.

By Lemma \ref{param} applied for the three inverses of the symbols of $A$, there exists an operator $Q_0 \in \Psi^{-m,-\alpha,-\beta}_{\cp}(X)$ with the following properties:
\begin{align*}
\NN( Q_0)= \NN(A)^{-1}, && \Te(Q_0)=\Te(A)^{-1}, && \sigma_{\cp}(Q_0)= \sigma_{\cp}(A)^{-1}.
\end{align*}
Denote by 
\[ R_0:=1-AQ_0 \in \Psi^{0,0,0}_{\cp} (X), \]
and notice that using the multiplicativity of the normal operators (see Section \ref{multiplnormali}), we have
\[ \NN \lp 1- A Q_0 \rp = 1- 1=0. \]
Similarly, using the multiplicativity of the cusp-surgery principal symbol and of the temporal operator (see Section \ref{multipltemp}), we obtain that $\sigma_{\cp} (R_0) =0$ and $\Te R_0 =0$. Therefore by Proposition \ref{sirscurtexact}, it follows that $R_0$ is actually a ``better" operator belonging to $\Psi^{-1,-1,-1}_{\cp}(X)$. Using the Neumann series, we remark that for $k \in \mathbb N^*$ we have
\begin{equation}\label{paramk}
A Q_0 \lp 1+ R_0 +R_0^2 +...+R_0^{k-1} \rp = \lp 1- R_0 \rp \lp 1+ R_0 +R_0^2 +...+R_0^{k-1} \rp = 1-R_0^k. 
\end{equation}
We consider $Q_k:=Q_0 \lp 1+ R_0 +R_0^2 +...+R_0^{k-1} \rp$ which is a right parametrix of order $k$ (since the remainder $R_0^k$ in \eqref{paramk} belongs to $\Psi_{\cp}^{-k,-k,-k}(X)$). Using Proposition \ref{asimptsum}, we obtain a right-parametrix $Q$ as the asymptotic sum of \eqref{paramk} as $k \to \infty$: 
\[ AQ=1-R, \text{ where } R \in \Psi_{\cp}^{-\infty,-\infty,-\infty}(X).  \]
By a standard argument, one can easily check that $Q$ is also a left-parametrix.
\end{proof}

\section{Family of resolvents for the Dirac operator}
\begin{proof}[Proof of Theorem \ref{resolvintro}]
In Proposition \ref{D-lambda} we proved that if $\lambda \in \RR \setminus \Spec \Dd_0$, then all the three symbols of $\Dd - \lambda$ are invertible. By Proposition \ref{parametrix}, there exists a parametrix $Q(\lambda) \in \Psi^{-1,-1,0}_{\cp}(X)$,
\begin{equation}\label{constrrez1}
\begin{aligned}
(\Dd-\lambda)Q(\lambda)=1 + T_1 (\lambda), && Q(\lambda)(\Dd-\lambda)=1 + T_2 (\lambda),
\end{aligned}
\end{equation}
where $T_1 , T_2 \in \Psi^{-\infty,-\infty,-\infty}_{\cp}(X)$ are residual operators depending on $\lambda$ (for simplicity, we will not emphasize this dependence in the notation). Remark that $T_1$ is a smooth kernel on the double space $\Xd$, but using Proposition \ref{desumflare}, we can actually regard it as a family of smooth kernels on $[0, \infty) \times X \times X$ with vanishing Taylor series at $\{ t=0 \}$. 

Notice that the operator norm of $T_1$ in the space of bounded operators acting on $L^2 (X, g_t)$ satisfies  
\[ \Vert T_1(t) \Vert_{\mathcal B \lp L^2(X,g_t) \rp}  \longrightarrow 0  \]
as $t$ goes to zero. Thus there exists a time $t_0 (\lambda)$ such that for $t \leq t_0(\lambda)$, all the $L^2$-operator norms of $T_1(t)$ on $L^2(M, g_t)$ are less than $1$. It follows that $1 + T_1$ is invertible for $t \leq t_0 (\lambda)$, and the inverse is given by the Neumann series 
\[ \lp 1 + T_1  \rp^{-1}=  1-T_1+ T_1^2 -T_1^3... \in \mathcal B \lp  L^2(X,g_t)   \rp.\] 
Furthermore, one can easily check that the series $-T_1+ T_1^2 -T_1^3...$ converges in all the $\mathcal C^k_{t,x_1,x_2}$ norms on $[0,\infty) \times X \times X$, thus $V:=-T_1+ T_1^2 -T_1^3...$ is well-defined and it is a residual operator, i.e.,
\[ V \in \Psi^{-\infty,-\infty,-\infty}_{\cp} (X).\] 
Therefore we proved that for small time $t$ (less than $t_0 (\lambda)$), the inverse $\lp 1+ T_1(\lambda)\rp^{-1}$ exists, and, very importantly, it belongs to the cusp-surgery calculus since we wrote it as
\[ \lp 1+ T_1(\lambda)\rp^{-1} = 1+V   \in 1+ \Psi^{-\infty,-\infty,-\infty}_{\cp}(X)  .\]
By \eqref{constrrez1}, we have that
\begin{align*}
(\Dd - \lambda) Q(\lambda) (1+V) = \lp 1 + T_1 \rp (1+V) =1.
\end{align*}
Since both $Q(\lambda)$ and $(1+V)$ are cusp-surgery operators, by the Composition Theorem \ref{thcomp}, their composition
\[ \Rr(\lambda):= Q(\lambda) (1+V) \]
also belongs to the calculus, more precisely, $\Rr(\lambda) \in \Psi^{-1,-1,0}_{\cp}(X)$. Therefore we constructed a right-inverse for $(\Dd-\lambda)$, and in a similar manner, we also get a left-inverse. Then, by a standard argument, the two inverses coincide, and for small time $t \leq t_0 (\lambda)$, we obtain \emph{the resolvent} of $\Dd- \lambda$ as a cusp-surgery operator
\[  \Rr(\lambda) = \lp \Dd-\lambda \rp^{-1} \in \Psi^{-1,-1,0}_{\cp}(X).    \qedhere  \] 
\end{proof}

The importance of Theorem \ref{resolvintro} is that we managed to merge the family of resolvents $(\Dd_t-\lambda)^{-1}$ for $0< t \leq t_0(\lambda)$ together with the zero time resolvent of $\Dd_0 - \lambda$ on $\lp X \setminus \gamma, g_0 \rp$ into a cusp-surgery operator of orders $(-1,-1,0)$. 
\begin{remark}\label{holom}
Recall that due to Hypothesis \ref{hypot}, the spectrum of $\Dd_0$ is discrete (see e.g. \cite{moroweyl}). Notice that the family of resolvents 
\begin{align*}
\Rr: \mathbb C \longrightarrow \Psi^{-1,-1,0}_{\cp} (X), && \lambda \longmapsto \Rr(\lambda),
\end{align*}
is holomorphic on the complement of $\Spec \Dd_0$, since there it is the inverse of the invertible holomorphic family $\lp \Dd- \lambda \rp$. Furthermore, given an $\epsilon>0$, if $t_0 (\lambda_0)$ is  the time provided by Theorem \ref{resolvintro} for $\lambda_0 \notin \Spec \Dd_0$, notice that for $\lambda \notin \Spec \Dd_0$ in a small neighborhood of $\lambda_0$ (depending on $\epsilon$), the time $t_{0}(\lambda)$ is not less than $t_0(\lambda_0) - \epsilon$.
 \end{remark}

\section{The convergence of the spectral projectors}
\begin{proof}[Proof of Theorem \ref{convprspectrintro1}]
Under Hypothesis \ref{hypot}, the spectrum of $\Dd_0$ on $X \setminus \gamma$ is discrete (see \cite[Theorem 2]{moroweyl}). Therefore we can consider a circle $\mathcal C (\lambda_0, \epsilon)$ centered at $\lambda_0$ of fixed radius $\epsilon $ such that $\lambda_0$ is the only eigenvalue of $\Dd_0$ in the interior of $\mathcal C$. Since the contour $\mathcal C (\lambda_0, \epsilon)$ is compact, by Remark \ref{holom}, there exists a common time $t_1(\lambda_0) \in (0, \infty)$ such that for any $\lambda \in \mathcal C (\lambda_0, \epsilon)$, the resolvent $\Rr(\lambda)$ is well-defined up to time $t_1$. 

As seen in Remark \ref{holom}, the resolvent family $\Rr (\lambda)$ is holomorphic on $\mathbb C \setminus \Spec \Dd_0$, hence for all $0 \leq t \leq t_1$, the spectral projector of $\Dd_t$ onto the interval $[\lambda_0- \epsilon,\lambda_0 + \epsilon]$ is given by
\begin{equation}\label{projector}
\PP_{ \left[ \lambda_0- \epsilon,\lambda_0 + \epsilon \right] }= \frac{1}{2 \pi i} \int_{\mathcal C (\lambda_0, \epsilon)} R(\lambda) d\lambda \in \Psi^{-1,-1,0}_{\cp} (X).
\end{equation}
Furthermore, notice that 
\begin{equation}
(\sigma_{\cp})_{-1} R(\lambda)= \left[ (\sigma_{\cp})_{1} \lp \Dd-\lambda \rp \right]^{-1} = \left[ (\sigma_{\cp})_{1}  \Dd \right]^{-1} 
\end{equation}
does not depend on $\lambda$. Using \eqref{projector}, the cusp-surgery principal symbol of the spectral projector $\PP_{ \left[ \lambda_0- \epsilon,\lambda_0 + \epsilon \right] }$ is given by
\[ \sigma_{\cp} \lp \PP_{ \left[ \lambda_0- \epsilon,\lambda_0 + \epsilon \right] } \rp =  \int_{\mathcal C (\lambda_0, \epsilon) } \sigma_{\cp}  R(\lambda) d\lambda =0,  \]
because the integrand is holomorphic in the disk $D(\lambda_0, \epsilon)$. By Proposition \ref{sirscurtexact} $i)$, it follows that actually $\PP_{ \left[ \lambda_0- \epsilon,\lambda_0 + \epsilon \right] } \in \Psi^{-2,-1,0}_{\cp}(X)$. With a similar argument, we notice that
\[ \NN  \PP_{ \left[ \lambda_0- \epsilon,\lambda_0 + \epsilon \right] } =  \int_{\mathcal C (\lambda_0, \epsilon)} \NN  R(\lambda)  d\lambda =0,  \]
and using Proposition \ref{sirscurtexact} $ii)$, we obtain that the spectral projector belongs to $\Psi^{-2,-2,0}_{\cp}(X)$.
We claim that actually $\PP_{ \left[ \lambda_0- \epsilon,\lambda_0 + \epsilon \right] }$ is smoothing and it vanishes rapidly towards the cusp front face, i.e.,
\[ \PP_{[\lambda_0 -\epsilon,\lambda_0 +\epsilon]} \in \Psi^{-\infty,-\infty,0}_{\cp}(X).\]
Indeed, let $\lambda \in \mathbb C$. Since both the normal operator $\NN$ and the cusp-surgery symbol $\sigma_{\cp}$ of $\Dd - \lambda \in \Psi^{1,1,0}_{\cp}(X)$ are invertible, by a similar argument as in Proposition \ref{parametrix}, there exists $Q (\lambda) \in \Psi^{-1,-1,0}_{\cp} (X)$ an inverse modulo $\Psi^{-\infty,-\infty,0}_{\cp}(X)$
\begin{align*}
\lp \Dd - \lambda \rp Q(\lambda) = 1 +T, && T \in \Psi^{-\infty,-\infty,0}_{\cp}(X),
\end{align*}
and furthermore, $Q(\lambda)$ is holomorphic on $\mathbb C$.

Notice that if $\lambda \notin \Spec \Dd_0$, then $Q(\lambda) \equiv \Rr (\lambda) \lp \modd \Psi^{- \infty, - \infty, 0}_{\cp}(X) \rp$, and in particular, this relation holds true for $\lambda \in \mathcal C (\lambda_0, \epsilon)$. We regard equation \eqref{projector} modulo the space of operators $\Psi^{-\infty,-\infty,0}_{\cp}(X)$, and we get
\begin{equation}
\PP_{ \left[ \lambda_0- \epsilon,\lambda_0 + \epsilon \right] } \equiv \frac{1}{2 \pi i} \int_{\mathcal C (\lambda_0, \epsilon)} Q(\lambda) d\lambda  \  \lp \modd \Psi^{- \infty, - \infty, 0}_{\cp}(X) \rp .
\end{equation}
Since the integrand is holomorphic on $\mathbb C$, we obtain that
\[ \PP_{ \left[ \lambda_0- \epsilon,\lambda_0 + \epsilon \right] } \lp \modd \Psi^{- \infty, - \infty, 0}_{\cp}(X) \rp =0, \]
which yields the conclusion
\[  \PP_{ \left[ \lambda_0- \epsilon,\lambda_0 + \epsilon \right] }  \in \Psi^{-\infty, - \infty, 0}_{\cp}(X).   \]
Furthermore, by Proposition \ref{desumflare}, the Schwartz kernel of the spectral projector $\PP_{ \left[ \lambda_0- \epsilon,\lambda_0 + \epsilon \right]}$ is smooth on $[0, \infty) \times X \times X$ and it vanishes rapidly at $\{ t=0 \} \times \gamma$.
\end{proof}

\begin{corollary}
Under Hypothesis \ref{hypot}, if $\lambda_0$ is an eigenvalue for $\Dd_0$ of multiplicity $1$, the corresponding eigenfunctions of the Dirac operator $\Dd_t$ on the spin Riemannian surface $(X, g_t)$ are smooth up to $\{ t=0 \}$ and they vanish rapidly towards $\{ t=0 \} \times \gamma \times \gamma$.
\end{corollary}

Pfäffle \cite{pfaffle} studied the continuity of the spectrum of the Dirac operator in \emph{hyperbolic} degenerations of manifolds of dimensions $2$ and $3$, under the hypothesis that the spectrum of the Dirac operator on the limit manifold is discrete. He proved that given $\epsilon, \Lambda >0$, there exists a time $t_0$ from which the eigenvalues inside the compact interval $I:=[- \Lambda, \Lambda]$ of all the Dirac operators $(\Dd_t)_{0 \leq t \leq t_0}$ converge to the limit eigenvalues of $\Dd_0$ from $I$. The following result generalizes the $2$-dimensional case of \cite[Theorem~1.2]{pfaffle} for metrics which are hyperbolic \emph{only} near the pinched curve.

\begin{proof}[Proof of Corollary \ref{convspectruluii}]
By Theorem \ref{convprspectrintro1}, it follows that the map $t \longmapsto \Tr P^t_{[a,b]}$ is smooth up to $\{ t=0 \}$, where we denoted by $ P^t_{[a,b]}$ the spectral projector of the Dirac operator on the interval $[a,b]$ at time $t$. Remark that
\[ \Tr  P^t_{[a,b]} = \rank  P^t_{[a,b]} : [0, \infty) \longrightarrow \mathbb N,  \]
thus $ \rank  P^t_{[a,b]}$ is actually a constant function for $t \in [0, t_1']$, where $t_1'$ is the minimum of the $t_1$'s provided by Theorem \ref{convprspectrintro1} for each eigenvalue of the limit operator in the interval $(a,b)$, and the first conclusion follows. 

Let $\lambda_j(0) \in \Spec \Dd_0 \cap (a,b)$ be a limit eigenvalue of multiplicity $m$, and let $\epsilon>0$ such that $[\lambda_j(0) - \epsilon , \lambda_j(0) + \epsilon] \cap \Spec \Dd_0 = \{ \lambda_j(0) \}$. With a similar argument as above, there exists a time $t_2(\epsilon)$ from which there are exactly $m$ eigenvalues of $\Dd_t$, $t \leq t_2(\epsilon)$, in the interval $[\lambda_j(0) - \epsilon , \lambda_j(0) + \epsilon]$. The number of limit eigenvalues in the interval $(a,b)$ is finite, hence the second part of the statement follows as well.
\end{proof}

\section{Trace-class operators in the cusp-surgery calculus}
It is well known that a classical pseudodifferential operator $A$ of order $m$ on a compact manifold $X$ is trace-class if and only if the order $m < - \dim X$ (see e.g. \cite[Section 2.6]{berline}). In this case, the trace is obtained by integrating the point-wise trace of the Schwartz kernel $k_A$ over the diagonal inside $X^2$, written equivalently as
\[ \Tr A = \int_X \Tr k_A(x,x). \]

Furthermore, following \cite[Section 4]{morolauter}, the \emph{cusp trace} of a cusp pseudodifferential operator $A \in \Psi^{m,\alpha}_{\cc}(X)$ satisfying $m<- \dim X$, $\alpha<-1$, is well-defined as
\[ {}^{\cc}{\!} \Tr A  = \int_{\Delta_{\cc}} k_A \in \mathbb C, \]
where $\Delta_{\cc} \cong X$ is the cusp diagonal inside the cusp double space.

Now let us consider an operator $A \in \Psi^{m,\alpha,\beta}_{\cp}(X)$ in the $n$-dimensional context of Section \ref{calculndim}. In order to obtain a well-defined trace functional, we ask the Schwartz kernel of $A$ to be continuous towards the temporal boundary $\tb$, and we impose the necessary conditions for the temporal operator $\Te (A) \in \Psi^{m,\alpha}_{\cc}(X)$ to be trace-class. Therefore the three conditions imposed on the orders of $A$ are  
\begin{equation}\label{conditiiordine}
\begin{aligned}
m< - \dim X, && \alpha < -1, && \beta \leq 0.
\end{aligned}
\end{equation}
We say that a cusp-surgery operator satisfying these conditions is \emph{trace-class}. For $t \geq 0$, let us denote by $\Delta_t$ the intersection of $\{ t \} \times X \times X \subset \Xd$ with the conormality ``plane" $\Delta$ (see Fig. \ref{defop}). For $t>0$, $\Delta_t$ is actually the diagonal inside $X^2$, and the slice $\{ t \} \times X \times X \subset \Xd$ is transversal to the conormality ``plane" $\Delta$. Thus we can restrict a cusp-surgery operator $A \in \Psi^{*,*,*}_{\cp}(X)$ to $\Delta_t$, and the result is a classical pseudodifferential operator acting on $(X,g_t)$.

\begin{definition}
The \emph{cusp-surgery trace} of a trace-class operator $A \in \Psi^{m, \alpha,\beta}_{\cp}(X)$ satisfying \eqref{conditiiordine} is a function ${}^{\cp}{\!} \Tr (A) : [0, \infty) \longrightarrow \mathbb C$ defined as
$ {}^{\cp}{\!} \Tr (A) (t) =\int_{\Delta_t} \Tr k_A. $
\end{definition}

Notice that for fixed time $t >0$, using Lidskii's Theorem (see for instance \cite{shubin}), the cusp-surgery trace maps the Schwartz kernel
\[  {k_A}_{ \vert_{\{ t \} \times \Diag}}   \longmapsto  \int_{(X,g_t)} \Tr  {k_A}_{\vert_{\Diag}} = \Tr_{L^2 (X, g_t)} \lp A_t \rp , \]
where $\Diag$ is the diagonal inside $X^2$, and $A_t$ is the restriction of the operator $A$ to $(X, h_t )$.

\begin{figure}[H]
\begin{center}
\includegraphics[width=11cm, height=4.7cm]{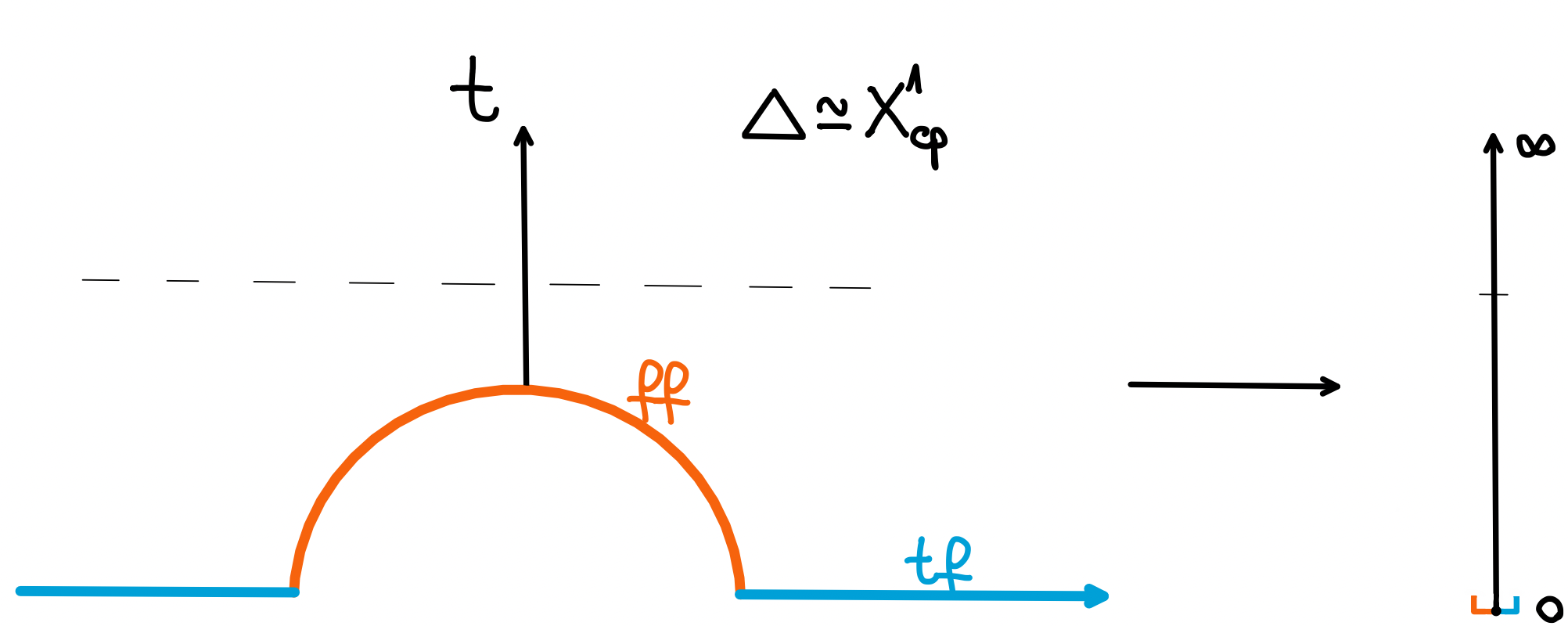}
\caption{The cusp-surgery trace as a push-forward}\label{figtrace}
\end{center}
\end{figure}
\begin{proposition}\label{traceclassndim}
Let $A \in \Psi^{m, \alpha, \beta}_{\cp}(X)$ be a trace-class cusp-surgery pseudodifferential operator, i.e., suppose that
\begin{align*}
 m<  - \dim X, && \alpha < -1, && \beta \leq 0. 
 \end{align*}
\begin{itemize}
\item[$i)$] If $\alpha - \beta \notin \mathbb Z$, then 
\[ {}^{\cp}{\!} \Tr A \in t^{- \alpha} \mathcal C^{\infty}  [0, \infty) + t^{- \beta} \mathcal C^{\infty}  [0, \infty) . \]
\item[$ii)$] If $\alpha - \beta \in \mathbb Z$, then 
\[  {}^{\cp}{\!} \Tr A \in t^{ \min (-\alpha, - \beta)} \mathcal C^{\infty}  [0, \infty) + t^{ \max(-\alpha,-\beta) } \log t \cdot \mathcal C^{\infty}  [0, \infty) . \]
\end{itemize}
\end{proposition}
\begin{proof}
Remark that $\Delta$ is diffeomorphic to the simple space $\Xs$ (see Fig. \ref{defop}). Therefore, as a function of the time variable $t$, the cusp-surgery trace of the operator $A$ is obtained as the push-forward of ${k_A}_{\vert_{\Delta}}$ from the simple space $\Xs$ to $[0, \infty)$ (see Fig. \ref{figtrace}).
One can easily check the \emph{b}-fibration property for the map 
\begin{align*}
\pi : \Xs \longrightarrow [0, \infty), && \pi = \pi_+ \circ \beta,
\end{align*}
where $\beta$ is the blow-down map in \eqref{beta}, and $\pi_+ : [0, \infty) \times X \longrightarrow [0, \infty)$ is the projection onto the first factor. Furthermore, both boundary hypersurfaces $\ff$ and $\tf$ are mapped through the map $\pi$ to $\{  0 \}$, thus the integrability condition in the Push-Forward Theorem \ref{pft} is empty.

Notice that the index set of ${(k_A)}_{\vert_{\Delta}}$ towards the front face $\ff$ is $- \alpha + \mathbb N^*:= \{ (-\alpha + n,0): \ n \geq 1 \}$, where $\alpha < -1$ (see \eqref{conditiiordine}). Furthermore, the index set towards the temporal face $\tf$ is $-\beta + \mathbb N:= \{ (-\beta +n,0): \ n \geq 0 \}$, with $\beta \leq 0$. Therefore by the Push-Forward Theorem \ref{pft}, the index set towards the (only) boundary hypersurface $\{ 0 \} \subset [0, \infty)$ is 
\begin{equation}
\begin{aligned}
 I_{\{ 0 \}} ={}& \lp -\alpha + \mathbb N \rp \overline{\cup} \lp - \beta + \mathbb N \rp = \lp -\alpha + \mathbb N \rp \cup \lp -\beta + \mathbb N \rp \cup \left\{  \lp \alpha+n=\beta+m , 1 \rp : \ n,m \in \mathbb N  \right\}.
\end{aligned}
\end{equation}
Notice that logarithmic terms appear if and only if $\alpha-\beta \in \mathbb Z$, and the conclusion follows.
\end{proof}

In the particular $2$-dimensional case described in Section \ref{introducere}, we obtain Proposition \ref{traceclassintro}.

\section{Cusp-surgery trace of resolvents}
\begin{proof}[Proof of Corollary \ref{puterirez}]
By Theorem \ref{resolvintro}, it follows that $\Rr(\lambda) \in \Psi^{-1,-1,0}_{\cp}(X)$, and using the Composition Theorem \ref{compositiontheorem}, we get
\[ \Rr(\lambda)^k \in \Psi^{-k,-k,0}_{\cp}(X).  \]
Using Proposition \ref{traceclassintro} $ii)$ for $\alpha=-k$, $\beta=0$ (hence $\alpha-\beta \in \mathbb Z$), we get
\[ {}^{\cp} {\!} \Tr \lp \Rr(\lambda)^k \rp \in \mathcal C^{\infty}[0, \infty) + t^{k} \log t \cdot \mathcal C^{\infty}[0, \infty),  \]
and the conclusion follows.
\end{proof}

\begin{proof}[Proof of Corollary \ref{improverares}]
We remark that by the Composition Theorem \ref{compositiontheorem}, we have 
\[ \sigma_{\cp} \widetilde{\Rr} (\lambda) = \left[ \sigma_{\cp} \lp  \Dd^2- \lambda   \rp \right]^{-1} =\sigma_{\cp} \widetilde{\Rr} (\lambda_0), \]
thus $\sigma_{\cp} \lp \widetilde{\Rr} (\lambda)-\widetilde{\Rr} (\lambda_0) \rp =0$ and by Proposition \ref{sirscurtexact}, it follows that the difference satisfies
\[    R(\lambda)-R(\lambda_0) \in \Psi^{-3,-2,0}_{\cp}(X), \]
thus it is trace-class (see \eqref{conditiiordine}). Proposition \ref{traceclassintro} implies the conclusion.
\end{proof}

\end{document}